\pdfoutput=1
\documentclass{amsart}
\usepackage[T1]{fontenc}
\usepackage[utf8]{inputenc}

      \title          [K-theory of group rings  and the cyclotomic trace map]
            {Algebraic K-theory of group rings\\and the cyclotomic trace map}

     \author{Wolfgang L\"uck}
    \address{HIM (Hausdorff Research Institute for Mathematics)
             \&\ Mathematisches Institut, Rheinische Friedrich-Wilhelms-Universit\"at Bonn, Germany}
      \email{\hemail{wolfgang.lueck@him.uni-bonn.de}}
    \urladdr{\hurl{him.uni-bonn.de/lueck/}}

     \author{Holger Reich}
    \address{Institut f\"ur Mathematik, Freie Universit\"at Berlin, Germany}
      \email{\hemail{holger.reich@fu-berlin.de}}
    \urladdr{\hurl{mi.fu-berlin.de/math/groups/top/members/Professoren/reich.html}}

     \author{John Rognes}
    \address{Department of Mathematics, University of Oslo, Norway}
      \email{\hemail{rognes@math.uio.no}}
    \urladdr{\hurl[]{folk.uio.no/rognes/home.html}}

     \author{Marco Varisco}
    \address{Department of Mathematics and Statistics, University at Albany, SUNY, USA}
      \email{\hemail{mvarisco@albany.edu}}
    \urladdr{\hurl{albany.edu/~mv312143/}}

       \date{September 7, 2016}

\usepackage{amssymb}
\usepackage{aliascnt}
\usepackage{mathtools}  \mathtoolsset{showonlyrefs}
\usepackage{enumitem}
\usepackage{etoolbox}
\usepackage{adjustbox}
\usepackage{mfirstuc}
\usepackage{manfnt}
\usepackage{pifont}
\usepackage{mathrsfs}
\usepackage{stmaryrd}
\usepackage{dsfont}
\usepackage{tikz}       \usetikzlibrary{cd}
\usepackage[shortcuts]{extdash}
\usepackage[pagebackref, pdfinfo={
      Title={Algebraic K-theory of group rings and the cyclotomic trace map},
     Author={Lück, Reich, Rognes, Varisco},
    Subject={MSC2010: 19D50, 19D55, 19B28, 55P91, 55P42}
}]{hyperref}
\usepackage[alphabetic, backrefs, msc-links, nobysame]{amsrefs}

\makeatletter
\renewcommand{\tocsection}[3]{%
  \indentlabel{\@ifnotempty{#2}{\makebox[2.25em][r]{\ignorespaces#1 #2.\quad}}}#3}
\makeatother

\makeatletter
\def\@seccntformat#1{%
  \protect\textup{%
    \protect\@secnumfont
    \expandafter\protect\csname format#1\endcsname
    \csname the#1\endcsname
    \protect\@secnumpunct
  }%
}
\makeatother

\renewcommand{\thesubsection}{\thesection\Alph{subsection}}

\newcommand*{\hurl}  [2][www.]{\href{http://#1#2}{\nolinkurl{#2}}}
\newcommand*{\hemail}[1]{\href{mailto:#1}{\nolinkurl{#1}}}
\newcommand*{\DOI}   [1]{\href{http://dx.doi.org/#1}{\nolinkurl{#1}}}
\newcommand*{\arXiv} [1]{\href{http://www.arxiv.org/abs/#1}{\nolinkurl{arXiv:#1}}}
\newcommand*{\MSC}   [1]{\href{http://www.ams.org/msc/msc2010.html?t=#1}{#1}}

\setlist{leftmargin=*}
\setlist[enumerate]{label=(\roman*)}

\numberwithin{equation}{section}

\newcommand*{\definetheorem}[3][equation]{%
  \newaliascnt{#2}{#1}
  \newtheorem{#2}[#2]{#3}
  \aliascntresetthe{#2}
  \expandafter\def\csname #2autorefname\endcsname{#3}
}
\makeatletter
\newcommand*{\definetheoremsame}[2][equation]{%
  \definetheorem[#1]{\zap@space#2 \@empty}{\capitalisewords{#2}}
}
\makeatother

\theoremstyle{plain}
\forcsvlist{\definetheoremsame}%
{addendum, conjecture, corollary, lemma, main technical theorem, main theorem, proposition, theorem}

\theoremstyle{definition}
\forcsvlist{\definetheoremsame}%
{definition, example, fact, facts, remark, warning}

\newcommand*{\one}  {\text{\ding{192}}} \newcommand*{\bone}  {\text{\ding{202}}}
\newcommand*{\two}  {\text{\ding{193}}} \newcommand*{\btwo}  {\text{\ding{203}}}
\newcommand*{\three}{\text{\ding{194}}} \newcommand*{\bthree}{\text{\ding{204}}}
\newcommand*{\four} {\text{\ding{195}}} \newcommand*{\bfour} {\text{\ding{205}}}
\newcommand*{\five} {\text{\ding{196}}} 
 \newcommand*{\bsix}  {\text{\ding{207}}}
\newcommand*{\seven}{\text{\ding{198}}} 
\newcommand*{\eight}{\text{\ding{199}}}

\DeclareMathAlphabet{\matheurm}{U}{eur}{m}{n}

\newcommand*{\define}[5]{%
  \ifstrequal{#2}{*}{\expandafter#1\expandafter*}{\expandafter#1}%
  \csname#4#5\endcsname{#3{#5}}
}
\forcsvlist{\define{\newcommand}{*}{\mathbf}{B}}{%
  A,B,C,D,E,F,G,H,I,J,K,L,M,N,O,P,Q,R,S,T,U,V,W,X,Y,Z,%
}
\forcsvlist{\define{\newcommand}{*}{\mathcal}{C}}{%
  A,B,C,D,E,F,G,H,I,J,K,L,M,N,O,P,Q,R,S,T,U,V,W,X,Y,Z,%
}
\forcsvlist{\define{\newcommand}{*}{\mathscr}{S}}{%
  A,B,C,D,E,F,G,H,I,J,K,L,M,N,O,P,Q,R,  T,U,V,W,X,Y,Z,%
}                                  
\forcsvlist{\define{\newcommand}{*}{\mathbb}{I}}{%
  A,B,C,D,E,F,G,H,I,  K,L,M,N,O,P,Q,R,S,T,U,V,W,X,Y,Z%
}                
\renewcommand*{\SS}{\mathscr{S}}
\renewcommand*{\IJ}{\mathbb{J}}
\newcommand*{\Ione}{\mathds{1}}

\forcsvlist{\define{\DeclareMathOperator}{}{}{}}%
{asbl, coend, coker, conhom, conj, ext, Ext, fun, hocofib, hoeq, hofib, Hom, id, incl, ind, Ind, inn, Lan, map, mor, obj, pr, pt, res, Res, sd, tors, trc}
\DeclareMathOperator{\ii}{in}

\forcsvlist{\define{\DeclareMathOperator}{*}{}{}}%
{colim, hocolim, holim}

\forcsvlist{\define{\newcommand}{*}{\matheurm}{}}%
{Ab, Cat, CAT, Groupoids, Mod, Or, Sets, Sp, Sub, Sym, Mon, Th, Top}
\newcommand*{\FGI}{\matheurm{FinGroupsInj}}
\newcommand*{\RFcat}{{\mathcal{RF}}}

\forcsvlist{\define{\newcommand}{*}{\WT}{}}%
{K, Wh, THH, F, C, TF, TR, TC, T}
\newcommand*{\wh}{\mathit{Wh}}

\newcommand*{\TO}  [1][]{\stackrel{#1}{\longrightarrow}}
\newcommand*{\FROM}[1][]{\stackrel{#1}{\longleftarrow}}
\newcommand*{\MOR} [4][]{#2\colon#3\TO[#1]#4}
\newcommand*{\AND}{\qquad\text{and}\qquad}

\DeclarePairedDelimiterX\SET[2]{\{}{\}}{\,#1\;\delimsize\vert\;#2\,}
\DeclarePairedDelimiter\real{\lvert}{\rvert}
\DeclarePairedDelimiter\reall{\|}{\|}

\newcommand*{\ds}{\displaystyle}
\newcommand*{\ts}{\textstyle}

\newcommand*{\ab}  {{\operatorname{ab}}}
\newcommand*{\cont}{{\operatorname{cont}}}
\newcommand*{\copr}{{\operatorname{copr}}}
\newcommand*{\et}  {{\operatorname{\acute{e}t}}}
\newcommand*{\op}  {{\operatorname{op}}}

\DeclareMathOperator*{\smallcoprod}{\ts\coprod}
\DeclareMathOperator*{\smallprod}  {\ts\prod}
\DeclareMathOperator*{\tensor}     {\otimes}
\DeclareMathOperator*{\timesd}     {\times}
\DeclareMathOperator*{\sma}        {\wedge}

\newcommand*{\D}{\text{-}} 

\newcommand*{\forget}{\IU}           
\newcommand*{\spec}[1]{\mathbb{#1}}  
\newcommand*{\WT}  [1]{\mathbf{#1}}  
\newcommand*{\GS}  [1]{\mathcal{#1}} 
\newcommand*{\Q}   [1]{Q^{#1}\!}     
\newcommand*{\cofr}{\Gamma}          
\newcommand*{\hm}{\varphi}           

\newcommand*{\sh}{\smash{\operatorname{sh}}\vphantom{\Omega}}

\newcommand*{\pcompl}[1][p]{\widehat{{}_#1}}
\newcommand*{\prodp} [1][p]{\;\smallprod_{\mathclap{#1\textup{ prime}}}\;}
\newcommand*{\Cp}    [2][p]{C_{{#1}^{#2}}}

\newcommand*{\modu}[3]{\vphantom{#2}_{#1}{#2}_{#3}}

\newcommand*{\assembly}[1]{\alpha^#1}

\newcommand*{\Cyc} {{\mathcal{C} \hspace{-.2ex}\mathit{yc}}}
\newcommand*{\VCyc}{{\mathcal{VC}\hspace{-.2ex}\mathit{yc}}}
\newcommand*{\FCyc}{{\mathcal{FC}\hspace{-.2ex}\mathit{yc}}}
\newcommand*{\Fin} {{\mathcal{F} \hspace{-.2ex}\mathit{in}}}

\newcommand*{\oid}[2]{#1\!\smallint\!#2} 

\newcommand*{\coffree}[1]{\Sets^{#1}_{\text{\textup{cof,free}}}}

\newcommand*{\assum}[2]{$[#1_#2]$}
\newcommand*{\techA}[2]{$[A'_{#1,\,\le#2}]$}
\newcommand*{\techB}[2]{$[ B_{#1,\,   #2}]$}
\newcommand*{\TT}{T}

\newcommand*{\BHM}{B\"okstedt-Hsiang-Madsen}
\newcommand*{\piiso}{$\underline{\pi}_*$\=/isomorphism}

\newcommand*{\cplus}[1][]{\ifstrequal{#1}{*}{C}{c}onnective\textsuperscript{+}}

\newcommand*{\RMfiller}[1]{\phantom{\ds\Bigl(\mspace{#1mu}}}
\newcommand*{\RMheader}[2]{\underline{\RMfiller{#1}\text{#2}\RMfiller{#1}}}
\newcommand*{\RMref}   [1]{\text{\makebox[4.2em][s]{#1}}\qquad}

\begin{document}

\begin{abstract}
We prove that the Farrell-Jones assembly map for connective algebraic $K$-theory is rationally injective, under mild homological finiteness conditions on the group and assuming that a weak version of the Leopoldt-Schneider conjecture holds for cyclotomic fields.
This generalizes a result of B\"okstedt, Hsiang, and Madsen, and leads to a concrete description of a large direct summand of $K_n(\IZ[G])\tensor_\IZ\IQ$ in terms of group homology.
In many cases the number theoretic conjectures are true, so we obtain rational injectivity results about assembly maps, in particular for Whitehead groups, under homological finiteness assumptions on the group only.
The proof uses the cyclotomic trace map to topological cyclic homology, \BHM's functor~$C$, and new general isomorphism and injectivity results about the assembly maps for topological Hochschild homology and~$C$.
\end{abstract}


\maketitle
\tableofcontents
\thispagestyle{empty}


\section{Introduction}
\label{INTRO}

The algebraic $K$-theory groups~$K_n(\IZ[G])$ of the integral group ring of a discrete group~$G$ have far-reaching geometric applications to the study of manifolds with fundamental group~$G$ and of sufficiently high dimension.
One of the most famous and important manifestations of this phenomenon is given by the Whitehead group~$\wh(G)$, which is the quotient of~$K_1(\IZ[G])=GL(\IZ[G])_\ab$ by the subgroup generated by the units in~$\IZ[G]$ of the form~$\pm g$ with $g\in G$.
By the celebrated $s$-Cobordism Theorem, $\wh(G)$ completely classifies the isomorphism classes of $h$-cobordisms over any closed, connected manifold~$M$ with $\dim(M)\ge5$ and $\pi_1(M)\cong G$.

A well-known and still open conjecture, which we review below, predicts that the Whitehead group of any torsion-free group vanishes.
However, if a group~$G$ has torsion, then in general~$\wh(G)$ is not trivial.
For example, the structure of the Whitehead groups of finite groups~$G$ is well understood (see~\cite{Oliver}), and if~$C$ is any finite cyclic group of order $c\not\in\{1,2,3,4,6\}$, then $\wh(C)$ is not zero, not even after tensoring with the rational numbers~$\IQ$.

One of the main consequences of this work is the following theorem about Whitehead groups of infinite groups with torsion.
Its conclusion says that rationally the Whitehead groups~$\wh(H)$ of all finite subgroups~$H$ of~$G$ contribute to~$\wh(G)$, and only the obvious relations between these contributions hold.
Its assumption is a very weak and natural homological finiteness condition only up to dimension two.
As we explain in \autoref{ASSUMPTIONS}, it is satisfied by many geometrically interesting groups, such as outer automorphism groups of free groups and Thompson's group~$T$, thus yielding the first known results about the Whitehead groups of these groups.
But, in contrast to other known results in this area, the assumption is purely homological, and no geometric input is required.

\begin{theorem}
\label{Whitehead}
Assume that, for every finite cyclic subgroup~$C$ of a group~$G$, the first and second integral group homology $H_1(BZ_GC;\IZ)$ and $H_2(BZ_GC;\IZ)$ of the centralizer~$Z_G C$ of~$C$ in~$G$ are finitely generated abelian groups.
Then the map
\begin{equation}
\label{eq:whitehead}
\colim_{H\in\obj\Sub G(\Fin)}\wh(H)\tensor_\IZ\IQ
\TO
\wh(G)\tensor_\IZ\IQ
\end{equation}
is injective.
\end{theorem}

Here the colimit is taken over the finite subgroup category~$\Sub G(\Fin)$, whose objects are the finite subgroups $H$ of~$G$ and whose morphisms are induced by subconjugation in~$G$.
In the special case when $G$ is abelian, then $\Sub G(\Fin)$ is just the poset of finite subgroups of~$G$ ordered by inclusion.
The general definition is reviewed at the beginning of \autoref{WHITEHEAD}.

The injectivity of the map~\eqref{eq:whitehead}, as well as the vanishing of the Whitehead groups of torsion-free groups, are in fact consequences of much more general conjectures about assembly maps.
Similarly, \autoref{Whitehead} is a consequence of more general results, namely \autoref{main} and \autoref{main-technical} below.
In the rest of this section we introduce assembly maps, recall these conjectures, and state our main results.

We denote by~$\K(R)$ the non-connective algebraic $K$-theory spectrum of a ring~$R$, with homotopy groups $\pi_n (\K( R ) ) = K_n ( R )$ for all~$n \in \IZ$,
and we let~$BG$ be the classifying space of the group~$G$.
The classical assembly map is a map of spectra
\begin{equation}
\label{eq:classical-assembly}
BG_+ \sma \K (\IZ ) \TO \K ( \IZ[G] )
\,,
\end{equation}
which was first introduced by Loday~\cite{Loday}*{Proposition~4.1.1 on page 356}.
This map is conjectured to be a $\pi_*$-isomorphism if the group~$G$ has no torsion, as we explain below.
After taking homotopy groups and rationalizing, the left-hand side of~\eqref{eq:classical-assembly} can be explicitly computed:
for every $n \in \IZ$ there is an isomorphism
\begin{equation}
\label{eq:Dold}
\bigoplus_{\substack{s+t=n\\s,t\ge0}}
H_s ( BG ; \IQ ) \tensor_{\IQ} \Bigl(  K_t ( \IZ ) \tensor_{\IZ} \IQ \Bigr) \TO[\cong] \pi_n \bigl( BG_+ \sma \K ( \IZ ) \bigr) \tensor_{\IZ} \IQ
\,.
\end{equation}

We call the map
\begin{equation}
\label{eq:Q-classical-assembly}
\bigoplus_{\substack{s+t=n\\s,t\ge0}}
H_s ( BG ; \IQ ) \tensor_{\IQ} \Bigl(  K_t ( \IZ ) \tensor_{\IZ} \IQ \Bigr) \TO K_n ( \IZ G ) \tensor_{\IZ} \IQ
\end{equation}
induced by the composition of \eqref{eq:Dold} and~\eqref{eq:classical-assembly} the \emph{rationalized classical assembly map}.
Hence we obtain the following conjecture, which was formulated by Hsiang in his plenary address at the 1983 ICM under the additional assumption that $G$ has a finite classifying space~$BG$; see~\cite{Hsiang}*{Conjecture~4 on page~114}.

\begin{conjecture}[The Hsiang Conjecture]
\label{Hsiang-Conj}
For every torsion-free group $G$ and every $n \in \IZ$ the rationalized classical assembly map~\eqref{eq:Q-classical-assembly} is an isomorphism.
\end{conjecture}

There are analogous versions of the classical assembly map~\eqref{eq:classical-assembly} and \autoref{Hsiang-Conj} for $L$-theory instead of algebraic $K$-theory.
The rational injectivity of the $L$-theory version of~\eqref{eq:classical-assembly} is equivalent to the famous Novikov Conjecture about the homotopy invariance of higher signatures.

In~\cite{BHM} B\"okstedt, Hsiang, and Madsen invented topological cyclic homology and the cyclotomic trace map from algebraic $K$-theory to topological cyclic homology, and used them to prove the following theorem, which is often referred to as the algebraic $K$-theory Novikov Conjecture; see~\cite{BHM}*{Theorem~9.13 on page~535} and~\cite{Madsen}*{Theorem~4.5.4 on page~284}.

\begin{theorem}[\BHM]
\label{BHM}
Let $G$ be a not necessarily torsion-free group.
Assume that the following condition holds.
\begin{enumerate}[label=\({[}\Alph*_{1}{]}\)]
\item
\label{BHM-A}
For every~$s\ge1$ the integral group homology $H_s(BG;\IZ)$ is a finitely generated abelian group.
\end{enumerate}
Then the rationalized classical assembly map~\eqref{eq:Q-classical-assembly} is injective for all $n\geq0$.
\end{theorem}

For groups with torsion the classical assembly map is not surjective in general.
For example, the rationalized Whitehead group~$\wh(G)\tensor_\IZ\IQ$ is the cokernel of~$\eqref{eq:Q-classical-assembly}$ for~$n=1$, and it is usually not zero when $G$ has torsion, as we recalled before \autoref{Whitehead}.
Notice that, in particular, \autoref{BHM} does not give any information about Whitehead groups.
In contrast, as a consequence of our main result, we detect in \autoref{Whitehead} a large direct summand of~$\wh(G)\tensor_\IZ\IQ$ under extremely mild homological finiteness assumptions on~$G$.

The Farrell-Jones Conjecture~\cite{FJ-iso} predicts the structure of $K_n ( \IZ G )$ for arbitrary groups $G$, and reduces to the above-mentioned conjecture about the map \eqref{eq:classical-assembly} in the case of torsion-free groups.
In the equivalent reformulation due to J.~Davis and the first author~\cite{Davis-Lueck}, the Farrell-Jones Conjecture says that the Farrell-Jones assembly map
\begin{equation}
\label{eq:FJ-assembly}
EG(\VCyc)_+  \sma_{\Or G} \K(\IZ[\oid{G}{-}]) \TO \K (\IZ[G])
\,,
\end{equation}
a certain generalization of the classical assembly map \eqref{eq:classical-assembly}, is a $\pi_*$-isomorphism.
Here $EG(\VCyc)$ is the universal space for the family of virtually cyclic subgroups of~$G$ reviewed at the beginning of \autoref{ASSUMPTIONS}, $\Or G$ is the orbit category, $\oid{G}{-}$ is the action groupoid functor from \autoref{ACTION-GROUPOID}, and the map~\eqref{eq:FJ-assembly} is defined in \autoref{ASSEMBLY}.
We refer to the end of this introduction for some additional remarks about the Farrell-Jones Conjecture.
Now we concentrate on computational consequences for the rationalized homotopy groups and describe the analogue of the isomorphism~\eqref{eq:Dold}.

In order to give a precise statement we need more notation.
Given a group $G$ and a subgroup $H$, we denote by $N_G H$ the normalizer and by $Z_G H$ the centralizer of $H$ in $G$, and we define the Weyl group as the quotient $W_G H = N_G H / (Z_G H \cdot H)$.
Notice that the Weyl group $W_G H$ of a finite subgroup $H$ is always finite, since it embeds into the outer automorphism group of~$H$.
We write $\FCyc$ for the family of finite cyclic subgroups of $G$, and $( \FCyc )$ for the set of conjugacy classes of finite cyclic subgroups.

Using results of the first author in~\cite{L-Chern}, Grunewald established in~\cite{Grunewald}*{Corollary on page~165} an isomorphism
\begin{equation}
\label{eq:Lueck-Grunewald}
\begin{tikzcd}[row sep=small]
\ds
\bigoplus_{(C)\in(\FCyc)}
\bigoplus_{\substack{s+t=n\\s\ge0,t\ge-1}}
H_s(BZ_GC;\IQ)\tensor_{\IQ[W_G C]}\Theta_C\Bigl(K_t(\IZ[C])\tensor_{\IZ}\IQ\Bigr)
\arrow[d, "\cong", pos=.1, shorten <=-.75em]
\\
\ds
\pi_n\Bigl(EG(\VCyc)_+\sma_{\Or G}\K(\IZ[\oid{G}{-}])\Bigr)\tensor_\IZ\IQ
\end{tikzcd}
\end{equation}
for every $n\in\IZ$; see also~\cite{LS}*{Theorem~0.3 on page~2}.
Here $\Theta_C$ is an idempotent endomorphism of $K_t(\IZ[C])\tensor_{\IZ}\IQ$, which corresponds to a specific idempotent in the rationalized Burnside ring of~$C$, and whose image is a direct summand of $K_t(\IZ[C])\tensor_{\IZ}\IQ$ isomorphic to
\begin{equation}
\label{eq:Artin-defect}
\coker\Biggl(\, \bigoplus_{D\lneqq C} \ind_D^C\colon \bigoplus_{D\lneqq C}
     K_t(\IZ[D])\tensor_{\IZ}\IQ \TO K_t(\IZ[C])\tensor_{\IZ}\IQ \Biggr).
\end{equation}
See \autoref{CHERN} for more details.
The Weyl group acts via conjugation on $C$ and hence on $\Theta_C ( K_t(\IZ[C])\tensor_{\IZ}\IQ)$.
The Weyl group action on the homology comes from the fact that $EN_G C / Z_G C$ is a model for $BZ_G C$.
We remark that the dimensions of the $\IQ$-vector spaces in~\eqref{eq:Artin-defect} are explicitly computed in~\cite{Patronas}*{Theorem on page~9} for any~$t$ and any finite cyclic group~$C$.
Notice that if $t<-1$ then $K_t(\IZ[C])$ vanishes for any virtually cyclic group~$C$ by~\cite{FJ-vcyc}.

We call the map
\begin{equation}
\label{eq:Q-FJ-assembly}
\bigoplus_{(C)\in(\FCyc)}
\bigoplus_{\substack{s+t=n\\s\ge0,t\ge-1}}
H_s(BZ_GC;\IQ)\tensor_{\IQ[W_G C]}\Theta_C\Bigl(K_t(\IZ[C])\tensor_{\IZ}\IQ\Bigr)
\TO
K_n(\IZ[G])\tensor_\IZ\IQ
\end{equation}
induced by the composition of \eqref{eq:Lueck-Grunewald} and~\eqref{eq:FJ-assembly} the \emph{rationalized Farrell-Jones assembly map}.
Hence we obtain the following conjecture, which is a generalization of the Hsiang \autoref{Hsiang-Conj} for groups that are not necessarily torsion-free.

\begin{conjecture}[The Rationalized Farrell-Jones Conjecture]
\label{Q-FJ-Conj}
For every group $G$ and every $n \in \IZ$ the rationalized Farrell-Jones assembly map~\eqref{eq:Q-FJ-assembly} is an isomorphism.
\end{conjecture}

The restriction of the map \eqref{eq:Q-FJ-assembly} to the summand indexed by the trivial subgroup of $G$ is the map \eqref{eq:Q-classical-assembly}.
So for a torsion-free group the Rationalized Farrell-Jones \autoref{Q-FJ-Conj} specializes to the Hsiang \autoref{Hsiang-Conj}.

Our main result generalizes \BHM's \autoref{BHM} to the Rationalized Farrell-Jones \autoref{Q-FJ-Conj}.

\begin{maintheorem}
\label{main}
Let $G$ be a group.
Assume that the following two conditions hold.
\begin{enumerate}[label=\assum{\Alph*}{\FCyc}]
\item
\label{our-A}
For every finite cyclic subgroup~$C$ of~$G$ and for every $s\geq1$, the integral group homology $H_s(BZ_GC;\IZ)$ of the centralizer of~$C$ in~$G$ is a finitely generated abelian group.
\item
\label{our-B}
For every finite cyclic subgroup~$C$ of~$G$ and for every $t\geq0$, the natural homomorphism
\[
K_t(\IZ[\zeta_c])
\TO
\;\smallprod_{\mathclap{p\textup{ prime}}}\;
K_t\Bigl(\IZ_p\tensor_\IZ\IZ[\zeta_c];\IZ_p\Bigr)
\]
is $\IQ$-injective, i.e., injective after applying~$-\tensor_\IZ\IQ$.
Here $c$ is the order of~$C$ and $\zeta_c$ is any primitive $c$-th root of unity.
\end{enumerate}
Then the restriction
\begin{equation}
\label{eq:Q-FJ-assembly-conn}
\bigoplus_{(C)\in(\FCyc)}
\,
\bigoplus_{\substack{s+t=n\\s,t\ge0}}
H_s(BZ_GC;\IQ)\tensor_{\IQ[W_G C]}\Theta_C\Bigl(K_t(\IZ[C])\tensor_{\IZ}\IQ\Bigr)
\TO
K_n(\IZ[G])\tensor_\IZ\IQ
\end{equation}
of the rationalized Farrell-Jones assembly map~\eqref{eq:Q-FJ-assembly} to the summands satisfying $t \geq 0$ is injective for all $n \in \IZ$.
\end{maintheorem}

In Assumption~\ref{our-B} we write $\IZ_p$ for the ring of $p$-adic integers and we use the notation $K_t(R;\IZ_p)$ for
$\pi_t(\BK(R)\pcompl)$,
the $t$-th homotopy group of the $p$-completed algebraic $K$-theory spectrum
of the ring $R$;
see \autoref{P-COMPL}.
The map to the factor indexed by~$p$ is induced by the natural map $K_t (R ) \TO K_t( R ; \IZ_p )$ and the ring homomorphism $\IZ [ \zeta_c ] \TO \IZ_p \tensor_{\IZ} \IZ [ \zeta_c ]$.

Assumption~\ref{our-A} implies and is the obvious generalization of Assumption~\ref{BHM-A}.

In the next section we thoroughly discuss Assumptions \ref{our-A} and~\ref{our-B}.
The discussion can be summarized as follows.
\begin{enumerate}[label=\assum{\Alph*}{\FCyc}]
\item
is satisfied by a very large class of groups.
\item
is conjecturally always true:
it is automatically satisfied if a weak version of the Leopoldt-Schneider Conjecture holds for cyclotomic fields.
\end{enumerate}
Moreover, Assumption~\ref{our-B} is satisfied for all~$c$ when $t=0$ or~$1$, and it is satisfied for all~$t$ if $c=1$;
for any fixed~$c$, it is satisfied for almost all~$t$.
Because of this, only finiteness assumptions appear in Theorems \ref{Whitehead} and~\ref{BHM} and in the following result.

\begin{theorem}[Eventual injectivity]
\label{above-L}
Assume that there is a finite universal space for proper actions $EG(\Fin)$.
Then there exists an integer~$L>0$ such that the rationalized Farrell-Jones assembly map~\eqref{eq:Q-FJ-assembly} is injective for all $n\ge L$.
The bound~$L$ only depends on the dimension of~$EG(\Fin)$ and on the orders of the finite cyclic subgroups of~$G$.
\end{theorem}

Universal spaces are reviewed at the beginning of the next section, where we also give many interesting examples of groups having finite~$EG(\Fin)$.
Here we only remark that, for example, the assumption of \autoref{above-L} is satisfied by outer automorphism groups of free groups, for which there seems to be no other known result about the Farrell-Jones Conjecture.
\autoref{above-L} is proved at the end of \autoref{SCHNEIDER}.

We now formulate a stronger but more technical version of \autoref{main}, which in particular is needed to deduce \autoref{Whitehead}.
We first need some algebraic definitions from~\cite{LRV}.
An abelian group~$M$ is called almost trivial if there exists an $r \in \IZ-\{0\}$ such that $rm = 0$ for all $m \in M$.
We say that $M$ is \emph{almost finitely generated} if its torsion part $\tors M$ is almost trivial and if $M/ \tors M$ is a finitely generated abelian group.
Every finitely generated abelian group~$M$ is in particular almost finitely generated.
However, notice that $M \tensor_{\IZ} \IQ$ being finitely generated as a rational vector space does not necessarily imply that $M$ is almost finitely generated.
Almost finitely generated abelian groups form a Serre class.

\begin{maintechnicaltheorem}
\label{main-technical}
Let $G$ be a group and let $\CF\subseteq\FCyc$ be a family of finite cyclic subgroups of~$G$.
Let $N\geq0$ be an integer and let $\TT$ be a subset of the non-negative integers $\IN=\{0,1,2,\dotsc\}$.
Assume that the following two conditions hold.
\begin{enumerate}[label=\({[}\Alph*'_{\CF,\,\le N+1}{]}\)]
\item
\label{techA}
For every~$C\in\CF$ and for every $1\leq s\leq N+1$, the integral group homology $H_s(BZ_GC;\IZ)$ of the centralizer of~$C$ in~$G$ is an almost finitely generated abelian group.
\item[\techB{\CF}{\TT}]
For every~$C\in\CF$ and for every $t\in\TT$, the natural homomorphism
\[
K_t(\IZ[\zeta_c])
\TO
\;\smallprod_{\mathclap{p\textup{ prime}}}\;
K_t\Bigl(\IZ_p\tensor_\IZ\IZ[\zeta_c];\IZ_p\Bigr)
\]
is $\IQ$-injective, where $c$ is the order of~$C$ and $\zeta_c$ is any primitive $c$-th root of unity.
\end{enumerate}
Then the restriction
\begin{equation}
\label{eq:main-technical}
\bigoplus_{(C)\in(\CF)}
\,
\bigoplus_{\substack{s+t=n\\s\ge0,\,t\in\TT}}
H_s(BZ_GC;\IQ)\tensor_{\IQ[W_G C]}\Theta_C\Bigl(K_t(\IZ[C])\tensor_{\IZ}\IQ\Bigr)
\TO
K_n(\IZ[G])\tensor_\IZ\IQ
\end{equation}
of the rationalized Farrell-Jones assembly map~\eqref{eq:Q-FJ-assembly} to the summands satisfying $(C)\in(\CF)$ and $t\in\TT$ is injective for all~$0\leq n\leq N$.
\end{maintechnicaltheorem}

It is clear that Theorem~\ref{main} is a special case of Theorem~\ref{main-technical}, since Assumption \assum{A}{\CF} implies \techA{\CF}{N} for all~$N$, and \assum{B}{\CF} is the same as~\techB{\CF}{\IN}.
Notice that also \autoref{BHM} is a special case of Theorem~\ref{main-technical}.
This is because for the trivial family $\CF=1$ the maps \eqref{eq:Q-classical-assembly} and~\eqref{eq:main-technical} coincide,
and Assumption~\techB{1}{\IN} holds by \autoref{no-B-in-BHM}.
Moreover, for torsion-free groups Theorems \ref{BHM} and~\ref{main} coincide.

\autoref{no-B-in-Wh} says that Assumption~\techB{\FCyc}{\{0,1\}} is always satisfied.
This allows us to deduce \autoref{Whitehead} about the rationalized assembly map for Whitehead groups from Theorem~\ref{main-technical}.
The proof is given in \autoref{WHITEHEAD}.
In fact, the following strengthening of \autoref{Whitehead} is also true.

\begin{addendum}
\label{add-Whitehead}
The map~\eqref{eq:whitehead} is injective if, for every finite cyclic subgroup~$C$ of~$G$, the abelian groups $H_1(BZ_GC;\IZ)$ and $H_2(BZ_GC;\IZ)$ are \emph{almost} finitely generated.
\end{addendum}

In \autoref{STRATEGY} we outline the strategy of the proof of Theorems \ref{main} and~\ref{main-technical}, and explain in \autoref{limitation} why with this strategy the homological finiteness assumptions cannot be removed.
The main ingredient in the proof is, as in~\cite{BHM}, the cyclotomic trace map from algebraic $K$-theory to topological cyclic homology.
A~key role is played by a generalization of the functor~$C$ defined by B\"okstedt, Hsiang, and Madsen in~\cite{BHM}*{(5.14) on page~497}; see \autoref{FUNCTOR-C}.
In Theorems~\ref{split-THH} and~\ref{split-C}, we prove the following general results about the assembly maps for topological Hochschild homology and for~$C$, which we believe are of independent interest.
Here we say that a map of spectra~$\MOR{f}{\BX}{\BY}$ is split injective if there is a map $\MOR{g}{\BY}{\BZ}$ to some other spectrum~$\BZ$ such that $g\circ f$ is a $\pi_*$-isomorphism.
We say that $f$ is $\pi_n^\IQ$-injective if $\pi_n(f)\tensor_\IZ\IQ$ is injective.
The conditions very well pointed and \cplus{} for symmetric ring spectra are introduced in \autoref{SYMMETRIC} and are very mild.
\cplus[*] is a minor strengthening of connectivity; the sphere spectrum and suitable models for all Eilenberg-Mac Lane ring spectra are \cplus.

\begin{theorem}
\label{THH-and-C}
Let $G$ be a group and let $\CF$ be a family of subgroups of~$G$.
Let $\spec{A}$ be a very well pointed symmetric ring spectrum.
\begin{enumerate}
\item\label{i:THH}
The assembly map for topological Hochschild homology
\[
EG(\CF)_+\sma_{\Or G}\THH(\spec{A}[\oid{G}{-}])
\TO
\THH(\spec{A}[G])
\]
is always split injective, and it is a $\pi_*$-isomorphism if $\Cyc\subseteq\CF$, i.e., if $\CF$ contains all finite and infinite cyclic subgroups of~$G$.
\item\label{i:C}
Assume that $\spec{A}$ is \cplus.
If $\CF\subseteq\Fin$ and if Assumption~\ref{techA} in \autoref{main-technical} holds,
then the assembly map for \BHM's functor~$C$
\[
EG(\CF)_+\sma_{\Or G}\C(\spec{A}[\oid{G}{-}];p)
\TO
\C(\spec{A}[G];p)
\]
is $\pi_n^\IQ$-injective for all~$n\le N$ and all primes~$p$.
\end{enumerate}
\end{theorem}

In a separate paper~\cite{tc} we use the techniques developed here to show that, under stronger finiteness assumptions on~$G$, one also obtains injectivity statements for the assembly map for topological cyclic homology.
In particular, we show that if there is a universal space for proper actions~$EG(\Fin)$ of finite type, then the assembly map for topological cyclic homology
\[
EG(\Fin)_+\sma_{\Or G}\TC(\spec{A}[\oid{G}{-}];p)
\TO
\TC(\spec{A}[G];p)
\]
is split injective for any \cplus{} symmetric ring spectrum~$\spec{A}$.
These results are not needed here and do not lead to stronger statements in algebraic $K$-theory.

We conclude this introduction with some more information about the Farrell-Jones Conjecture in general, not just its rationalized version \autoref{Q-FJ-Conj}.
Recall that the conjecture predicts that the Farrell-Jones assembly map~\eqref{eq:FJ-assembly} is a $\pi_*$-isomorphism.
There is a completely analogous conjecture for $L$-theory.
Both conjectures were stated by Farrell and Jones in~\cite{FJ-iso}*{1.6 on page~257}, using a different but equivalent formulation; see~\cite{Hambleton-Pedersen}.
Among many others, the results in~\cites{FH-spaceform, FH-poly, FJ-dyn, FJ-Mostow} provided evidence and paved the way for these conjectures.
The Farrell-Jones Conjectures have many applications; most notably, they imply the Borel Conjecture about the topological rigidity of aspherical manifolds in dimensions greater than or equal to~five.
We refer to~\cites{LR-survey, Bartels-Lueck-Reich-appl, L-ICM} for comprehensive surveys about these conjectures and their applications.

Besides \autoref{BHM}, there are many other important injectivity results, among which~\cites{FH-Novikov, Carlsson-Pedersen, Carlsson-Goldfarb, Rosenthal, Bartels-Rosenthal, Bartels-Lueck-HK, Kasprowski, RTY}.
Furthermore, in recent years there has been tremendous progress in proving isomorphisms results, even for the more general version of the Farrell-Jones Conjectures with coefficients in additive categories and with finite wreath products~\cites{Bartels-Reich-coeff, Bartels-Lueck-coeff}.
This more general version has good inheritance properties:
it passes to subgroups, overgroups of finite index, finite direct products, free products, and colimits over arbitrary directed systems;
see for instance \cites{Bartels-Echterhoff-Lueck, Bartels-Lueck-ind}.
Even this more general version has now been verified for hyperbolic groups, CAT(0)-groups, arithmetic groups, lattices in virtually connected locally compact second countable Hausdorff groups, virtually solvable groups, fundamental groups of $3$-manifolds;
see for example~\cites{Bartels-Farrell-Lueck, Bartels-Lueck-annals, Bartels-Lueck-Reich-invent, Bartels-Lueck-Reich-Rueping, Bartels-Reich-jams, FW, Kammeyer-Lueck-Rueping, Roushon, Rueping, Wegner-CAT0, Wegner-solv}.

In this work we only consider the algebraic $K$-theory version of the conjecture and only with coefficients in the ring of integers~$\IZ$ or the sphere spectrum~$\IS$.
For more information about why the method used in~\cite{BHM} and here does not work for $L$-theory, we refer to~\cite{Madsen}*{Remark~4.5.5 on pages~285--286}.
Notice, however, that all the results mentioned above are about groups satisfying some geometric condition, typically some version of non-positive curvature, and the proofs use geometric tools like controlled topology and flows.
Our results only need homological finiteness conditions, and the methods of proof are completely different, coming mostly from equivariant stable homotopy theory.
We cover prominent groups, such as outer automorphism groups of free groups and Thompson's groups, for which no previous results were known.
In particular, \autoref{Whitehead} about Whitehead groups gives a deep result under very mild assumptions.

\subsection*{Acknowledgments}
Part of this work was financially supported by CRC~647 in Berlin, by the first author's Leibniz Award, granted by the DFG (Deutsche Forschungsgemeinschaft), and by his ERC Advanced Grant ``KL2MG-interactions'' (no.~662400), granted by the European Research Council.


\section{Discussion of the assumptions}
\label{ASSUMPTIONS}

In this section we discuss the assumptions of \autoref{main}.
We describe many groups that satisfy Assumption~\ref{our-A} and explain why Assumption~\ref{our-B} is conjecturally always true.

Let $\CF$ be a family of subgroups of $G$, i.e., a collection of subgroups that is closed under passage to subgroups and conjugates.
Examples of families are
\[
\CF=1,\FCyc,\Fin,\Cyc,\VCyc,
\]
the families consisting respectively of the trivial subgroup only, of all finite cyclic, finite, cyclic, or virtually cyclic subgroups of~$G$.
Recall that a group is called virtually cyclic if it contains a cyclic subgroup of finite index.
A \emph{universal space for~$\CF$} is a $G$-CW complex~$EG(\CF)$ such that, for all subgroups~$H\le G$, the $H$-fixed point space
\[
\bigl(EG(\CF)\bigr)^H
\,
\text{ is }
\begin{cases}
\text{empty}&\text{if $H\not\in\CF$;}\\
\text{contractible}&\text{if $H\in\CF$.}
\end{cases}
\]
Such a space is unique up to $G$-homotopy equivalence, and it receives a $G$-map that is unique up to $G$-homotopy from every $G$-CW complex all of whose isotropy groups are in~$\CF$.

When $\CF$ is the trivial family~$1$, then $EG(1)$ if the universal cover $EG=\widetilde{BG}$ of the classifying space of~$G$.
When $\CF$ is the family~$\Fin$ of finite subgroups, $EG(\Fin)$ is often denoted $\underline{E}G$ and called the universal space for proper actions.

The following result can be used to verify Assumption~\ref{our-A} for a large class of examples, as illustrated below.
We say that a $G$-CW-complex has finite type if it has only finitely many $G$-cells in each dimension.
A group~$G$ has type~$F_\infty$ if there is a classifying space~$BG$ of finite type.
Clearly, if $G$ has type~$F_\infty$, then $H_s(BG;\IZ)$ is a finitely generated abelian group for every~$s$.

\begin{proposition}
\label{EGF-finite-type}
Let $\CF\subseteq\Fin$ be a family of finite subgroups of~$G$.
Then there is a universal space~$EG(\CF)$ of finite type if and only if there are only finitely many conjugacy classes of subgroups in~$\CF$ and $Z_GH$ has type~$F_\infty$ for each~$H\in\CF$.
\end{proposition}

\begin{proof}
Since $H$ is finite and $Z_GH$ has finite index in~$N_GH$, notice that $Z_GH$ has type~$F_\infty$ if and only if~$N_GH$ has type~$F_\infty$ if and only if~$N_GH/H$ has type~$F_\infty$; see for example~\cite{Geoghegan}*{Corollary~7.2.4 on page~170, Theorem~7.2.21 and Exercise~7.2.1 on pages~175--176}.
In the case when $\CF=\Fin$ \cite{L-type}*{Theorem~4.2 on page~195} states that there is an $EG(\CF)$ of finite type if and only if there are only finitely many conjugacy classes of subgroups in~$\CF$ and $N_GH/H$ has type~$F_\infty$ for each~$H\in\CF$.
By inspection, the proof of this result in~\cite{L-type} goes through for arbitrary~$\CF\subseteq\Fin$.
Notice that \cite{L-type}*{Lemma~1.3 on page~180}, which is used in the proof, in this general case shows that if~$H\in\CF$ then $(EG(\CF))^H$ is a model for~$E(N_GH/H)(q_*(\CF))$, where $\MOR{q}{N_GH}{N_GH/H}$ is the quotient map and $q_*(\CF)=\SET{F\le N_GH/H}{q^{-1}(F)\in\CF}$.
The rest of the proof works verbatim.
\end{proof}

\begin{corollary}
\label{EGFin-finite-type=>A}
If there is a universal space~$EG(\Fin)$ of finite type then Assumption~\ref{our-A} is satisfied.
\end{corollary}

The existence of a universal space~$EG(\Fin)$ of finite type is already a fairly mild condition.
For example, all the following groups have even finite (i.e., finite type and finite dimensional) models for~$EG(\Fin)$, and therefore satisfy Assumption~\ref{our-A} by \autoref{EGFin-finite-type=>A}.
\autoref{above-L} about eventual injectivity applies to all these groups.
\begin{itemize}
\item
Word-hyperbolic groups.
\item
Groups acting properly, cocompactly, and by isometries on a CAT(0)-space.
\item
Arithmetic groups in semisimple connected linear $\IQ$-algebraic groups.
\item
Mapping class groups.
\item
Outer automorphism groups of free groups.
\end{itemize}
For more information we refer to \cite{L-survey}*{Section~4} and also to~\cite{Mislin} in the case of mapping class groups.

Moreover, Assumption~\ref{our-A} is much weaker than $EG(\Fin)$ having finite type.
For example, Thompson's group~$T$ of orientation preserving, piece-wise linear, dyadic homeomorphisms of the circle contains finite cyclic subgroups of any given order, and hence there can be no $ET(\Fin)$ of finite type by \autoref{EGF-finite-type}.
However, in~\cite{GV}*{Corollary~1.5} it is proved that Thompson's group~$T$ satisfies Assumption~\ref{our-A}.

For Thompson's group~$T$ and for outer automorphism groups of free groups there seem to be no previously known results about the Farrell-Jones Conjecture.

We now turn to Assumption~\ref{our-B}.
First of all, notice that this assumption depends on the group~$G$ only through the set of numbers~$c$ that appear as orders of finite cyclic subgroups of $G$.
Fix integers~$c\ge1$, $t\ge0$, and a prime number~$p$.
The map in Assumption~\ref{our-B}, followed by the projection to the factor indexed by~$p$, is the diagonal~$\beta$ of the following commutative diagram.
\begin{equation}
\label{eq:pcompl-pcompl}
\begin{tikzcd}[row sep=scriptsize]
\ds K_t\bigl( \IZ [ \zeta_c ] \bigr)
\arrow[dr, "\beta" description]
\arrow[r]
\arrow[d, shorten <=-.25em]
&
\ds K_t\bigl(\IZ_p\tensor_\IZ \IZ[\zeta_c]\bigr)
\arrow[d, shorten <=-.25em]
\\
\ds K_t\bigl(\IZ[\zeta_c]; \IZ_p \bigr)
\arrow[r]
&
\ds K_t\bigl(\IZ_p\tensor_\IZ \IZ[\zeta_c];\IZ_p\bigr)
\end{tikzcd}
\end{equation}
The vertical maps are induced by the natural map $\K(R)\TO\K(R)\pcompl$ from the algebraic $K$-theory spectrum to its $p$-completion.
The horizontal maps are induced by the ring homomorphism
\(
\IZ [ \zeta_c ] \TO \IZ_p \tensor_{\IZ} \IZ [ \zeta_c ]
\).
Notice that $\IZ [\zeta_c ]$ is the ring of integers in the number field $\IQ ( \zeta_c )$, and that there is a decomposition
\[
\IZ_p \tensor_{\IZ} \IZ [ \zeta_c ] = \smallprod_{{\mathfrak{p}} | p } \CO_{\mathfrak{p}}
\,,
\]
where the product runs over all primes $\mathfrak{p}$ in $\IZ[\zeta_c]$ over~$p$, and $\CO_{\mathfrak{p}}$ is the completion of $\IZ[ \zeta_c ]$ with respect to~$\mathfrak p$.
Since $K_t ( - )$ and $K_t (- ; \IZ_p )$ respect finite products, one can rewrite diagram~\eqref{eq:pcompl-pcompl} with suitable products on the right-hand side.

We first consider the cases $t=0$ and~$t=1$.

\begin{proposition}
\label{no-B-in-Wh}
If $t=0$ or~$t=1$, then the map~$\beta$ in diagram~\eqref{eq:pcompl-pcompl} is $\IQ$-injective for any prime~$p$ and for any $c\ge1$.
In particular, in the notation of \autoref{main-technical}, Assumption~\techB{\FCyc}{\{0,1\}} is always satisfied.
\end{proposition}

\begin{proof}
The second statement clearly follows from the first.

If $t=0$, the map $\beta$ is easily seen to be $\IQ$-injective, because $K_0(\IZ[\zeta_c])$ is the sum of~$\IZ$ and a finite group, the ideal class group of $\IQ(\zeta_c)$.

If $t=1$, the determinant identifies $K_1(R)$ with the group of units~$R^\times$, provided that $R$ is a ring of integers in a number field~\cite{BMS}*{Corollary~4.3 on page~95}, or a local ring~\cite{Weibel-K-book}*{Lemma~III.1.4 on page~202}.
The abelian group $\IZ[\zeta_c]^\times$ is finitely generated, $\CO_{\mathfrak{p}}^\times$ is the sum of a finite group and a finitely generated $\IZ_p$-module, and both $K_0(\IZ[\zeta_c])$ and $K_0(\smallprod_{\mathfrak{p}|p}\CO_\mathfrak{p})$ are finitely generated abelian groups.
Hence all these groups have bounded $p$-torsion and diagram~\eqref{eq:pcompl-pcompl} can be identified with the following diagram; compare \autoref{P-COMPL}.
\begin{equation}
\label{eq:units-pcompl-pcompl}
\begin{tikzcd}[row sep=tiny]
\ds\IZ[\zeta_c]^\times
\arrow[r]
\arrow[d]
&
\ds\smallprod_{\mathfrak{p}|p}\CO_\mathfrak{p}^\times
\arrow[d, shorten <=-.5em]
\\
\ds\bigl(\IZ[\zeta_c]^\times\bigr)\pcompl
\arrow[r]
&
\ds\smallprod_{\mathfrak{p}|p}\bigl(\CO_\mathfrak{p}^\times\bigr)\pcompl
\end{tikzcd}
\vspace{-1.75ex}
\end{equation}
Here $A\pcompl$ is the $p$-completion of the abelian group~$A$, defined as the inverse limit $\lim_i A/p^i$.
The natural map $A\TO A\pcompl$ is $\IQ$-injective provided that $A$ is a  finitely generated $\IZ$- or $\IZ_p$-module.
Therefore the vertical maps in~\eqref{eq:units-pcompl-pcompl} are $\IQ$-injective.
Since the upper horizontal map is clearly injective, too, the result is proved.
\end{proof}

We remark that also the lower horizontal map in diagram~\eqref{eq:units-pcompl-pcompl} is $\IQ$-injective.
However, this is highly non-trivial.
It follows from the fact that the following conjecture was proved by Brumer~\cite{Brumer}*{Theorem~2 on page~121} for abelian number fields, and so in particular for $F=\IQ(\zeta_c)$.

\begin{conjecture}[The Leopoldt Conjecture]
\label{Leopoldt-Conj}
Let $\CO_F$ be the ring of integers in a number field $F$, and let $p$ be a prime.
Then the map
\[
\bigl(\CO_F^\times\bigr)\pcompl
\TO
\bigl(\bigl(\IZ_p\tensor_\IZ\CO_F\bigr)\!{}^\times\bigr)\pcompl
\]
induced by the natural inclusion $\CO_F \TO \IZ_p \tensor_\IZ\CO_F$ is injective.
\end{conjecture}

For more information about the Leopoldt \autoref{Leopoldt-Conj} we refer to~\cite{NSW}*{Section~X.3} and \cite{Washington}*{Section~5.5}.
The conjecture is open for arbitrary number fields.

Next we discuss diagram \eqref{eq:pcompl-pcompl} for~$t\geq2$.
We can restrict to the case when $t$ is odd, because if $t\geq2$ is even then the rationalized upper left corner~$K_t(\IZ[\zeta_c])\tensor_\IZ\IQ$ vanishes by~\cite{Borel}*{Proposition~12.2 on page~271}.

The $p$-completed $K$-theory groups of $\CO_{\mathfrak{p}}$ and hence of $\IZ_p\tensor_\IZ\IZ[\zeta_c]$ are completely known; see for example~\cite{Weibel-survey}*{Theorem~61 on page~165}.
However, the non-completed $K$-theory groups $K_t(\CO_\mathfrak{p})\tensor_\IZ\IQ$ seem to be rather intractable; compare~\cite{Weibel-survey}*{Warning~60 on pages~164--165}.
Hence the strategy of going first right and then down in diagram \eqref{eq:pcompl-pcompl}, which was successful without any effort in the case $t=1$, does not seem very promising for~$t\geq2$.

However, since the abelian groups $K_t(\IZ[\zeta_c])$ are finitely generated by~\cite{Quillen-fg}, the left-hand vertical map in diagram~\eqref{eq:pcompl-pcompl} is always $\IQ$-injective.

The following conjecture concerns the lower horizontal map in diagram~\eqref{eq:pcompl-pcompl}.
It is the rationalized analogue for higher algebraic $K$-theory of the Leopoldt \autoref{Leopoldt-Conj}.
It appeared in~\cite{Schneider}*{page~192}, and we refer to it as the Schneider Conjecture.
A related question was asked by Wagoner in~\cite{Wagoner}*{page~242}.
In \autoref{SCHNEIDER} we briefly discuss some equivalent formulations, including Schneider's original one.

\begin{conjecture}[The Schneider Conjecture]
\label{Schneider-Conj}
Let $\CO_F$ be the ring of integers in a number field $F$, and let $p$ be a prime.
Then for any $t\geq2$ the map
\[
K_t(\CO_F;\IZ_p)
\TO
K_t(\IZ_p\tensor_\IZ\CO_F;\IZ_p)
\]
induced by the natural inclusion $\CO_F\TO\IZ_p\tensor_\IZ\CO_F$ is $\IQ$-injective.
\end{conjecture}

The discussion above can then be summarized as follows.

\begin{proposition}
\label{Schneider=>B}
If the Schneider \autoref{Schneider-Conj} holds for $F=\IQ(\zeta_c)$, $t\ge2$, and~$p$, then the map~$\beta$ in diagram~\eqref{eq:pcompl-pcompl} is $\IQ$-injective.
If for each $c\ge1$ and each $t\ge2$ there is a prime~$p$ such that the Schneider \autoref{Schneider-Conj} holds for $F=\IQ(\zeta_c)$ and~$t$, then Assumption~\ref{our-B} is always satisfied.
\end{proposition}

By~\autoref{Schneider} we know that for a fixed number field~$F$ and a fixed odd prime~$p$ there can only be finitely many values of~$t$ for which the Schneider \autoref{Schneider-Conj} is false.

In light of Propositions \ref{no-B-in-Wh} and~\ref{Schneider=>B}, we view our Assumption~\ref{our-B} as a weak version of the conjectures of Leopoldt and Schneider.

We conclude this section with the following well known result, which explains why there is no Assumption~\assum{B}{1} in \autoref{BHM}.
In fact, a version of this result enters in the original proof of \autoref{BHM} by \BHM; compare~\cite{BHM}*{Remark~9.6(ii) on page~533}.
Recall that a prime $p$ is called regular if $p$ does not divide the order of the ideal class group of $\IQ(\zeta_p)$.

\begin{proposition}
\label{no-B-in-BHM}
If $c=1$ and $p$ is a regular prime, then the map
\[
\MOR{\beta}{K_t(\IZ)}{K_t(\IZ_p;\IZ_p)}
\]
in diagram~\eqref{eq:pcompl-pcompl} is $\IQ$-injective for any~$t\ge0$.
In the notation of \autoref{main-technical}, Assumption~\techB{1}{\IN} is satisfied.
\end{proposition}

\begin{proof}
If $p$ is an odd regular prime, by \autoref{Schneider-reductions}\ref{i:extensions} and \autoref{Schneider} it suffices to verify statement~\ref{i:Iwasawa} in \autoref{Schneider} for $F=\IQ(\zeta_p)$.
This follows from the fact that in this case $Z=0$; compare~\cite{Schneider}*{Beispiele~ii) on page~201}.

If $p=2$, this is proved in~\cite{Rognes-2adic}*{Theorem~7.7 on page~318}.
\end{proof}


\section{Strategy of the proof and outline}
\label{STRATEGY}

In this section we describe the strategy of the proof of \autoref{main-technical} and the outline of the paper.
Recall that, as explained in the introduction, Theorems \ref{main} and~\ref{BHM} are special cases of Theorem~\ref{main-technical}.
The strategy can be summarized by the following commutative diagram, which we explain below and refer to as \emph{main diagram}.
\begin{equation}
\label{eq:main-diagram}
\hspace{-.7em}\adjustbox{scale=.968}{
\begin{tikzcd}[column sep=-.1em, row sep=scriptsize]
\ds H^G_n\bigl(EG(\VCyc);\K^{\ge0}(\IZ[\oid{G}{-}])\bigr)
\arrow[rr]
&
\raisebox{1.25ex}{$\scriptstyle\assembly{1}$}
\vphantom{\Bigr[\smallprod_{p}}
&
K_n(\IZ[G])
\\
\ds H^G_n\bigl(EG(\CF);\K^{\ge0}(\IZ[\oid{G}{-}])\bigr)
\arrow[rr]
\arrow[u, "H^G_n(\incl)\,"]
&
\raisebox{1.25ex}{$\scriptstyle\assembly{2}$}
\vphantom{\Bigr[\smallprod_{p}}
&
\ds\pi_n\K^{\ge0}(\IZ[G])
\arrow[equal, u]
\\
\ds H^G_n\bigl(EG(\CF);\K^{\ge0}(\IS[\oid{G}{-}])\bigr)
\arrow[rr]
\arrow[u, "\ell_\%\,"]
\arrow[d, "\tau_\%\,"']
&
\raisebox{1.25ex}{$\scriptstyle\assembly{3}$}
\vphantom{\Bigr[\smallprod_{p}}
&
\pi_n\K^{\ge0}(\IS[G])
\arrow[u, "\,\ell"']
\arrow[d, "\,\tau"]
\\
\ds H^G_n\Bigl(EG(\CF);\smash{\smallprod_{p}}\TC(\IS[\oid{G}{-}];p)\Bigr)
\arrow[rr]
\arrow[d, "\sigma_\%\,"']
&
\raisebox{1.25ex}{$\scriptstyle\assembly{4}$}
&
\ds\pi_n\Bigl(\smash{\smallprod_{p}}\TC(\IS[G];p)\Bigr)
\arrow[d, "\,\sigma"]
\\
\ds H^G_n\Bigl(EG(\CF);\THH(\IS[\oid{G}{-}])\!\times\!\smallprod_{p}\C(\IS[\oid{G}{-}];p)\Bigr)
\arrow[rr]
&
\raisebox{1.25ex}{$\scriptstyle\assembly{5}$}
&
\ds\pi_n\Bigl(\THH(\IS[G])\!\times\!\smallprod_{p}\C(\IS[G];p)\Bigr)
\end{tikzcd}
}
\end{equation}
The products are indexed over all prime numbers~$p$.

Here $\CF$ is a family of subgroups of~$G$ and $EG(\CF)$ is the corresponding universal space; see the beginning of the previous section.
We assume that $\CF\subseteq\FCyc$, i.e., that all subgroups in~$\CF$ are finite cyclic.

The notation~$\K^{\ge0}$ stands for the connective algebraic $K$-theory spectrum, and $\THH$, $\TC$, and~$\C$ denote topological Hochschild homology, topological cyclic homology, and \BHM's functor~$C$; see Sections~\ref{THH}, \ref{TC}, and~\ref{FUNCTOR-C}.
We construct functors
\begin{equation}
\label{eq:theories}
\K^{\ge0}(\spec{A}[\oid{G}{-}])
,\quad
     \THH(\spec{A}[\oid{G}{-}])
,\quad
       \C(\spec{A}[\oid{G}{-}];p)
,\quad
      \TC(\spec{A}[\oid{G}{-}];p)
,
\end{equation}
for any coefficient ring or symmetric ring spectrum~$\spec{A}$; compare \autoref{TRC}.
These are all functors from the orbit category~$\Or G$ of~$G$ to the category of naive spectra~$\IN\Sp$.
Given any such functor
\(
\MOR{\BE}{\Or G}{\IN\Sp}
\),
the {assembly map} for~$\BE$ with respect to a family~$\CF$ of subgroups of~$G$ is a map of spectra
\begin{equation}
\label{eq:generic-assembly-outline}
\MOR{\asbl}{EG(\CF)_+\sma_{\Or G}\BE}{\BE(G/G)}
\end{equation}
defined in~\autoref{ASSEMBLY}.
When $\BE$ is one of the functors in~\eqref{eq:theories}, then there are isomorphisms
\[
\K^{\ge0}(\spec{A}[\oid{G}{\,(G/G)}])\cong\K^{\ge0}(\spec{A}[G])
\,,
\]
and similarly for the other examples.

Taking homotopy groups of~$X_+\sma_{\Or G}\BE$ defines a $G$-equivariant homology theory $H^G_*(X;\BE)$.
The assembly map~\eqref{eq:generic-assembly-outline} then gives a group homomorphism
\[
\MOR{\alpha=\pi_n(\asbl)}{H^G_n(EG(\CF);\BE)}{\pi_n\BE(G/G)}
\,.
\]
All horizontal maps in diagram~\eqref{eq:main-diagram} are given by assembly maps.
The top horizontal map~$\assembly{1}$ is the Farrell-Jones assembly map for connective algebraic $K$-theory.

The vertical map~$H^G_n(\incl)$ is induced by the inclusions $\CF\subseteq\FCyc\subseteq\VCyc$.
If~$\CF=\FCyc$ then $H^G_n(\incl)$ is a $\IQ$-isomorphism by \autoref{S_t(H;K)-computation} and~\cite{LS}*{Theorem~0.3 on page~2}.

\autoref{Chern} and \autoref{EG(F)-computation} identify the source of~$\assembly{2}$ rationally with
\begin{equation}
\label{eq:Q-computation-of-source}
\bigoplus_{(C)\in(\CF)}
\bigoplus_{\substack{s+t=n\\s,t\ge0}}
H_s(BZ_GC;\IQ)\tensor_{\IQ[W_G C]}\Theta_C\Bigl(K_t(\IZ[C])\tensor_{\IZ}\IQ\Bigr)
\,.
\end{equation}
Under this identification, the map~\eqref{eq:main-technical} in \autoref{main-technical} is the restriction of $\assembly{2}\tensor_\IZ\IQ$ to the summands satisfying~$t\in\TT$.
In the case $\CF=1$, it is the rationalized classical assembly map~\eqref{eq:Q-classical-assembly}.
The map~\eqref{eq:Q-FJ-assembly-conn} in \autoref{main} is identified with $\assembly{1}\tensor_\IZ\IQ$.

All other vertical maps are induced by natural transformations of functors from $\Or G$ to~$\IN\Sp$; compare \autoref{TRC}.

The map $\ell$ is the linearization map given by changing the group ring coefficients from the sphere spectrum~$\IS$ to the integers~$\IZ$.
The spectrum $\K^{\ge0}(\IS[G])$ is a model for~$\mathbf{A}(BG)$, Waldhausen's algebraic $K$-theory of the classifying space of~$G$.
By a result of Waldhausen~\cite{Waldhausen-A1}*{Proposition~2.2 on page~43}, $\ell$ is a $\IQ$-isomorphism for any group~$G$.
The linearization map extends to a natural transformation $\K^{\ge0}(\IS[\oid{G}{-}])\TO\K^{\ge0}(\IZ[\oid{G}{-}])$ of functors $\Or G\TO\IN\Sp$, and therefore it induces the map~$\ell_\%$.
Using \autoref{Chern} and \autoref{rel-Chern} we conclude that also $\ell_\%$ is a $\IQ$-isomorphism.

The map~$\tau$, followed by the projection to the factor indexed by the prime~$p$, is \BHM's cyclotomic trace map~$\trc_p$ from algebraic $K$-theory to topological cyclic homology at~$p$.
Using the model of Dundas-Goodwillie-McCarthy \cite{Dundas} recalled in \autoref{TRC}, this map also extends to a natural transformation and induces the map~$\tau_\%$.

For any coefficient ring or symmetric ring spectrum~$\spec{A}$ there is a projection map $\TC(\spec{A}[G];p)\TO\THH(\spec{A}[G])$.
In the special case of spherical coefficients $\spec{A}=\IS$, there is also a map $\TC(\IS[G];p)\TO\C(\IS[G];p)$ from topological cyclic homology to \BHM's functor~$C$; compare \autoref{A=S}.
The map~$\sigma$ is then defined as the product of the following maps.
For $p=11$ take the induced map $\TC(\IS[G];11)\TO\THH(\IS[G])\times\C(\IS[G];11)$.
For $p\ne11$ take $\TC(\IS[G];p)\TO\C(\IS[G];p)$.
The choice of the prime~$11$ is arbitrary; we could fix any other prime.
Also $\sigma$ extends to a natural transformation and induces the map~$\sigma_\%$.

This completes the explanation of the main diagram~\eqref{eq:main-diagram}.
The proof of \autoref{main-technical} is then an immediate consequence of the following two results.
Recall that our goal is to show that the restriction of~$\assembly{2}\tensor_\IZ\IQ$ to the summands satisfying~$t\in\TT$ in~\eqref{eq:Q-computation-of-source} is injective.
We already know that $\ell_\%$ and~$\ell$ are $\IQ$-isomorphisms.

\begin{theorem}[Splitting Theorem]
\label{splitting}
If Assumption~\ref{techA} holds, then for every~$n\le N$ the map~$\assembly{5}\tensor_\IZ\IQ$ is injective.
\end{theorem}

\begin{theorem}[Detection Theorem]
\label{detection}
If Assumption~\techB{\CF}{\TT} holds, then for every~$n$ the restriction of~$(\sigma_\%\circ\tau_\%)\tensor_\IZ\IQ$ to the summands satisfying~$t\in\TT$ is injective.
\end{theorem}

The main ingredient in the proof of the Detection \autoref{detection} is the theory of equivariant Chern characters due to the first author, which we explain in \autoref{CHERN}.
To show in \autoref{TRC} that the theory can be applied to our examples, we develop a general framework in \autoref{MACKEY}, which uses the formalism of enriched categories of modules that we introduce in \autoref{MODULES}.
The proof of the Detection \autoref{detection} is then completed in \autoref{PROOF-DETECTION}.
In this step we also use a deep result of Hesselholt and Madsen~\cite{HM-top}, building on a theorem of McCarthy~\cite{McCarthy}, about the cyclotomic trace map.

The proof of the Splitting \autoref{splitting} occupies the first half of the paper, Sections~\ref{CN} to~\ref{FUNCTOR-C}.
The proof proceeds in stages, starting with general results about the assembly map for the cyclic nerve~$CN$, and leading to \autoref{split-C}; see also \autoref{THH-and-C} mentioned at the end of the introduction.
The most delicate step occurs when passing from topological Hochschild homology to \BHM's functor~$C$, and is based on the results we developed in \cite{LRV} and~\cite{RV}.

We summarize the steps of the proof in the following roadmap.
Every line represents a result about the assembly map for the theory indicated on the left.
Next to the vertical implication symbols, we indicate which assumptions are needed, and some of the important results that are used.
\[
\hspace{-7.5em}\adjustbox{scale=.87}{
\begin{tikzcd}[column sep=.8em]
&
{\RMheader{143}{For \ldots}}
&
{\RMheader{ 43}{\ldots\ the $\CF$-assembly map is \ldots}}
\\
\RMref{Prop \ref{split-CN}}
&
\smash{\ds\real{CN_\bullet(G)}}
&
\text{split inj for all~$\CF$, $\pi_*$-iso if $\Cyc\subseteq\CF$}
\arrow[dd, Rightarrow, shorten <=1em, shorten >=1em,
"\ \ \text{uses \cites{HM-top, Shipley}}"
]
\\
\\
\RMref{Thm \ref{split-THH}}
&
\smash{\ds\THH(\spec{A}[G])}
&
\text{split inj for all~$\CF$, $\pi_*$-iso if $\Cyc\subseteq\CF$}
\arrow[dd, Rightarrow, shorten <=1em, shorten >=1em]
\\
\\
\RMref{Cor \ref{split-THH-hC}}
&
\smash{\ds\THH(\spec{A}[G])_{hC_{p^{n}}}}
&
\text{split inj for all~$\CF$, $\pi_*$-iso if $\Cyc\subseteq\CF$}
\arrow[dd, Rightarrow, shorten <=1em, shorten >=1em,
"\text{\scriptsize\textsl{if $\spec{A}$ is \cplus}}\ \ "' pos=.46,
"\ \ \text{uses \cites{HM-top, Blumberg, RV}}"
]
\\
\\
\RMref{Cor \ref{split-THH-sh-fix}}
&
\smash{\ds\hofib\left(
{\THH(\spec{A}[G])^{C_{p^{n+1}}}}
\TO[R]
{\THH(\spec{A}[G])^{C_{p^{n}}}}
\!
\right)}
&
\text{split inj for all~$\CF$, $\pi_*$-iso if $\Cyc\subseteq\CF$}
\arrow[dd, Rightarrow, shorten <=1em, shorten >=1em,
"\text{\scriptsize\textsl{\& if $G$ fulfills \ref{techA}}}\ \ "',
"\ \ \text{uses \cite{LRV}}"
]
\\
\\
\RMref{Thm \ref{split-C}}
&
\smash{\ds
\C(\spec{A}[G];p)
\quad\text{\&}\quad
\THH(\spec{A}[G])\times\prodp\C(\spec{A}[G];p)}
&
\text{$\smash{\pi^\IQ_n}$-injective if~$\CF\subseteq\FCyc$, $n\le N$}
\arrow[dd, Rightarrow, shorten <=1em, shorten >=1em,
"\text{\scriptsize\textsl{\& if $G$ fulfills \techB{\CF}{T}}}\ \ "',
"\ \ \text{uses \cites{BHM, HM-top}}" pos=.38,
"\ \ \text{\makebox[\widthof{uses}][l]{and} \cites{L-Chern, Dundas}}" pos=.64
]
\\
\\
\RMref{Thm \ref{main-technical}}
&
\smash{\ds
\K^{\ge0}(\IS[G])
\quad\text{\&}\quad
\K^{\ge0}(\IZ[G])
}
&
\text{$\smash{\pi^\IQ_n}$-injective if~$\CF\subseteq\FCyc$, $n\le N$, $t\in\TT$}
\end{tikzcd}
}
\]
\medskip

We conclude this section with two important remarks.

\begin{remark}[Limitation of trace methods]
\label{limitation}
The strategy used both here and in~\cite{BHM} to prove rational injectivity of the assembly map for algebraic $K$-theory breaks down if the homological finiteness assumptions \ref{BHM-A}, \ref{our-A}, or~\ref{techA} are dropped.

For example, let $G$ be the additive group of the rational numbers~$\IQ$.
Since $\IQ$ is torsion-free, $\FCyc=1$ and Theorems \ref{BHM} and~\ref{main} would coincide.
Since $H_1(B\IQ;\IZ)\cong\IQ$,
assumption \ref{BHM-A} is not satisfied by $G=\IQ$.
Recall that $B\IQ$ is a Moore space $M(\IQ,1)$, and therefore for any spectrum~$\WT{E}$ we have that $B\IQ\sma\WT{E}\simeq\Sigma\WT{E}_\IQ$, the suspension of the rationalization~$\WT{E}$.

Fixing a prime~$p$, the assembly map for~$\TC(\IS[\IQ];p)$ factors as
\begin{equation*}
\begin{tikzcd}[row sep=scriptsize]
\ds B\IQ_+\sma\TC(\IS;p)
\arrow[rr, "\alpha"]
\arrow[dr, shorten >=2em, near start, "t"']
&
&
\ds\TC(\IS[\IQ];p)
\\
&
\mathclap{\ds\holim_\RFcat\,B\IQ_+\sma\bigl(\THH(\IS)^{\Cp{m}}\bigr)\mathrlap{\,,}}
\arrow[ru, shorten <=2em]
\end{tikzcd}
\end{equation*}
where $t$ interchanges smash product and homotopy limit.
The assembly map for the trivial group~$1$ is a $\pi_*$-isomorphism and splits off from the assembly map for~$\TC(\IS[G];p)$ for any group~$G$.
So the interesting part of~$\alpha$ factors through
\begin{equation}
\label{eq:holim-BQ-sma-THH}
\holim_\RFcat\,B\IQ\sma\bigl(\THH(\spec{S})^{\Cp{m}}\bigr)
\,,
\end{equation}
since $B\IQ=(B\IQ_+)/(B1_+)$.
Using that $\THH(\IS)\simeq\IS$ (\autoref{THH-free-loop}) and the fundamental fibration sequence (\autoref{THH-Adams}), we conclude that the homotopy groups of the spectrum in~\eqref{eq:holim-BQ-sma-THH} vanish for every~$n\ge2$.
A similar argument applies to $\smallprod_p\TC(-;p)$.
Therefore, we see that diagram~\eqref{eq:main-diagram} cannot be used to prove rational injectivity of the assembly map in algebraic $K$-theory in this case.
\end{remark}

\begin{remark}
With only Assumption~\ref{our-A} we are not able to prove that the assembly map~$\assembly{4}$ for topological cyclic homology in diagram~\eqref{eq:main-diagram} is $\IQ$-injective, not even in the case of the trivial family~$\CF=1$.
Even treating one prime at the time, our method of proof breaks down without further assumptions on~$G$.
The reason for this is subtle, and is explained in detail in \autoref{no-R}, \autoref{no-R-revisited}, and \autoref{S_F}.
\end{remark}


\section{Conventions}

\subsection{Spaces}
\label{SPACES}
We denote by $\Top$ the category of compactly generated and weak Hausdorff spaces (e.g., see~\cite{tD-top}*{Section~7.9 on pages~186--195} and~\cite{Strickland}), which from now on we simply call \emph{spaces}.
A space homeomorphic to a (finite or countable) CW-complex is called a (finite or countable, respectively) \emph{CW-space}.
We denote by~$\CT$ the category of pointed spaces, and we let~$\CW$ be the full subcategory of~$\CT$ whose objects are the pointed countable CW-spaces.

\subsection{Equivariant spaces}
\label{G-SPACES}
Given a discrete group~$G$, we denote by $\Top^G$ the category of left $G$-spaces and $G$-equivariant maps.
Given $G$-spaces $X$ and~$Y$, the group~$G$ acts by conjugation on the space of maps from~$X$ to~$Y$, and the fixed points are the $G$-equivariant maps.
The pointed version of this category is denoted~$\CT^G$.

We denote by $\Sets^G$ the category of left $G$-sets and $G$-equivariant functions, and we consider $\Sets^G$ as a full subcategory of~$\Top^G$.
The \emph{orbit category}~$\Or G$ is the full subcategory of~$\Sets^G$ with objects~$G/H$ for all subgroups $H$ of~$G$.


Given a group homomorphism $\MOR{\alpha}{H}{G}$, the restriction functor
\[
\MOR{\res_\alpha}{\Top^G}{\Top^H}
\qquad\text{has a left adjoint}\qquad
\MOR{\ind_\alpha}{\Top^H}{\Top^G}
\]
which sends an $H$-space~$X$ to the $G$-space~$G\times_H X$ defined as the quotient of~$G\times X$ by the $H$-action $h(g,x)=(g\alpha(h^{-1}),hx)$.


\subsection{Action groupoids}
\label{ACTION-GROUPOID}
Given a group~$G$ and a $G$-set~$S$, we denote by~$\oid{G}{S}$ the groupoid with $\obj\oid{G}{S}=S$ and $\mor_{\oid{G}{S}}(s,s')=\SET{g\in G}{gs=s'}$.

Identifying the group~$G$ with the category~$\oid{G}{\pt}$ with only one object, and viewing the $G$-set~$S$ as a functor from $G$ to~$\Sets$, then $\oid{G}{S}$ is the Grothendieck construction of this functor; compare~\cite{Thomason}*{Definition~1.1 on page~92}.
This definition is functorial in~$S$ and thus yields a functor
\[
\MOR{\oid{G}{-}}{\Sets^G}{\Groupoids}
\,.
\]

\subsection{Enriched categories}
\label{ENRICHED}
We refer to \cite{MacLane}*{Section~VII.7 on pages~180--181} and~\cite{Kelly} for the theory of symmetric monoidal categories $(\CV,\odot,\Ione)$ and categories and functors enriched over~$\CV$, or simply $\CV$-categories and $\CV$-functors.
A monoid in a monoidal category~$\CV$ is the same as a $\CV$-category with only one object.
The important examples for us are summarized in the following table.

\begin{center}
\begin{tabular}{c|c|c}
$(\CV,\odot,\Ione)$&$\CV$-categories&monoids in~$\CV$\\\hline
$(\Sets,\times,\pt)$&ordinary categories&ordinary monoids\\
$(\Cat,\times,\pt)$&strict $2$-categories&strict monoidal categories\\
$(\Ab,\tensor,\IZ)$&pre-additive categories&rings\\
$(\CT,\sma,S^0)$&topological categories&pointed topological monoids\\
$(\Sigma\Sp,\sma,\spec{S})$&symmetric spectral categories&symmetric ring spectra
\end{tabular}
\end{center}
Given a lax monoidal functor~$\MOR{F}{\CV}{\CV'}$, any $\CV$-category~$\SX$ has an associated $\CV'$-category~$F\SX$ with the same objects and $(F\SX)(x,y)=F(\SX(x,y))$.
We use the term  base change for this passage from $\CV$- to $\CV'$-categories.

As indicated in the table, we refer to $\CT$-categories, i.e., categories enriched over pointed spaces, as {topological categories}; we refer to $\CT$-enriched functors as {continuous functors}.
Given a skeletally small topological category~$\CC$, we denote by $\CC\CT$ the category of continuous functors $\CC\TO\CT$ and continuous natural transformations.
Of great importance in this paper is the category~$\CW\CT$ of continuous functors to~$\CT$ from the full subcategory~$\CW$ of pointed countable CW-spaces defined in~\ref{SPACES}.
The objects of~$\CW\CT$ are called $\CW$-spaces and are models for spectra.
In fact, given $\WT{X}\in\obj\CW\CT$, for every $A,B\in\obj\CW$ there is a map $A\sma\WT{X}(B)\TO\WT{X}(A\sma B)$, and so in particular we get a symmetric spectrum by evaluating on spheres.
The category $\CW\CT$ is itself a closed symmetric monoidal category, with respect to the smash product of $\CW$-spaces $\WT{X}$ and~$\WT{Y}$ defined as the continuous left Kan extension along $\MOR{-\sma-}{\CW\times\CW}{\CW}$ of the composition
\[
\CW\times\CW\xrightarrow{\;\WT{X}\times\WT{Y}\;}
\CT\times\CT\xrightarrow{\;-\sma-\;}\CT
\,.
\]
Compare Subsections~\ref{KAN} and~\ref{SPECTRA}.

\subsection{Coends and left Kan extensions}
\label{KAN}
Given a small topological category~$\CC$ and continuous functors
\(
\MOR{W}{\CC^\op}{\CT}
\)
and
\(
\MOR{X}{\CC}{\CT}
\),
we let
\[
W\sma_\CC X
=
\coend\Bigl(\CC^\op\times\CC\xrightarrow{W\times X}\CT\times\CT\xrightarrow{-\sma-}\CT\Bigr)
\,.
\]
Given a continuous functor~$\MOR{\alpha}{\CC}{\CD}$, the continuous left Kan extension of~$X$ along~$\alpha$ is the continuous functor
\[
\MOR{\Lan_\alpha X}{\CD}{\CT}
\,,\qquad
d\longmapsto\CD\bigl(\alpha(-),d\bigr)\sma_\CC X
\,.
\]

If we replace $\CT$ by $\CW\CT$ or by one of the other topological categories of spectra in~\autoref{SPECTRA}, then coends and Kan extensions are defined levelwise.

\subsection{Homotopy colimits}
\label{HOCOLIM}
Given an ordinary small category~$\CC$, consider the functor
\[
\MOR{E\CC}{\CC^\op}{\Top}
\,,\qquad
c\longmapsto\real{N_\bullet(c\downarrow\CC)}
\]
that sends an object~$c$ to the classifying space of the under category~$c\downarrow\CC$.
The homotopy colimit of a functor $\MOR{X}{\CC}{\CT}$ is defined by
\[
\hocolim_\CC X=E\CC_+\sma_\CC X
\,.
\]
If $\MOR{W}{\CC^\op}{\Top}$ is a contravariant free $\CC$-CW-complex in the sense of~\cite{Davis-Lueck}*{Definition~3.2 on page~221} that is objectwise contractible, then $W\sma_\CC X$ is homotopy equivalent to~$\hocolim_\CC X$.

If we replace $\CT$ by $\CW\CT$ or by one of the other topological categories of spectra in~\autoref{SPECTRA}, then homotopy colimits are defined levelwise.

\subsection{Natural modules}
\label{NAT-MOD}
The category $\Mod_{\CT}$ of \emph{natural modules} in~$\CT$ is defined as follows; compare~\cite{Dundas}*{Section~A.10.4.1 on page~404}.
The objects are pairs~$(\CC,\spec{X})$, with $\CC$ an ordinary small category and $\MOR{\spec{X}}{\CC}{\CT}$ a functor.
A morphism in~$\Mod_{\CT}$ from $(\CC,\spec{X})$ to~$(\CD,\spec{Y})$ is a pair $(f,\nu)$, with $\MOR{f}{\CC}{\CD}$ a functor and $\MOR{\nu}{\spec{X}}{f^*\spec{Y}}$ a natural transformation.
Homotopy colimits define a functor
\[
\MOR{\hocolim}{\Mod_{\CT}}{\CT}
\,,\qquad
(\CC,\spec{X})\longmapsto\hocolim_\CC\spec{X}=E\CC_+\sma_\CC\spec{X}
\,.
\]
Analogously we define $\Mod_{\CW\CT}$ and get a functor
\begin{equation}
\label{eq:hocolim-WT}
\MOR{\hocolim}{\Mod_{\CW\CT}}{\CW\CT}
\,.
\end{equation}

\subsection{Assembly maps}
\label{ASSEMBLY}
As a special case of~\ref{KAN} and~\ref{HOCOLIM}, take $\CC=\Or G$ and $\CD=\Top^G$ for a discrete group~$G$, and consider them as topological categories by base change along the strong symmetric monoidal functor $\MOR{(-)_+}{\Top}{\CT}$.
Take $\MOR{\alpha}{\Or G}{\Top^G}$ to be the inclusion functor.
Let $\IN\Sp$ be the topological category of naive spectra from~\ref{SPECTRA}, and let $\MOR{\WT{E}}{\Or G}{\IN\Sp}$ be a functor.
Define $\WT{E}_\%=\Lan_\alpha\WT{E}$.
\[
\begin{tikzcd}[row sep=scriptsize]
\Or G\arrow[r, "\WT{E}"]\arrow[hook, d, "\alpha\,"']&\IN\Sp\\
\Top^G\arrow[ru, "\WT{E}_\%"']
\end{tikzcd}
\]
By the definition in~\ref{KAN}, for any $G$-space~$X$ the spectrum~$\WT{E}_\%(X)$ is the coend of the functor \((\Or G)^\op\times\Or G\TO\IN\Sp\) given by
\[
(G/H,G/K)\longmapsto\map(G/H,X)^G_+\sma\WT{E}(G/K)\cong X^H_+\sma\WT{E}(G/K)
\,.
\]
We write
\(
\WT{E}_\%(X)=X_+\sma_{\Or G}\WT{E}
\).
When $X=G/H$ there is a natural isomorphism
\begin{equation}
\label{eq:G/H}
G/H_+\sma_{\Or G}\WT{E}
\cong
\WT{E}(G/H)
\,.
\end{equation}
The \emph{assembly map} for~$\BE$ with respect to a family~$\CF$ of subgroups of~$G$ is the map of spectra
\begin{equation}
\label{eq:generic-assembly}
\MOR{\asbl}{EG(\CF)_+\sma_{\Or G}\BE}{\BE(G/G)}
\end{equation}
induced by the projection~$EG(\CF)\TO\pt=G/G$ and by the isomorphism~\eqref{eq:G/H}.
Here $EG(\CF)$ is the universal space for~$\CF$; see the beginning of \autoref{ASSUMPTIONS}.
Moreover,
\[
EG(\CF)_+\sma_{\Or G}\WT{E}
\cong
EG(\CF)_+\sma_{\Or G(\CF)}\WT{E}_{|\Or G(\CF)}
\,,
\]
where $\Or G(\CF)$ is the full subcategory of~$\Or G$ with objects $G/H$ for $H\in\CF$.
Since the functor $\MOR{\map(-,EG(\CF))^G}{\Or G(\CF)^\op}{\Top}$ is an objectwise contractible contravariant free $\Or G(\CF)$-CW-complex, the source of~\eqref{eq:generic-assembly} is homotopy equivalent to
\[
\hocolim\bigl(\Or G(\CF)\hookrightarrow\Or G\TO[\WT{E}]\IN\Sp\bigr)
\,.
\]

\subsection{Spectra}
\label{SPECTRA}
As motivated in \autoref{why-all-these-models} below, we need to use several different model categories of spectra:
\begin{itemize}[leftmargin=6.55em, labelsep=2.75em]
\item[$\IN\Sp$]
naive spectra (of topological spaces),
\item[$\Sigma\Sp$]
symmetric spectra (of topological spaces),
\item[$\Th\CT$]
orthogonal spectra,
\item[$\CW\CT$]
$\CW$-(pointed) spaces,
\item[$\Gamma\CS_*$]
$\Gamma$-(pointed) simplicial sets.
\end{itemize}
We refer to~\cite{MMSS} for definitions, details, and historical remarks.
Two comments are in order.

Our category~$\CW$, defined in~\ref{SPACES}, differs from~\cite{MMSS}*{Example~4.6 on page~454}, where only finite CW-spaces are considered.
However, it is remarked in loc.~cit.\ that the theory works equally well for countable rather than finite CW-spaces.
In \autoref{Adams} we use a result of Blumberg~\cite{Blumberg} for finite CW-spaces, and there we make that restriction explicit.

Moreover, our notation is slightly different from the one in~\cite{MMSS}.
In particular, we write $\Sp$ instead of~$\SS$, $\Gamma$ instead of~$\SF$ for Segal's category~$\Gamma^\op$ and,
following~\cite{RV}, $\Th$ instead of~$\SI$ for the topological category of real finite dimensional inner product spaces and Thom spaces of orthogonal complement bundles.

These categories are connected by forgetful functors
\begin{equation}
\label{eq:spectra}
\CW\CT
\TO[\forget]
\Th\CT
\TO[\forget]
\Sigma\Sp
\TO[\forget]
\IN\Sp
\,;
\end{equation}
compare~\cite{MMSS}*{Main Diagram on page~442}.
With the sole exception of~$\IN\Sp$, all other categories in~\eqref{eq:spectra} are closed symmetric monoidal with respect to the smash product, and the two functors on the left are lax symmetric monoidal.

For all categories in~\eqref{eq:spectra}, homotopy groups and hence $\pi_*$-isomorphisms are defined after forgetting down to~$\IN\Sp$.
All categories are Quillen model categories with respect to the stable model structure from~\cite{MMSS}*{Theorem~9.2 on page~471}.
With the only exception of~$\Sigma\Sp$, the weak equivalences are the $\pi_*$-isomorphisms.
In all cases, the fibrant objects are the ones whose underlying naive spectrum is an $\Omega$-spectrum in the classical sense~\cite{MMSS}*{Proposition~9.9(ii) on page~472}.
All categories in~\eqref{eq:spectra} are Quillen equivalent~\cite{MMSS}*{Theorem~0.1 on page~443}.

In all cases, limits, colimits, and homotopy colimits are defined levelwise.
However, homotopy limits have to be defined by first applying a fibrant replacement functor.

Finally, in \autoref{TRC} we need to compare the topological categories of spectra that we use with the category~$\Gamma\CS_*$, which is used in~\cite{Dundas}.
There are strong symmetric monoidal functors
\begin{equation}
\label{T-then-P}
\Gamma\CS_*
\TO[\IT]
\Gamma\CT
\TO[\IP]
\CW\CT
\,;
\end{equation}
see~\cite{MMSS}*{Section~19 on page~499 and Proposition~3.2 on page~450}.
We introduce the notation~$\MOR{\reall{-}}{\Gamma\CS_*}{\Sigma\Sp}$ for the composition of~\eqref{T-then-P} and the forgetful functor~$\MOR{\forget}{\CW\CT}{\Sigma\Sp}$, because we need it explicitly in~\autoref{TRC}.
Notice that the choice of using countable instead of finite CW-spaces in the definition of~$\CW$ does not affect the composition $\MOR{\forget\circ\IP}{\Gamma\CT}{\Sigma\Sp}$.

\begin{remark}
\label{why-all-these-models}
Notice that assembly maps cannot be defined in the stable homotopy category; the category~$\IN\Sp$ of naive spectra is sufficient to define and study them.
Symmetric spectra have the additional structure that is needed for the input of topological Hochschild homology and its variants.
The output naturally comes equipped with more structure, and $\THH$ takes values in~$\CW\CT$ and actually even in~$\CW\CT^{S^1}$; see~\ref{S1} and \autoref{THH}.
Striving to use the same category for both input and output is irrelevant for our purposes and would lead to superfluous complications.
Working in~$\CW\CT^{S^1}$ enables us to consider equivariant shifts as in \autoref{shift} and~\ref{S1} below, which we use in \autoref{HOFIB-R} to obtain a convenient point-set level model for the homotopy fiber of the restriction map between the fixed points of~$\THH$.
This is crucial for the proof of our splitting theorems; see \autoref{no-R-revisited}.
In \autoref{ADAMS} we use (equivariant) orthogonal spectra for the natural model of the Adams isomorphism between homotopy orbits and fixed points developed in~\cite{RV}.
Finally, in \autoref{TRC} we use constructions of~\cite{Dundas} that are performed using $\Gamma$-spaces in the simplicial setting, and so we need to compare their setup with ours.
\end{remark}

\subsection{Properties of symmetric spectra and symmetric spectral categories}
\label{SYMMETRIC}
As indicated in~\ref{ENRICHED}, we refer to $\Sigma\Sp$-categories as {symmetric spectral categories}.
Symmetric ring spectra are the monoids in~$\Sigma\Sp$, i.e, symmetric spectral categories with only one object.

\begin{definition}
Let $\spec{E}$ be a symmetric spectrum.
We say that $\spec{E}$ is:
\begin{enumerate}

\item\emph{well pointed}
if for every $x\in\IN$ the space~$\spec{E}_x$ is well pointed.

\item\emph{cofibrant}
if it is cofibrant in the stable model category structure on~$\Sigma\Sp$.

\item\emph{strictly connective}
if for every~$x\geq0$ the space $\spec{E}_x$ is $(x-1)$-connected; compare~\cite{MMSS}*{Definition 17.17 on page 496}.

\item\emph{convergent}
if there exists a non-decreasing function $\MOR{\lambda}{\IN}{\IZ}$ such that $\lim_{x\to\infty}\lambda(x)=\infty$ and the adjoint structure map $\spec{E}_x\TO\Omega\spec{E}_{x+1}$ is $(x+\lambda(x))$-connected for every~$x\geq0$.

\end{enumerate}
\end{definition}

\begin{definition}
Let $\spec{D}$ be a symmetric spectral category.
We say that $\spec{D}$ is:
\begin{enumerate}

\item\emph{objectwise well pointed}
if for every $c,d\in\obj\spec{D}$ the symmetric spectrum $\spec{D}(c,d)$ is well pointed.
Similarly we define objectwise cofibrant, objectwise strictly connective, and objectwise convergent.

\item\emph{very well pointed}
if it is objectwise well pointed and for every $d\in\obj\spec{D}$ the identity induces a cofibration $S^0\TO\spec{D}(d,d)(S^0)$ of the underlying unpointed spaces.
We say that a symmetric ring spectrum~$\spec{A}$ is very well pointed if it is so when considered as a symmetric spectral category with only one object.

\end{enumerate}
\end{definition}

\begin{definition}[{\cplus[*]}]
\label{cplus}
We say that a symmetric ring spectrum~$\spec{A}$ is \emph{\cplus} if it is strictly connective, convergent, and very well pointed.
\end{definition}

We emphasize that the last two conditions are very mild.
The last one simply means that $S^0\TO\spec{A}_0$ and $\pt\TO\spec{A}_n$ are cofibrations for all~$n>0$.
Any $(-1)$-connected $\Omega$-spectrum is strictly connective and convergent.
The sphere spectrum~$\spec{S}$ and suitable models for all Eilenberg-Mac Lane ring spectra~$\spec{H}R$ are \cplus.

\subsection{Equivariant spectra}
\label{G-SPECTRA}
Given a finite group~$G$, we consider naive $G$-spectra, i.e., objects of the categories in~\eqref{eq:spectra} equipped with a $G$-action in the categorical sense, and $G$-equivariant maps.
We write $\CW\CT^G$, $\Th\CT^G$, $\Sigma\Sp^G$ for the corresponding subcategories.
Notice that, as the notation suggests, $\CW\CT^G$ and $\Th\CT^G$ are the same as the categories of continuous functors from $\CW$ or~$\Th$ to~$\CT^G$.
A $G$-map $\MOR{f}{\WT{X}}{\WT{Y}}$ in any of these categories is called a $\pi_*^1$-isomorphism if it is a $\pi_*$-isomorphism of the underlying non-equivariant spectra.

When working with orthogonal spectra, though, naive is not naive, in the precise sense that the categories $\Th\CT^G$ and $\Th_G\CT_G$ are equivalent; see \citelist{\cite{HHR}*{Proposition~A.19 on page~147} \cite{MM}*{Theorem~1.5 on pages~75--76} \cite{RV}*{Theorem~4.12 on pages~1507--1508}}.
Here $\CT_G$ is the category with $\obj\CT_G=\obj\CT^G$ and morphism $G$-spaces of not necessarily equivariant maps equipped with the conjugation $G$-action.
This is a topological $G$-category, i.e., a $\CT^G$-enriched category.
The topological $G$-category $\Th_G$ is defined analogously, and we denote by $\Th_G\CT_G$ the corresponding category of continuous $G$-functors.
This is the category of non-naive orthogonal $G$-spectra from~\cite{MM}; see also~\cite{RV}*{Section~4, pages~1504--1507}.

Analogously, we write $\CW_G\CT_G$ for the category of non-naive $G$-equivariant $\CW$-spaces, studied in detail in~\cite{Blumberg}.
Here $\CW_G$ is the full subcategory of~$\CT_G$ whose objects are the pointed countable $G$-CW-spaces, where a (countable) $G$-CW-space is a $G$-space $G$-homeomorphic to a $G$-CW-complex (with only countably many $G$-cells).
The following definition plays an important role in our work.

\begin{definition}[Equivariant shift]
\label{shift}
Given $E\in\obj\CW_G$ and $\WT{X}\in\obj\CW\CT^G$, we define $\sh^E\WT{X}\in\obj\CW\CT^G$ as follows.
Let~$\MOR{\iota}{\CW}{\CW_G}$ be the inclusion of the trivial $G$-spaces.
Restriction along~$\iota$ gives a functor $\MOR{\iota^*}{\CW_G\CT_G}{\CW\CT^G}$.
In the other direction, given $\WT{X}\in\obj\CW\CT^G$, define its diagonal extension $\ext\WT{X}\in\obj\CW_G\CT_G$ by $(\ext\WT{X})(B)=\WT{X}(\res B)$, where $\res=\res_{1\to G}$ as in~\ref{G-SPACES}.
If $\MOR{\ell_g^B}{B}{B}$ denotes the action of~$g\in G$ on~$B$, then the action of~$g$ on~$(\ext\WT{X})(B)$ is defined as
\[
\ell_g^{(\ext\WT{X})(B)}
=
\ell_g^{\WT{X}(\res B)}\circ\WT{X}\bigl(\ell_g^B\bigr)
=
\WT{X}\bigl(\ell_g^B\bigr)\circ\ell_g^{\WT{X}(\res B)}
\,.
\]
For $\WT{Y}\in\obj\CW_G\CT_G$ and~$E\in\obj\CW_G$, define $\sh^E\WT{Y}\in\obj\CW_G\CT_G$ by $(\sh^E\WT{Y})(B)=\WT{Y}(E\sma B)$.
Finally, we put
\[
\sh^E\WT{X} = \iota^*\sh^E\ext\WT{X}
\,.
\]
In words: extend diagonally to non-naive $G$-equivariant $\CW$-spaces, shift, and then restrict back to naive $G$-equivariant $\CW$-spaces.
\end{definition}

\subsection{$S^1$-equivariant spaces and spectra}
\label{S1}
The circle group~$S^1$ acts on the geometric realization of cyclic spaces and spectra; compare~\ref{DELTA-LAMBDA-GOOD}.
For actions of the non-discrete group~$S^1$, we keep the notations and conventions introduced in~\ref{G-SPACES} and~\ref{G-SPECTRA}, e.g., we write $\CT^{S^1}$, $\CW\CT^{S^1}$, $\Sigma\Sp^{S^1}$, etc.

Now let $\WT{X}\in\obj\CW\CT^{S^1}$ and let $C$ be a finite subgroup of~$S^1$.
The fixed points~$\WT{X}^C$ are defined levelwise: $\WT{X}^C(A)=\WT{X}(A)^C$ for every~$A\in\obj\CW$.
Given a pointed $S^1$-space~$E$ such that $\res_{C\to S^1}\!E\in\obj\CW_C$, using \autoref{shift} we put
\[
\bigl(\sh^E\WT{X}\bigr)^C
=
\bigl(\sh^{\res_{C\to S^1}\!E}\res_{C\to S^1}\!\WT{X}\bigr)^C
\in\obj\CW\CT
\,.
\]
Notice that the assumption on $\res_{C\to S^1}\!E$ is satisfied for any subgroup~$C$ of~$S^1$ if~$E=ES^1_+=\IC P^\infty_{\,+}$.

\subsection{Simplicial and cyclic objects}
\label{DELTA-LAMBDA-GOOD}
Given a category~$\CC$, we write $\Delta^\op\CC$ and $\Lambda^\op\CC$ for the categories of simplicial and cyclic objects in~$\CC$, respectively.
We use the abbreviation $\CS_*=\Delta^\op\Sets_*$ for the category of pointed simplicial sets.

Recall that the geometric realization of a cyclic space carries a natural $S^1$-action, and so we get a functor $\MOR{\real{-}}{\Lambda^\op\Top}{\Top^{S^1}}$;
e.g., see~\citelist{\cite{DHK} \cite{Jones}*{Section~3 on pages~411--414} \cite{Drinfeld}}.

The following condition, introduced by Segal~\cite{Segal}*{Definition~A.4 on page~309}, becomes important when realizing simplicial spaces.

\begin{definition}
We say that a simplicial unpointed space~$X_\bullet\in\obj\Delta^\op\Top$ is \emph{Segal good} if for every $n\in\IN$ and every $i=0,\ldots,n-1$ the degeneracy map $\MOR{s_i}{X_{n-1}}{X_{n}}$ is a cofibration.
We call a simplicial or a cyclic pointed space Segal good if the underlying simplicial unpointed space is Segal good.
\end{definition}

\begin{lemma}
\label{real-lemma}
Let $X_\bullet, Y_\bullet\in\obj\Delta^\op\Top$ be simplicial spaces and let $\MOR{f_\bullet}{X_\bullet}{Y_\bullet}$ be a simplicial map.
Assume that $X_{\bullet}$ and $Y_{\bullet}$ are both Segal good.
\begin{enumerate}
\item
\label{i:real-hocolim}
There is a natural weak equivalence
\(
\hocolim_{\Delta^\op}X_\bullet\TO\real{X_\bullet}
\,.
\)
\item
\label{i:real-we}
If $f_{\bullet}$ is a weak equivalence in each simplicial degree, then $|f_{\bullet}|$ is a weak equivalence.
\item
\label{i:real-cofib}
If $f_\bullet$ is a cofibration in each simplicial degree, then $\real{f_\bullet}$ is a cofibration.
\end{enumerate}
\end{lemma}

\begin{proof}
\ref{i:real-we} follows at once from~\ref{i:real-hocolim}.

\ref{i:real-hocolim} and~\ref{i:real-cofib} follow from more general model-theoretic statements, which we now briefly recall.
A simplicial space~$X_\bullet$ is called {Reedy cofibrant} if $L_nX_\bullet\subseteq X_n$ is a cofibration for each~$n$, where
\(
L_nX_\bullet=\bigcup_{0\leq i\leq n-1}s_i(X_{n-1})
\,.
\)
(Slightly stronger versions of the Reedy cofibrant and Segal good conditions are called proper and strictly proper in~\cite{May}*{Definition~11.2 on page~102}.)
A simplicial map~$f_\bullet$ is called a {Reedy cofibration} if $X_n\cup_{L_nX_\bullet}L_nY_\bullet\TO Y_n$ is a cofibration for each~$n$.
By~\cite{Lewis}*{proof of Corollary~2.4(b) on page~153}, Segal good implies Reedy cofibrant.
If $X_\bullet$ is Reedy cofibrant, then \ref{i:real-hocolim}~is proved in~\cite{Hirschhorn}*{Theorem~19.8.7(i) on page~427}.
Furthermore, if $X_\bullet$ and~$Y_\bullet$ are Reedy cofibrant and $f_\bullet$ is a cofibration in each simplicial degree, then $f_\bullet$ is a Reedy cofibration.
This follows from Lillig's union theorem for cofibrations~\cite{tD-top}*{Theorem~5.4.5 on page~114} and the fact that a pushout square in~$\Top$ along a closed embedding is a pullback square~\cite{Strickland}*{Proposition~2.35 on page~9}.
Then \ref{i:real-cofib} is a special case of the fact that geometric realization sends Reedy cofibrations between Reedy cofibrant objects to cofibrations; see~\cite{Goerss-Jardine}*{Proposition~VII.3.6 on page~374}.
\end{proof}

For simplicial or cyclic objects in one of our categories of spectra we define Segal goodness levelwise.
Since homotopy colimits are defined levelwise, we have the following immediate consequence of \autoref{real-lemma}.
We formulate it explicitly for~$\CW\CT$, but it holds verbatim for all other categories of spectra in~\eqref{eq:spectra}.

\begin{corollary}
\label{real-lemma-WT}
If $\WT{X}_\bullet\in\obj\Delta^\op(\CW\CT)$ is levelwise Segal good then there is a natural level equivalence in~$\CW\CT$
\[
\hocolim_{\Delta^\op}\WT{X}_\bullet
\TO
\real{\WT{X}_\bullet}
\,.
\]
If $\WT{X}_\bullet,\WT{Y}_\bullet\in\obj\Delta^\op(\CW\CT)$ are both levelwise Segal good and $\MOR{f_\bullet}{\WT{X}_\bullet}{\WT{Y}_\bullet}$ is a $\pi_*$-isomorphism in each simplicial degree, then $\real{f_\bullet}$ is a $\pi_*$-isomorphism.
\end{corollary}

\begin{lemma}
\label{real-hofib}
Let
\begin{equation}
\label{eq:fib-seq}
\WT{X}_\bullet\TO[f_\bullet]\WT{Y}_\bullet\TO[g_\bullet]\WT{Z}_\bullet
\end{equation}
be a sequence of maps in~$\Delta^\op(\CW\CT)$.
Assume that $\WT{X}_\bullet$, $\WT{Y}_\bullet$, and $\WT{Z}_\bullet$ are levelwise Segal good, and assume that \eqref{eq:fib-seq} is in every simplicial degree a stable homotopy fibration sequence in~$\CW\CT$.
Then
\[
\real*{\WT{X}_\bullet}\TO[\real{f_\bullet}]\real*{\WT{Y}_\bullet}\TO[\real{g_\bullet}]\real*{\WT{Z}_\bullet}
\]
is a stable homotopy fibration sequence in~$\CW\CT$.
\end{lemma}

\begin{proof}
This follows from \autoref{real-lemma-WT}, using that stable homotopy fibration sequences and stable homotopy cofibration sequences are the same in~$\CW\CT$, that $\hocofib f_\bullet$ is levelwise Segal good if $\WT{X}_\bullet$ and $\WT{Y}_\bullet$ are levelwise Segal good, and that geometric realization commutes with homotopy cofibers.
\end{proof}

\subsection{$p$-completion of spectra}
\label{P-COMPL}
For a prime~$p$,
we define the $p$-completion~$\WT{X}\pcompl$ of a naive spectrum~$\WT{X}$ as the Bousfield localization of~$\WT{X}$ with respect to the Moore spectrum of~$\IZ/p\IZ$, in the sense of~\cite{Bousfield}.
This construction is functorial in the stable homotopy category; it can be made functorial in the symmetric monoidal categories of spectra above, but we do not need it for our purposes.
By~\cite{Bousfield}*{Proposition~2.5 on page~262}, for any $n\in\IZ$ there is a splittable short exact sequence
\[
0
\TO
\Ext\bigl(\IZ[\tfrac{1}{p}]/\IZ,\pi_n(\WT{X})\bigr)
\TO
\pi_n\bigl(\WT{X}\pcompl\bigr)
\TO
\Hom\bigl(\IZ[\tfrac{1}{p}]/\IZ,\pi_{n-1}(\WT{X})\bigr)
\TO
0
\,.
\]
If the $p$-torsion elements in $\pi_{n}(\WT{X})$ and~$\pi_{n-1}(\WT{X})$ have bounded order, then
\begin{align*}
\Ext\bigl(&\IZ[\tfrac{1}{p}]/\IZ,\pi_n(\WT{X})\bigr)
\cong
\pi_{n}(\WT{X})\pcompl
=
\lim_{i\in\IN}\,\pi_{n}(\WT{X})/p^i
\,,
\\
\Hom\bigl(&\IZ[\tfrac{1}{p}]/\IZ,\pi_{n-1}(\WT{X})\bigr)
\cong
0
\,;
\end{align*}
see~\cite{Bousfield-Kan}*{Section~IV.2.1 on pages~166--167}.
In particular, if the abelian groups $\pi_*(\WT{X})$ are finitely generated, then 
\(
\pi_*(\WT{X}\pcompl)
\cong
\pi_*(\WT{X})\tensor_\IZ\IZ_p
\).
By~\cite{Bousfield}*{Lemma~1.10 on page~259}, $p$-completion preserves stable homotopy fibration sequences of spectra.


\section{Cyclic nerves}
\label{CN}

In this section we recall the definition of the cyclic nerve of categories, study the corresponding assembly map, and then prove our first splitting result, Proposition~\ref{split-CN}.
This proposition serves as model for all our subsequent splitting theorems.

\begin{definition}[Cyclic nerve]
\label{cyclic-nerve}
The cyclic nerve of a small category~$\CC$ is the cyclic set~$CN_\bullet(\CC)$ with $q$-simplices
\begin{equation*}
CN_q(\CC)
=
\smallcoprod_{\substack{c_0,c_1,\dotsc,c_q\\\text{in}\,\obj\CC}}
\CC(c_0, c_q) \times \CC(c_1, c_0) \times \dotsb \times \CC(c_q , c_{q-1})
\,.
\end{equation*}
The face maps are given by composition (and~$d_q$ also uses cyclic permutation), the degeneracies by insertion of identities, and the cyclic structure by cyclic permutations.
\end{definition}

In the case of monoids, this construction appears in~\cite{Waldhausen-A2}*{(2.3) on page~368}, where the notation $N^{\operatorname{cy}}$ is used.

\autoref{cyclic-nerve} produces a functor
\[
\MOR{CN_\bullet}{\Cat}{\Lambda^\op\Sets}
\]
from small categories to cyclic sets.
Natural isomorphisms of functors induce cyclic homotopies, and so cyclic nerves of equivalent categories are cyclically homotopy equivalent.
Recall from \autoref{DELTA-LAMBDA-GOOD} that the geometric realization of a cyclic set has a natural $S^1$-action.

\begin{example}[Free loop space]
\label{free-loop}
If~$M$ is a monoid, seen as a category with only one object, then there is an $S^1$-equivariant map
\begin{equation}
\label{eq:free-loop}
\real{CN_\bullet(M)}\TO\CL BM=\map(S^1,BM)
\,,
\end{equation}
where $BM$ is the classifying space of $M$ and $\CL BM$ is the free loop space of~$BM$.
If~$M$ is group-like, i.e., if $\pi_0(M)$ is a group, then this map induces a weak equivalence between the fixed points
\[
\real{CN_\bullet(M)}^C \TO[\simeq] (\CL BM)^C
\]
of any finite subgroup~$C$ of~$S^1$.
However, the induced map on $S^1$-fixed points is not an equivalence in general.
For more details, see~\citelist{\cite{Goodwillie}*{proof of Lema~V.1.3 on pages~209--211} \cite{Jones}*{Theorem~6.2 on page~420}}.
\end{example}

Given a group~$G$, we consider the functor given by the composition
\[
\Or G
\xrightarrow{\oid{G}{-}}
\Groupoids
\xrightarrow{CN_\bullet}
\Lambda^\op\Sets
\xrightarrow{\real{-}_+}
\CT^{S^1}
\]
and the associated assembly map of pointed $S^1$-spaces
\[
EG(\CF)_+\sma_{\Or G}\real{CN_\bullet(\oid{G}{-})}_+
\TO
\real{CN_\bullet(G)}_+
\,.
\]
The key to study this assembly map is a certain retraction that we proceed to define.

Let $S$ be a $G$-set, and consider its action groupoid~$\oid{G}{S}$ from \autoref{ACTION-GROUPOID}.
Let $\conj G$ be the set of conjugacy classes~$[g]$ of elements $g\in G$.
We denote by $(g_0 , g_1 , \dotsc , g_q , s)$ the element in $CN_q(\oid{G}{S})$ consisting of the $(q+1)$-tuple of morphisms $(g_0, g_1, \dotsc, g_q)$ in the summand indexed by $(s_0, s_1, \dotsc, s_q)$, where $s_q = s$ and $s_i = g_{i+1} \dotsb g_q s$ for $0\le i < q$.
\[
\begin{tikzcd}
s_0
\arrow[rrrr, rounded corners,
       to path={    (\tikztostart.south)
                 |- +(3,-.3) [at end]\tikztonodes
                 -| (\tikztotarget.south)},
      "g_0"']
&
s_1
\arrow[l, "g_1"']
&
\dotsb
\arrow[l, "g_2"']
&
s_{q-1}
\arrow[l, "g_{q-1}"']
&
s_q
\arrow[l, "g_q"']
\end{tikzcd}
\]
Sending $(g_0, \dotsc , g_q, s)$ to~$[g_0 g_1 \dotsb g_q ] \in \conj G$ defines a map
\begin{equation}
\label{eq:maptoconG}
CN_{\bullet}(\oid{G}{S}) \TO \conj G
\,.
\end{equation}
If we consider $\conj G$ as a constant cyclic set, then~\eqref{eq:maptoconG} is a map of cyclic sets, and therefore it induces a decomposition
\begin{equation}
\label{eq:decomposition-CN}
CN_{\bullet}(\oid{G}{S})
=
\smallcoprod_{[c]\in\conj G}
CN_{\bullet,[c]}(\oid{G}{S})
\,,
\end{equation}
where $CN_{\bullet,[c]}(\oid{G}{S})$ denotes the preimage of~$[c]$ under the map~\eqref{eq:maptoconG}.
Similarly, given a family~$\CF$ of subgroups of~$G$, we denote by~$CN_{\bullet,\CF}(\oid{G}{S})$ the preimage of
\[
\conj_{\CF} G = \SET{[c]}{ \langle c \rangle \in \CF } \subseteq \conj G
\,,
\]
and we call $CN_{\bullet,\CF}(\oid{G}{S})$ the $\CF$-part of~$CN_{\bullet}(\oid{G}{S})$.
Here and elsewhere $\langle c \rangle$ denotes the cyclic subgroup of~$G$ generated by~$c$.

There are cyclic maps
\begin{equation} \label{eq:retraction-CN}
\begin{tikzcd}
CN_{\bullet}    (\oid{G}{S})_+
\arrow[r, shift left, "\pr_{\CF}"]
&
CN_{\bullet,\CF}(\oid{G}{S})_+
\arrow[l, shift left, "\ii_{\CF}"]
\,,
\end{tikzcd}
\end{equation}
where $\ii_{\CF}$ is the inclusion, and the projection $\pr_{\CF}$ sends the components with $\langle c \rangle \notin \CF$ to the disjoint basepoint and is the identity on the components with $\langle c \rangle \in \CF$.
They clearly satisfy $\pr_{\CF} \circ \ii_{\CF} = \id$.

Notice that the decomposition~\eqref{eq:decomposition-CN} and the maps~\eqref{eq:retraction-CN} are natural in~$S$.

The following lemma gives an explicit computation of the cyclic nerve of action groupoids, and is the basis for all our splitting results for assembly maps.
We need to introduce some notation.
We put $E_\bullet G=N_\bullet(\oid{G}{\,(G/1)})$, the nerve of the category with set of objects~$G$ and precisely one morphism between any two objects.
The geometric realization $EG=\real{E_\bullet G}$ is a functorial model for the universal contractible free $G$-space.
Recall that $Z_G H$ denotes the centralizer of~$H$ in~$G$.

\begin{lemma}
\label{computation-CN}
Choose a representative~$c$ in each conjugacy class~$[c]\in\conj G$.
Then for any $G$-set~$S$ there is a simplicial homotopy equivalence
\begin{equation}
\label{eq:computation-CN}
\smallcoprod_{[c]\in\conj G}
E_\bullet Z_G\langle c \rangle \timesd_{Z_G\langle c \rangle} \map(G/\langle c \rangle,S)^G
\TO
CN_{\bullet}(\oid{G}{S})
\,.
\end{equation}
The map~\eqref{eq:computation-CN} is natural in~$S$, and it is compatible with the decomposition~\eqref{eq:decomposition-CN} of the target.
\end{lemma}

\begin{proof}
See~\cite{LR-cyclic}*{Proposition~9.9 on page~627}.
\end{proof}

\begin{remark}
Notice that the source of the map~\eqref{eq:computation-CN} does not admit any obvious cyclic structure.
\end{remark}

\begin{proposition}
\label{computation-X-sma-CN}
For every $G$-space~$X$ there is a natural commutative diagram
\begin{equation}
\label{eq:computation-X-sma-CN}
\begin{tikzcd}[row sep=small]
\ds
\hphantom{{}_{\CF}}
{\bigvee_{[c]\in\conj G}}
\Bigl(
EZ_G\langle c \rangle \timesd_{Z_G\langle c \rangle} X^{\langle c \rangle}
\Bigr)_+
\arrow[r]
\arrow[d, shorten <=-1.5ex, shorten >=-.5ex]
&
\ds X_+ \sma_{\Or G}
{\real{CN_{\bullet}(\oid{G}{-})}_+}
\hphantom{{}_{,\CF}}
\arrow[d, shift right=2.7em, "\id\sma\pr_\CF"]
\\
\ds
{\bigvee_{[c]\in\conj_\CF G}}
\Bigl(
EZ_G\langle c \rangle \timesd_{Z_G\langle c \rangle} X^{\langle c \rangle}
\Bigr)_+
\arrow[r]
&
\ds X_+ \sma_{\Or G}
{\real{CN_{\bullet,\CF}(\oid{G}{-})}_+}
\end{tikzcd}
\end{equation}
where the left-hand vertical map is the projection to the wedge summands with $\langle c \rangle \in \CF$, and both horizontal maps are homotopy equivalences.
\end{proposition}

\begin{proof}
The simplicial homotopy equivalence from \autoref{computation-CN} induces the horizontal homotopy equivalence in the following display.
\[
\begin{tikzcd}[column sep=small, row sep=scriptsize]
\ds X\timesd_{\Or G} \real*{\smallcoprod_{[c]\in\conj G} E_\bullet Z_G\langle c \rangle \timesd_{Z_G\langle c \rangle} \map(G/\langle c \rangle,-)^G}
\arrow[r, "\simeq"]
&
\ds X\timesd_{\Or G}\real{CN_{\bullet}(\oid{G}{-})}
\\
\ds X\timesd_{\Or G}
\smallcoprod_{[c]\in\conj G}
\Bigl(
EZ_G\langle c \rangle \timesd_{Z_G\langle c \rangle} \map(G/\langle c \rangle,-)^G
\Bigr)
\arrow[u, "\ts\cong"', "\ts\one"]
\\
\ds
\smallcoprod_{[c]\in\conj G}
EZ_G\langle c \rangle \timesd_{Z_G\langle c \rangle}
\Bigl(
X\timesd_{\Or G}\map(G/\langle c \rangle,-)^G
\Bigr)
\arrow[u, "\ts\cong"', "\ts\two"]
\arrow[phantom, r, "\cong"]
\arrow[draw=none, r, shift left, "\ts\three"]
&
\ds
\smallcoprod_{[c]\in\conj G}
EZ_G\langle c \rangle \timesd_{Z_G\langle c \rangle} X^{\langle c \rangle}
\end{tikzcd}
\]
The isomorphism~$\one$ comes from the fact that geometric realization commutes with colimits and finite products, $\two$ comes from associativity and distributivity properties, and $\three$ comes from the Yoneda lemma.
Adding disjoint basepoints gives the top horizontal homotopy equivalence in diagram~\eqref{eq:computation-X-sma-CN}.
Since the vertical maps in~\eqref{eq:computation-X-sma-CN} are retractions, also the bottom horizontal map is a homotopy equivalence.
\end{proof}

We are now ready to prove the following result, which is the model case of all our splitting and isomorphisms results for assembly maps.
Notice that, using the equivalence~\eqref{eq:free-loop}, this result also gives a computation of the free loop space~$\CL BG$.

\begin{proposition}
\label{split-CN}
Let $G$ be any group.
Consider the following commutative diagram of pointed $S^1$-spaces.
\begin{equation}
\label{eq:split-CN}
\begin{tikzcd}[column sep=large]
\ds EG(\CF)_+\sma_{\Or G}\real{CN_{\bullet}(\oid{G}{-})}_+
\arrow[r, "\asbl"]
\arrow[d, "\id\sma\pr_\CF"']
&
\ds\real{CN_{\bullet}(G)}_+
\arrow[d, "\pr_\CF"]
\\
\ds EG(\CF)_+\sma_{\Or G}\real{CN_{\bullet,\CF}(\oid{G}{-})}_+
\arrow[r, "\asbl"']
&
\ds\real{CN_{\bullet,\CF}(G)}_+
\end{tikzcd}
\end{equation}
After forgetting the $S^1$-action, the left-hand map and the bottom map are homotopy equivalences.
If $\CF$ contains all cyclic groups, then the right-hand map is an isomorphism.
\end{proposition}

\begin{proof}
Apply \autoref{computation-X-sma-CN} to~$X=EG(\CF)$.
Since $EG(\CF)^{\langle c \rangle}=\emptyset$ if and only if $\langle c \rangle\notin\CF$,
the vertical map on the left in diagram~\eqref{eq:split-CN} corresponds under the homotopy equivalences of \autoref{computation-X-sma-CN} to the identity.

Recall that the horizontal assembly maps are the maps induced by the projection $EG(\CF)\TO\pt$ and the isomorphism $\pt_+\sma_{\Or G}\real{CN_{\bullet,\CF}(\oid{G}{-})}_+\cong\real{CN_{\bullet,\CF}(G)}_+$.
If $\langle c \rangle\in\CF$, the equivariant map
\(
EZ_G\langle c \rangle\times EG(\CF)^{\langle c \rangle}
\TO
EZ_G\langle c \rangle\times\pt
\)
is a non-equivariant homotopy equivalence of free $Z_G\langle c \rangle$-CW-spaces, and hence it remains a homotopy equivalence after taking quotients of the $Z_G\langle c \rangle$-actions.
Using \autoref{computation-X-sma-CN}, this shows that the horizontal map at the bottom of diagram~\eqref{eq:split-CN} is a homotopy equivalence.
Finally, the last statement is clear from the definition of the projection~$\pr_\CF$.
\end{proof}

\section{Topological Hochschild homology}
\label{THH}

In this section we prove the splitting theorem for assembly maps for topological Hochschild homology, \autoref{THH-and-C}\ref{i:THH}.
More precisely, we establish the following result.

\begin{theorem}
\label{split-THH}
Let $G$ be any group and $\spec{A}$ any very well pointed symmetric ring spectrum.
Consider the following commutative diagram in~$\CW\CT^{S^1}$.
\[
\begin{tikzcd}[column sep=large]
\ds EG(\CF)_+\sma_{\Or G}\THH    (\spec{A}[\oid{G}{-}])
\arrow[r, "\asbl"]
\arrow[d, "\id\sma\pr_\CF"']
&
\THH    (\spec{A}[G])
\arrow[d, "\pr_\CF"]
\\
\ds EG(\CF)_+\sma_{\Or G}\THH_\CF(\spec{A}[\oid{G}{-}])
\arrow[r, "\asbl"']
&
\THH_\CF(\spec{A}[G])
\end{tikzcd}
\]
The left-hand map and the bottom map are $\pi_*^1$-isomorphisms.
If $\CF$ contains all cyclic groups, then the right-hand map is an isomorphism.
\end{theorem}

We begin by recalling the definition of topological Hochschild homology, following~\cite{B-THH} and~\cite{DMcC-THH}.

Let $\CI$ be the category with objects the finite sets $x=\{1,2,\dotsc,x\}$ for all~$x\ge0$, and morphisms all the injective functions.
Ordered concatenation $(x,y)\longmapsto x\sqcup y\cong x+y$ defines a strict cocartesian monoidal structure on~$\CI$, symmetric up to natural isomorphism.

Let $\spec{D}$ be a symmetric spectral category; compare~\autoref{SYMMETRIC}.
Let $q\in\IN$.
In order to simplify the notation below, we write $\CI^{[q]}$ for~$\CI^{q+1}$.
Define $cn_{[q]}\spec{D}$ to be the functor
\begin{align*}
\MOR{cn_{[q]}\spec{D}}{\CI^{[q]}&}{\ \CT}
\\
(x_0,\dotsc,x_q)
&\longmapsto
\quad\bigvee_{\mathclap{\substack{d_0,\dotsc,d_q\\\text{in}\,\obj\spec{D}}}}\quad
\spec{D}(d_0,d_q)_{x_0}\sma
\spec{D}(d_1,d_0)_{x_1}\sma
\dotsb\sma
\spec{D}(d_q,d_{q-1})_{x_q}
\,.
\end{align*}
In the special case when $\spec{D}$ is the sphere spectrum~$\spec{S}$, notice that $cn_{[q]}\spec{S}(x_0,\dotsc,x_q)=S^{x_0}\sma\dotsb\sma S^{x_q}$.
Then define
\begin{align*}
\MOR{M_{[q]}\spec{D}}{\CI^{[q]}}{{}&\CW\CT}
\\
\vec{x}=(x_0,\dotsc,x_q)
\longmapsto{}
&
\map\bigl(cn_{[q]}\spec{S}(\vec{x}),\;-\sma cn_{[q]}\spec{D}(\vec{x})\bigr)
\\
&=
\map\bigl(S^{x_0}\sma\dotsb\sma S^{x_q},\;-\sma cn_{[q]}\spec{D}(x_0,\dotsc,x_q)\bigr)
,
\end{align*}
where $-$ denotes the variable in~$\CW$.

The key observation for the definition of topological Hochschild homology is that the functors $M_{[q]}\spec{D}$ assemble to a cyclic object in the category of natural modules in~$\CW\CT$ defined in \autoref{NAT-MOD}:
\begin{equation}
\label{eq:I-M}
\MOR{\bigl(\CI^{[\bullet]},M_{[\bullet]}\spec{D}\bigr)}{\Lambda^\op}{\Mod_{\CW\CT}}
\,,\qquad
[q]\longmapsto\bigl(\CI^{[q]},M_{[q]}\spec{D}\bigr)
\,.
\end{equation}
This is explained in detail in~\cite{Dundas}*{Sections~4.2.2.3--4 on pages~151--153}.

\begin{definition}[Topological Hochschild homology]
Let $\spec{D}$ be a symmetric spectral category.
Then $THH_\bullet(\spec{D})\in\obj\Lambda^\op(\CW\CT)$ is defined as the composition of \eqref{eq:I-M} and the functor $\MOR{\hocolim}{\Mod_{\CW\CT}}{\CW\CT}$ in~\eqref{eq:hocolim-WT}:
\[
THH_q(\spec{D})=\hocolim_{\CI^{[q]}}M_{[q]}\spec{D}
\,.
\]
This defines a functor
\[
\MOR{THH_\bullet(-)}{\Sigma\Sp\D\Cat}{\Lambda^\op(\CW\CT)}
\,,
\]
and its composition with the geometric realization functor $\MOR{\real{-}}{\Lambda^\op\CT}{\CT^{S^1}}$ gives
\[
\MOR{\THH(-)=\real{THH_\bullet(-)}}{\Sigma\Sp\D\Cat}{\CW\CT^{S^1}}
.
\]
\end{definition}

Now let $\spec{A}$ be a symmetric ring spectrum.
Any ordinary category~$\CC$ defines a symmetric spectral category~$\spec{A}[\CC]$ by base change along the lax symmetric monoidal functor
\[
\MOR{\spec{A}\sma-_+}{\Sets}{\Sigma\Sp}
\,,\qquad
X\longmapsto\spec{A}\sma X_+
\,;
\]
compare \autoref{ENRICHED}.

\begin{lemma}
\label{cn(A[C])}
Let $\spec{A}$ be a symmetric ring spectrum, and let $\CC$ be a small category.
Then there is a natural isomorphism of functors $\CI^{[q]}\TO\CT$
\begin{equation}
\label{eq:cn(A[C])}
cn_{[q]}(\spec{A}[\CC])(?)
\cong
cn_{[q]}\spec{A}(?)\sma CN_q(\CC)_+
\end{equation}
and a natural isomorphism of functors $\Lambda^\op\TO\Mod_{\CW\CT}$
\begin{equation}
\label{eq:M(A[C])}
\bigl(\CI^{[\bullet]},M_{[\bullet]}(\spec{A}[\CC])\bigr)
\cong
\bigl(\CI^{[\bullet]},
\map\bigl(cn_{[\bullet]}\spec{S}(?),\;-\sma cn_{[\bullet]}\spec{A}(?)\sma CN_\bullet(\CC)_+\bigr)
\bigr)
\,,
\end{equation}
where $?$ denotes the variable in~$\CI^{[\bullet]}$ and $-$ the variable in~$\CW$.
\end{lemma}

\begin{proof}
Evaluated at $?=(x_0,\dotsc,x_q)\in\obj\CI^{[q]}$, the isomorphism~\eqref{eq:cn(A[C])} is given by
\begin{multline*}
\quad\bigvee_{\mathclap{c_0,\dotsc,c_q}}\quad
\spec{A}_{x_0}\sma\CC(c_0,c_q)_+\sma
\spec{A}_{x_1}\sma\CC(c_1,c_0)_+\sma
\dotsb\sma
\spec{A}_{x_q}\sma\CC(c_q,c_{q-1})_+
\cong
\\
\cong
\bigl(\spec{A}_{x_0}\sma\spec{A}_{x_1}\sma\dotsb\sma\spec{A}_{x_q}\bigr)
\sma
\Bigl(\:
\smallcoprod_{c_0,\dotsc,c_q}\!
\CC(c_0,c_q)\times\CC(c_1,c_0)\times\dotsb\times\CC(c_q,c_{q-1})
\,\Bigr)_+
\,.
\end{multline*}
The rest is then clear.
\end{proof}

Now we specialize to~$\CC=\oid{G}{S}$ for a group~$G$ and a $G$-set~$S$.
Using the decomposition~\eqref{eq:decomposition-CN} of the cyclic nerve of action groupoids, from~\eqref{eq:cn(A[C])} we obtain a decomposition
\begin{equation*}
cn_{[\bullet]}(\spec{A}[\oid{G}{S}])
\cong
\bigvee_{[c]\in\conj G}
cn_{[\bullet],[c]}(\spec{A}[\oid{G}{S}])
\,,
\end{equation*}
where we define
\[
cn_{[\bullet],[c]}(\spec{A}[\oid{G}{S}])
=
\bigl(cn_{[\bullet]}\spec{A}\bigr)\sma CN_{\bullet,[c]}(\oid{G}{S})_+
\,.
\]
Similarly, given a family~$\CF$ of subgroups of~$G$, we define~$cn_{[\bullet],\CF}(\spec{A}[\oid{G}{S}])$.
There are maps
\begin{equation}
\label{eq:retraction-cn}
\begin{tikzcd}
cn_{[\bullet]}    (\spec{A}[\oid{G}{S}])
\arrow[r, shift left, "\pr_{\CF}"]
&
cn_{[\bullet],\CF}(\spec{A}[\oid{G}{S}])
\arrow[l, shift left, "\ii_{\CF}"]
\end{tikzcd}
\end{equation}
satisfying $\pr_{\CF} \circ \ii_{\CF} = \id$.
Then we define
\begin{align*}
\MOR{M_{[q],\CF}(\spec{A}[\oid{G}{S}])}{\CI^{[q]}}{{}&\CW\CT}
\\
\vec{x}
\
\longmapsto{}
&
\map\bigl(cn_{[q]}\spec{S}(\vec{x}),\;-\sma cn_{[q]}\spec{A}(\vec{x})\sma CN_{q,\CF}(\oid{G}{S})_+\bigr)
.
\end{align*}
As $[q]$ varies, these functors assemble to a cyclic object
\[
\MOR{\bigl(\CI^{[\bullet]},M_{[\bullet],\CF}(\spec{A}[\oid{G}{S}])\bigr)}{\Lambda^\op}{\Mod_{\CW\CT}}
\,.
\]
\begin{definition}
The $\CF$-parts of topological Hochschild homology are defined by
\begin{align*}
THH_{q,\CF}(\spec{A}[\oid{G}{S}])&=\hocolim_{\CI^{[q]}}M_{[q],\CF}(\spec{A}[\oid{G}{S}])
\\
\shortintertext{and}
\THH_\CF(\spec{A}[\oid{G}{S}])&=\real{THH_{\bullet,\CF}(\spec{A}[\oid{G}{S}])}
\,.
\end{align*}
The maps~\eqref{eq:retraction-cn} induce maps
\begin{equation}
\label{eq:retraction-THH}
\begin{tikzcd}
\THH    (\spec{A}[\oid{G}{S}])
\arrow[r, shift left, "\pr_{\CF}"]
&
\THH_\CF(\spec{A}[\oid{G}{S}])
\arrow[l, shift left, "\ii_{\CF}"]
\end{tikzcd}
\end{equation}
satisfying $\pr_{\CF} \circ \ii_{\CF} = \id$.
\end{definition}

\autoref{split-THH} follows from \autoref{split-CN} using the following result, which in the case of discrete rings is due to Hesselholt and Madsen~\cite{HM-top}*{Theorem~7.1 on page~81}; see also~\cite{H-survey}*{Proposition~3 on page~81}.
In fact, they additionally prove that the map~$\hm$ in~\eqref{eq:THH-CN} is a $\pi_*^C$-isomorphism for every finite cyclic subgroup $C$ of~$S^1$.
Even though this stronger equivariant statement can also be generalized here, we restrict to the case $C=1$, since this is all we need.

\begin{theorem}
\label{THH-CN}
Let $\spec{A}$ be a very well pointed symmetric ring spectrum, and let $\CC$ be a small category.
Then there is a natural $\pi_*^1$-isomorphism in $\CW\CT^{S^1}$
\begin{equation}
\label{eq:THH-CN}
\MOR{\hm}
{\THH(\IA)\sma\real{CN_\bullet(\CC)}_+}
{\THH(\spec{A}[\CC])}
\,.
\end{equation}
If $\CC=\oid{G}{S}$, then $\hm$ restricts to a $\pi_*^1$-isomorphism
\begin{equation}
\label{eq:THH-CN-F-part}
\MOR{\hm_\CF}
{\THH(\IA)\sma\real{CN_{\bullet,\CF}(\oid{G}{S})}_+}
{\THH_\CF(\spec{A}[\oid{G}{S}])}
\end{equation}
which commutes with the projections and inclusions in~\eqref{eq:retraction-CN} and~\eqref{eq:retraction-THH}.
\end{theorem}

\begin{proof}
First, we define the map~$\hm$ in~\eqref{eq:THH-CN}.
Given pointed spaces $S$, $A$, and~$C$, consider the natural map
\begin{equation}
\label{eq:reinziehen}
\map(S,A)\sma C\TO\map(S,A\sma C)
\end{equation}
that is adjoint to $f\longmapsto f\sma\id_C$.
Since we need it later, we observe that if $S\cong S^x$, $A$ is $a$-connected, and $C$ is discrete, then the map~\eqref{eq:reinziehen} is $(2(a-x)+1)$-connected.
The map~\eqref{eq:reinziehen} induces a natural transformation of functors $\Lambda^\op\TO\Mod_{\CW\CT}$
\begin{align*}
\Bigl(\CI^{[\bullet]},M_{[\bullet]}(\spec{A})\sma CN_\bullet(\CC)_+\Bigr)
=
\Bigl(\CI^{[\bullet]},
\map\bigl(cn_{[\bullet]}\spec{S}(?),\;&-\sma cn_{[\bullet]}\spec{A}(?)\bigr)\sma CN_\bullet(\CC)_+
\Bigr)
\\[-\smallskipamount]
&\big\downarrow
\\[-\smallskipamount]
\Bigl(\CI^{[\bullet]},
\map\bigl(cn_{[\bullet]}\spec{S}(?),\;&-\sma cn_{[\bullet]}\spec{A}(?)\sma CN_\bullet(\CC)_+\bigr)
\Bigr)
\,,
\end{align*}
where $?$ denotes the variable in~$\CI^{[\bullet]}$ and $-$ the variable in~$\CW$.
Using the isomorphism~\eqref{eq:M(A[C])} from \autoref{cn(A[C])} and applying the  functor $\MOR{\hocolim}{\Mod_{\CW\CT}}{\CW\CT}$ from~\eqref{eq:hocolim-WT}, we obtain a map in~$\Lambda^\op(\CW\CT)$
\begin{equation}
\label{eq:THH-CN-bullet}
\MOR{\hm_\bullet}
{THH_\bullet(\spec{A})\sma CN_\bullet(\CC)_+
\cong\hocolim_{\CI^{[\bullet]}}\bigl(M_{[\bullet]}(\spec{A})\sma CN_\bullet(\CC)_+\bigr)}
{THH_\bullet(\spec{A}[\CC])}
\,,
\end{equation}
where the first isomorphism comes from the fact that $CN_\bullet(\CC)$ is constant on~$\CI^{[\bullet]}$.
Taking geometric realizations and using the fact that geometric realization and smash products commute, we obtain the map~$\hm$ in~\eqref{eq:THH-CN}.

If $\CC=\oid{G}{S}$, it is clear from the definitions that $\hm$ commutes with the projections and inclusions in~\eqref{eq:retraction-CN} and~\eqref{eq:retraction-THH}.

Next, we show that $\hm$ is a $\pi_*^1$-isomorphism.
Since retracts of $\pi_*^1$-isomorphisms are $\pi_*^1$-isomorphisms, this then implies that also $\hm_\CF$ is a $\pi_*^1$-isomorphism in the case when $\CC=\oid{G}{S}$.

The assumption that $\spec{A}$ is very well pointed implies that the same is true for~$\spec{A}[\CC]$.
Then, by \autoref{THH-Segal-good}\ref{i:THH-Segal-good} below, both the source and the target of~$\hm_\bullet$ in~\eqref{eq:THH-CN-bullet} are levelwise Segal good.
For the source, we also use Fact~\ref{facts-cofib}\ref{i:cofib-smash}.
Therefore, by \autoref{real-lemma-WT}, it is enough to show that $\hm_\bullet$ is a $\pi_*$-isomorphism in each simplicial degree.
For a fixed~$q\ge0$, both the source and the target of~$\hm_q$ are defined for arbitrary symmetric spectra, not just for symmetric ring spectra, and we proceed to prove that $\hm_q$ is a $\pi_*$-isomorphism in~$\CW\CT$ for every symmetric spectrum~$\spec{A}$.

We claim that it is enough to show that $\hm_q$ is a $\pi_*$-isomorphism for cofibrant symmetric spectra.
To prove this claim, take a cofibrant replacement
\(
\MOR{\gamma}{\spec{A}_{c}}{\spec{A}}
\)
in the stable model category of symmetric spectra, with $\gamma$ a stable fibration and a $\pi_*$-isomorphism.
From~\cite{MMSS}*{Proposition~9.9(iii) on page~472} it follows that~$\gamma$ is a levelwise weak equivalence.
Inspecting the definitions, and using the fact that homotopy colimits preserve weak equivalences, we conclude that $\gamma$ induces levelwise weak equivalences $THH_q(\spec{A}_{c})\TO THH_q(\spec{A})$ and $THH_q(\spec{A}_{c}[\CC])\TO THH_q(\spec{A}[\CC])$.
Therefore, if $\hm_q$ is a $\pi_*$-isomorphism for~$\spec{A}_{c}$, then the same is true for~$\spec{A}$, establishing the claim.

So it remains to show that $\hm_q$ is a $\pi_*$-isomorphism for cofibrant symmetric spectra.
By~\cite{MMSS}*{Theorem~6.5(iii) on page~461}, cofibrant symmetric spectra are retracts of $FI$-cell complexes in the sense of~\cite{MMSS}*{Definition~5.4 on page~457}.
Therefore it is enough to show that $\hm_q$ is a $\pi_*$-isomorphism for $FI$-cell complexes.
We argue by cellular induction.
The base case is treated below.
The induction steps are then carried out like in~\cite{Shipley}*{Lemma~4.3.2 and proof of Proposition~4.2.3 on pages~179--180}, thus completing the proof.

We show that $\hm_q$ is a $\pi_*$-isomorphism for a free symmetric spectrum $\spec{A}={F}_{m}Z$, where $Z=S^{d-1}_+$ or~$D^d_+$ and $m,d\ge0$.
By definition
\[
({F}_{m}Z)_x=
\begin{cases}
\ds\Sigma_{x+}\sma_{\Sigma_{x-m}}S^{x-m}\sma Z
&\text{if $x\ge m$,}
\\
\pt
&\text{if $0\le x<m$,}
\end{cases}
\]
and therefore $({F}_{m}Z)_x$ is at least $(x-m-1)$-connected.

Recall from~\eqref{eq:THH-CN-bullet} that $\hm_q(B)$ is defined by taking $\hocolim_{\CI^{[q]}}$ of the natural map
\begin{align}
\map\bigl(cn_{[q]}\spec{S}(\vec{x}),\;B\sma&\,cn_{[q]}{F}_{m}Z(\vec{x})\bigr)\sma CN_q(\CC)_+
\\[-\smallskipamount]
\label{eq:THH-CN-before-hocolim}
&\big\downarrow
\\[-\smallskipamount]
\map\bigl(cn_{[q]}\spec{S}(\vec{x}),\;B\sma&\,cn_{[q]}{F}_{m}Z(\vec{x})\sma CN_q(\CC)_+\bigr)
\end{align}
as in~\eqref{eq:reinziehen}, where $\vec{x}=(x_0,\dotsc,x_q)\in\obj\CI^{[q]}$ and $B\in\obj\CW$.
Let $x=x_0+\dotsb+x_q$.
Then, evaluated at~$B=S^n$, the map~\eqref{eq:THH-CN-before-hocolim} is at least $(2n-2(q+1)m-1)$-connected, as observed after~\eqref{eq:reinziehen}.
Since homotopy colimits preserve connectivity of maps (e.g., see~\cite{Dundas}*{Lemma~A.7.3.1 on page~373}), the same is true for~$\hm_q(S^n)$, and this then implies that $\hm_q$ is a $\pi_*$-isomorphism.
Here we use the fact that a map $\MOR{\psi}{\WT{X}}{\WT{Y}}$ in~$\CW\CT$ is a $\pi_*$-isomorphism provided that there exists a non-decreasing function $\MOR{\lambda}{\IN}{\IZ}$ such that $\lim_{n\to\infty}\lambda(n)=\infty$ and $\psi(S^n)$ is $(n+\lambda(n))$-connected for each $n\ge0$.

Now for the base case of the induction one needs to allow different free spectra in the different smash factors.
More precisely, one replaces
$cn_{[q]}{F}_{m}Z(\vec{x})$
by
${F}_{m_0}Z_0(x_0)\sma{F}_{m_1}Z_1(x_1)\sma\dotsb\sma{F}_{m_q}Z_q(x_q)$
and $(q+1)m$ by~$m_0+m_1+\dotsb+m_q$ throughout.
The argument above then works unchanged.
\end{proof}

\begin{example}
\label{THH-free-loop}
Combining \autoref{THH-CN} with \autoref{computation-CN}, we obtain an isomorphism in the non-equivariant stable homotopy category
\[
\THH_\CF(\spec{A}[G])
\simeq
\bigvee_{[c]\in\conj_\CF G}
\THH(\spec{A})\sma BZ_G\langle c\rangle_+
\,.
\]
In particular, if $G$ is finite, then $\pi_n(\THH(\spec{S}[G]))\tensor_\IZ\IQ=0$ for all~$n>0$, since $\THH(\spec{S})\simeq\spec{S}$.
\end{example}

In the proof above, as well as in the proof of \autoref{hofib-R}, we use the following lemma.
Some versions of it are certainly known to the experts.
In particular, part~\ref{i:THH-Segal-good} is stated and used in \cite{HM-top}*{proof of Proposition~2.4 on page~40}.
For the sake of completeness, we provide here some details.
The edgewise subdivision construction~$\sd_c$ used in part~\ref{i:THH-fix-Segal-good} is recalled before \autoref{THH-fix}.

\begin{lemma}
\label{THH-Segal-good}
Assume that the symmetric spectral category $\spec{D}$ is very well pointed.
Then:
\begin{enumerate}
\item
\label{i:THH-Segal-good}
The cyclic $\CW$-space $THH_\bullet(\spec{D})$ is levelwise Segal good.
\item
\label{i:THH-fix-Segal-good}
For any finite cyclic subgroup~$C$ of~$S^1$ and any $E\in\obj\CW^C$, the simplicial $\CW$-space $(\sh^{E}\sd_c THH_\bullet(\spec{D}))^C$ is levelwise Segal good.
\end{enumerate}
If $\spec{A}$ is a very well pointed symmetric ring spectrum and $\spec{D}=\spec{A}[\oid{G}{S}]$, then the statements in \ref{i:THH-Segal-good} and \ref{i:THH-fix-Segal-good} also hold for the $\CF$-parts.
\end{lemma}

\begin{proof}
\ref{i:THH-Segal-good}
For every $A \in \obj\CW$, $i \in \{ 0, \ldots , q\}$, and $\vec{x} \in \obj \CI^{[q]}$ the  map
\[
\MOR{s_i}{A \sma cn_{[q]} \spec{D} (\vec{x} )}{ A \sma cn_{[q+1]} \spec{D} (s_i(\vec{x}) ) }
\]
is a cofibration.
This uses the assumption that $\spec{D}$ is very well pointed, that the object-repeating map $\MOR{s_i}{ \obj \spec{D}^{[q]} }{\obj \spec{D}^{[q+1]}}$ is injective, and Facts~\ref{facts-cofib}\ref{i:cofib-smash}, \ref{i:cofib-smash-well-pointed} and~\ref{i:cofib-vee} below.
Fact \ref{facts-cofib}\ref{i:cofib-mapK} implies that also the map $\MOR{s_i}{M_{[q]} \spec{D} (\vec{x})(A)}{M_{[q+1]}\spec{D} (s_i(\vec{x}))(A)}$ is a cofibration.
Since $\MOR{s_i}{\CI^{[q]}}{\CI^{[q+1]}}$ only inserts a copy of~$0=\emptyset$, the assumptions of Fact~\ref{facts-cofib}\ref{i:cofib-hocolim} are verified and we conclude that
\[
\MOR{s_i}{\hocolim_{\CI^{[q]}} M_{[q]} \spec{D} ( - ) (A)}{\hocolim_{\CI^{[q+1]}} M_{[q+1]} \spec{D} ( s_i( - )) (A) }
\]
is a cofibration.

\ref{i:THH-fix-Segal-good}
Observe that
\[
cn_{c[q]} \spec{D} (c \vec{x} ) =
\bigvee_{\substack{(d^1, d^2 , \ldots , d^c)\\\text{in}\,\obj \spec{D}^{c[q]}}}
\bigwedge_{g \in C} \spec{B}_{g(d^1,d^2, \ldots , d^c)} ( \vec{x} )
\,,
\]
where the cyclic group $C$ operates by cyclic block permutations on
\[
(d^1, d^2 , \dots , d^c) = ((d^1_0 , d^1_1 , \ldots , d^1_q), (d^2_0 , d^2_1 , \ldots , d^2_q),  \ldots , (d^c_0 , d^c_1 , \ldots , d^c_q)) \in \obj \spec{D} ^{c[q]}
\]
and
\[
\spec{B}_{(d^1,d^2, \ldots , d^c)} (\vec{x}) = \spec{D} (d_0^1,d_q^c )_{x_0} \sma \spec{D} (d_1^1,d_0^1 )_{x_1} \sma \ldots \sma \spec{D} (d_q^1,d_{q-1}^1 )_{x_q}.
\]
The assumption together with Fact~\ref{facts-cofib}\ref{i:cofib-smash}, \ref{i:cofib-smash-well-pointed}, \ref{i:cofib-complicated}, and \ref{i:cofib-mapK-equi} yield that
\[
\MOR{s_i}{\bigl(\sh^E M_{c[q]} \spec{D} ( \vec{x})(A)\bigr)^C}{\bigl(\sh^E M_{c[q+1]} \spec{D} ( \vec{x} )(A)\bigr)^C}
\]
is a cofibration.
The claim now follows from \eqref{eq:sd-THH-fix} and again Fact~\ref{facts-cofib}\ref{i:cofib-hocolim}.
\end{proof}

The proof of \autoref{THH-Segal-good} used the following basic facts.
We emphasize that, even though these facts are concerned with constructions in pointed spaces, the term cofibration always refers to the unpointed version.

\begin{facts}
\label{facts-cofib}
Suppose $A\TO X$ is a pointed map in $\CT$ that is a cofibration.
\begin{enumerate}
\item \label{i:cofib-smash}
For every $Z\in\obj\CT$ the induced map $A \sma Z \TO X \sma Z$ is a cofibration.
\item \label{i:cofib-smash-well-pointed}
If $X$ and $Z$ are well pointed, then so is $X \sma Z$.
\item \label{i:cofib-mapK}
If $Z\in\obj\CT$ is compact, then the induced map $\map(Z,A) \TO \map(Z,X)$ is a cofibration.
\item \label{i:cofib-mapK-equi}
If in \ref{i:cofib-smash} or \ref{i:cofib-mapK} we
additionally assume that $Z\in\obj\CT^G$ is a $G$-space and $A \TO X$ is a $G$-equivariant cofibration, then the induced map is a $G$-equivariant cofibration.
\end{enumerate}
Suppose $A_s \TO X_s$, $s \in S$ is a family of pointed maps in $\CT$ that are cofibrations.
\begin{enumerate}[resume*]
\item \label{i:cofib-vee}
If all spaces $A_s$ and $X_s$ are well pointed, then the induced map $\bigvee_{s \in S} A_s \TO \bigvee_{s \in S} X_s$ is a cofibration.
\item \label{i:cofib-complicated}
If $S$ is a $G$-set and all spaces $A_s$ and $X_s$ are well pointed, then the induced map
\[
\bigvee_{s \in S} \bigwedge_{g \in G} A_{gs} \TO \bigvee_{s \in S} \bigwedge_{g \in G} X_{gs}
\]
is a $G$-equivariant cofibration.
\end{enumerate}
Let $\MOR{(f, \nu)}{(\CC,A)}{(\CD,X)}$  be a morphism in $\Mod_{\CT}$; see \autoref{NAT-MOD}.
\begin{enumerate}[resume*]
\item \label{i:cofib-hocolim}
Suppose the functor $\MOR{f}{\CC}{\CD}$ induces an injective function on objects and for all $c$,~$c' \in \obj \CC$ the function $\CC ( c , c') \TO \CD ( f(c) , f(c') )$ is injective.
If for all $c \in \obj \CC$ the spaces $A(c)$ and $X(c)$ are well-pointed and the natural transformation $\MOR{\nu_c}{A(c)}{X(f(c))}$ is a  cofibration, then the induced map
\[
\hocolim_\CC A  \TO \hocolim_\CD X
\]
is a cofibration.
\end{enumerate}
\end{facts}

We can now prove the main result of this section.

\begin{proof}[Proof of \autoref{split-THH}]
Apply the functor $\THH(\IA)\sma-$ to the diagram in \autoref{split-CN} to obtain the following diagram.
\begin{equation}
\label{eq:THH-sma-square}
\begin{tikzcd}[column sep=large]
\ds\THH(\spec{A})\sma\Bigl(EG(\CF)_+\sma_{\Or G}\real{CN_{\bullet}(\oid{G}{-})}_+\Bigr)
\arrow[r, "\id\sma\asbl"]
\arrow[d, "\id\sma\id\sma\pr_\CF"']
&
\ds\THH(\spec{A})\sma\real{CN_{\bullet}(G)}_+
\arrow[d, "\,\id\sma\pr_\CF"]
\\
\ds\THH(\spec{A})\sma\Bigl(EG(\CF)_+\sma_{\Or G}\real{CN_{\bullet,\CF}(\oid{G}{-})}_+\Bigr)
\arrow[r, "\id\sma\asbl"']
&
\ds\THH(\spec{A})\sma\real{CN_{\bullet,\CF}(G)}_+
\end{tikzcd}
\end{equation}
Then consider the maps
\[
\hspace{-.8ex}\begin{tikzcd}
\ds\THH(\spec{A})\sma\Bigl(EG(\CF)_+\sma_{\Or G}\real{CN_{\bullet,\CF}(\oid{G}{-})}_+\Bigr)
\arrow[d, "\ts\cong\ "']
\\
\ds EG(\CF)_+\sma_{\Or G}\Bigl(\THH(\spec{A})\sma\real{CN_{\bullet,\CF}(\oid{G}{-})}_+\Bigr)
\arrow[r, "\id\sma\hm_\CF"']
&
\ds EG(\CF)_+\sma_{\Or G}\THH_\CF(\spec{A}[\oid{G}{-}])
\,,
\end{tikzcd}
\]
where the first isomorphism uses the associativity and commutativity of smash products, and the second map is induced by~\eqref{eq:THH-CN-F-part}, and similarly at the other corners of diagram~\eqref{eq:THH-sma-square}.
By \autoref{THH-CN} these maps induce an objectwise $\pi_*^1$-isomorphism between the commutative square~\eqref{eq:THH-sma-square} and the one in \autoref{split-THH}.
The result then follows from \autoref{split-CN}.
\end{proof}

\begin{corollary}
\label{split-THH-hC}
Let~$C$~be any finite cyclic subgroup of~$S^1$.
Consider the following commutative diagram in~$\CW\CT$.
\[
\begin{tikzcd}[column sep=large]
\ds EG(\CF)_+\sma_{\Or G}\THH    (\spec{A}[\oid{G}{-}])_{hC}
\arrow[r, "\asbl"]
\arrow[d, "\id\sma\pr_\CF"']
&
\THH    (\spec{A}[G])_{hC}
\arrow[d, "\pr_\CF"]
\\
\ds EG(\CF)_+\sma_{\Or G}\THH_\CF(\spec{A}[\oid{G}{-}])_{hC}
\arrow[r, "\asbl"']
&
\THH_\CF(\spec{A}[G])_{hC}
\end{tikzcd}
\]
The left-hand map and the bottom map are $\pi_*$-isomorphisms.
If $\CF$ contains all cyclic groups, then the right-hand map is an isomorphism.
\end{corollary}

\begin{proof}
Apply $C$-homotopy orbits to the diagram in \autoref{split-THH}.
Since $C$-homotopy orbits send $\pi_*^1$-isomorphisms to $\pi_*$-isomorphisms, and $C$-homotopy orbits commute up to isomorphisms with smash products over the orbit category, the result follows.
\end{proof}


\section{Homotopy fiber of the restriction map}
\label{HOFIB-R}

Fix a prime~$p$.
As $n\ge1$ varies, the $\Cp{n}$-fixed points of~$\THH(\spec{D})$ are related by maps
\begin{equation}
\label{eq:R-and-F}
\MOR{R,F}{\THH(\spec{D})^{C_{p^{n}}}}{\THH(\spec{D})^{C_{p^{n-1}}}}
\,,
\end{equation}
called restriction and Frobenius.
The Frobenius map~$F$ is the inclusion of fixed points.
The definition of the restriction map~$R$ is more involved, and is reviewed in detail below.

The goal of this section is to establish the following result, which gives a point-set level description of the homotopy fiber of the restriction map.
This plays a key role in the proof of our splitting theorems; compare \autoref{no-R-revisited}.

Recall that, given~$\WT{X}\in\obj\CW\CT^{S^1}$ and a finite subgroup~$C$ of~$S^1$, the fixed points~$\WT{X}^C\in\obj\CW\CT$ are defined levelwise: $\WT{X}^C(A)=\WT{X}(A)^C$ for every~$A\in\obj\CW$.
The $\CW$-space $(\sh^{ES^1_+}\WT{X})^C$ is defined in \autoref{S1}.
The projection $\MOR{\pr}{ES^1}{\pt}$ induces a map $\MOR{\pr_*}{(\sh^{ES^1_+}\WT{X})^C}{(\sh^{S^0}\WT{X})^C\cong\WT{X}^C}$.

\begin{theorem}
\label{hofib-R}
If $\spec{D}$ is very well pointed, then for each~$n\geq1$
\[
\bigl(\sh^{ES^1_+}\THH(\spec{D})\bigr)^{C_{p^{n  }}}
\xrightarrow{\pr_*}
                  \THH(\spec{D})^{C_{p^{n  }}}
\TO[R]
                  \THH(\spec{D})^{C_{p^{n-1}}}
\]
is a stable homotopy fibration sequence in~$\CW\CT$.
\end{theorem}

We proceed to explain the definition of~$R$; compare~\citelist{\cite{BHM}*{page~495} \cite{DMcC-THH}*{Definition~1.5.2 on page~255}}.
It begins with the following observation, whose proof is straightforward.
Recall the definition of natural modules in~$\CW\CT$ from \autoref{NAT-MOD} and the functor $\MOR{\hocolim}{\Mod_{\CW\CT}}{\CW\CT}$ in~\eqref{eq:hocolim-WT}.

\begin{lemma}
\label{hocolim-fix}
Let $G$ be a discrete group acting on a natural module~$(\CC,\spec{X})\in\obj\Mod_{\CW\CT}$.
Then there is a natural isomorphism in~$\CW\CT$
\begin{equation}
\label{eq:hocolim-fix}
\bigl(\hocolim(\CC,\spec{X})\bigr)^G
\cong
\hocolim\bigl((\CC,\spec{X})^G\bigr)
\,.
\end{equation}
where
\begin{equation}
\label{eq:module-fix}
(\CC,\spec{X})^G
=
\bigl(\CC^G,(-)^G\circ\spec{X}_{|\CC^G}\bigr)
\,.
\end{equation}
\end{lemma}

Now let $C$ be a finite cyclic subgroup of~$S^1$ and let $c$ be its order.
We denote by $c[\bullet]=[\bullet]\sqcup\dotsb\sqcup[\bullet]$ the concatenation of $[\bullet]$ with itself $c$~times, and for every $\vec{x}\in\obj\CI^{[\bullet]}$ we let $c\vec{x}=(\vec{x},\vec{x},\dotsc,\vec{x})\in\obj\CI^{c[\bullet]}$.
Recall the edgewise subdivision construction~$\sd_c$ from~\cite{BHM}*{Section~1, pages~467--471}.

\begin{lemma}
\label{THH-fix}
For any symmetric spectral category~$\spec{D}$, there is a natural isomorphism in~$\CW\CT$
\begin{equation}
\label{eq:real-THH-fix}
\THH(\spec{D})^C
\cong
\real*{\bigl(\sd_c THH_\bullet(\spec{D})\bigr)^C}
\end{equation}
and a natural isomorphism in~$\Delta^\op(\CW\CT)$
\begin{equation}
\label{eq:sd-THH-fix}
\bigl(\sd_c THH_\bullet(\spec{D})\bigr)^C
\cong
\hocolim\Bigl(\CI^{[\bullet]},\bigl(M_{c[\bullet]}\spec{D}(c\,-)\bigr)^C\Bigr)
\,.
\end{equation}
\end{lemma}

\begin{proof}

The first isomorphism is true for any cyclic space; see~\cite{BHM}*{Section~1, pages~467--471}.
For the second, we have
\begin{align*}
\bigl(\sd_c THH_\bullet(\spec{D})\bigr)^C
&=
\Bigl(\hocolim\bigl(\CI^{c[\bullet]},M_{c[\bullet]}\spec{D}\bigr)\Bigr)^C
&\text{by definition}
\\
&\cong
\hocolim\Bigl(\bigl(\CI^{c[\bullet]},M_{c[\bullet]}\spec{D}\bigr)^C\Bigr)
&\text{by \eqref{eq:hocolim-fix}}
\\
&=
\hocolim\Bigl(\bigl(\CI^{c[\bullet]}\bigr)^C,\bigl(M_{c[\bullet]}\spec{D}_{|(\CI^{c[\bullet]})^C}\bigr)^C\Bigr)
&\text{by \eqref{eq:module-fix}}
\\
&\cong
\hocolim\Bigl(\CI^{[\bullet]},\bigl(M_{c[\bullet]}\spec{D}(c\,-)\bigr)^C\Bigr)
\,,
\end{align*}
where in the last isomorphism we use the identification
\[
\MOR[\cong]{c\,-}{\CI^{[\bullet]}}{\bigl(\CI^{c[\bullet]}\bigr)^C}
\,,\qquad\vec{x}\longmapsto c\vec{x}
\,.
\qedhere
\]
\end{proof}

Now we apply \autoref{THH-fix} for $C=C_{p^n}$ and $C=C_{p^{n-1}}$.
In order to define the map~$R$ in~\eqref{eq:R-and-F} the key ingredient is then a map in~$\CW\CT$
\begin{equation}
\label{eq:r-map}
\begin{tikzcd}[column sep=-.5em, row sep=scriptsize]
M_{p^{n}[q]}\spec{D}(p^{n}\vec{x})^{C_{p^{n}}}
\arrow[d, "\ r"]
&=&
\map\Bigl(cn_{p^{n}[q]}\spec{S}(p^{n}\vec{x}),\;-\sma cn_{p^{n}[q]}\spec{D}(p^{n}\vec{x})\Bigr)^{C_{p^{n}}}
\arrow[d, "\ r", shorten <=-1ex, shorten >=-1ex]
\\
M_{p^{n-1}[q]}\spec{D}(p^{n-1}\vec{x})^{C_{p^{n-1}}}
\!
&=&
\map\Bigl(cn_{p^{n-1}[q]}\spec{S}(p^{n-1}\vec{x}),\;-\sma cn_{p^{n-1}[q]}\spec{D}(p^{n-1}\vec{x})\Bigr)^{C_{p^{n-1}}}
\end{tikzcd}
\end{equation}
for any~$\vec{x}\in\obj\CI^{[q]}$.
For any~$A\in\obj\CW$ the map~$r$ in~\eqref{eq:r-map} is defined by considering the following maps.
\begin{equation}
\label{eq:res-to-fix}
\begin{tikzcd}[row sep=scriptsize]
\map\Bigl(cn_{p^{n}[q]}\spec{S}(p^{n}\vec{x}),\;A\sma cn_{p^{n}[q]}\spec{D}(p^{n}\vec{x})\Bigr)^{C_{p^{n}}}
\arrow[d, "\ f\mapsto f^{C_p}", shorten <=-1ex, shorten >=-1ex]
\\
\map\Bigl(\bigl(cn_{p^{n}[q]}\spec{S}(p^{n}\vec{x})\bigr)^{C_p},\;A\sma \bigl(cn_{p^{n}[q]}\spec{D}(p^{n}\vec{x})\bigr)^{C_p}\Bigr)^{C_{p^{n}}/C_p}
\\
\map\Bigl(cn_{p^{n-1}[q]}\spec{S}(p^{n-1}\vec{x}),\;A\sma cn_{p^{n-1}[q]}\spec{D}(p^{n-1}\vec{x})\Bigr)^{C_{p^{n-1}}}
\arrow[u, "\cong\ ", "\ \Delta"', shorten <=-1ex, shorten >=-1ex]
\end{tikzcd}
\end{equation}
The first map is given by restricting to the $C_p$-fixed points, using the fact that the action on~$A$ is trivial.
For the second map, notice that there is a $C_{p^n}/C_p$-equivariant homeomorphism
\begin{equation}
\label{eq:Delta-map}
\MOR[\cong]{\Delta_{\spec{D}}}%
{cn_{p^{n-1}[q]}\spec{D}(p^{n-1}\vec{x})}%
{\bigl(cn_{p^{n}[q]}\spec{D}(p^{n}\vec{x})\bigr)^{C_p}}
\,.
\end{equation}
For $\spec{D}=\spec{S}$, this specializes to the $C_{p^n}/C_p$-equivariant homeomorphism
\[
\MOR[\cong]{\Delta_{\spec{S}}}%
{(S^{x_0} \wedge \dots \wedge S^{x_{q}})^{\wedge p^{n-1}}}%
{((S^{x_0} \wedge \dots \wedge S^{x_{q}})^{\wedge p^n})^{C_p}}
\,.
\]
The homeomorphisms $\Delta_{\spec{S}}$ and~$\Delta_{\spec{D}}$, together with the identification $C_{p^n}/C_p \cong C_{p^{n-1}}$ that takes every element of order~$p^n$ on the circle to its $p$-th power, induce the homeomorphism~$\Delta$ in~\eqref{eq:res-to-fix}.
Finally, the key map~$r$ in~\eqref{eq:r-map} is defined as the composition of the top map with the inverse of the bottom map in~\eqref{eq:res-to-fix}.

It is not hard to verify that, as $[q]$ varies, the maps~$r$ assemble to a cyclic map
\[
\Bigl(\CI^{[\bullet]},\bigl(M_{p^{n}[\bullet]}\spec{D}(p^{n}\,-)\bigr)^{C_{p^{n}}}\Bigr)
\TO
\Bigl(\CI^{[\bullet]},\bigl(M_{p^{n-1}[\bullet]}\spec{D}(p^{n-1}\,-)\bigr)^{C_{p^{n-1}}}\Bigr)
\]
in $\Lambda^\op\Mod_{\CW\CT}$.
Using \eqref{eq:sd-THH-fix} and \eqref{eq:real-THH-fix}, we thus get the desired map~$R$ in~\eqref{eq:R-and-F}.

One of the key properties of the $R$ and~$F$ maps is that they commute:
\begin{equation}
\label{eq:R-and-F-commute}
\begin{tikzcd}
\ds\THH(\spec{D})^{C_{p^{n+1}}}
\arrow[r, "R", pos=.4]
\arrow[d, "F"', shift right=1em]
&
\ds\THH(\spec{D})^{C_{p^{n\phantom{+1}}}}
\arrow[d, "F", shift right=1em]
\\
\ds\THH(\spec{D})^{C_{p^{n\phantom{+1}}}}
\arrow[r, "R"', pos=.4, shorten <=-1em]
&
\ds\THH(\spec{D})^{C_{p^{n-1}}}
\mathrlap{\,;}
\end{tikzcd}
\end{equation}
e.g., see~\cite{Dundas}*{Section~6.2.3 on pages~237--240}.

\begin{warning}
\label{no-R}
In general, for $\spec{D}=\spec{A}[\oid{G}{S}]$ there is no map
\[
\begin{tikzcd}[column sep=large]
\ds R_\CF\colon \THH_\CF(\spec{A}[\oid{G}{S}])^{C_{p^{n}}}
\arrow[dotted, r, "\text{\raisebox{1.5ex}{\tiny\dbend}}" description]
&
\ds\THH_\CF(\spec{A}[\oid{G}{S}])^{C_{p^{n-1}}}
\,.
\end{tikzcd}
\]
The problem is with the wrong-way homeomorphism~$\Delta$ in~\eqref{eq:res-to-fix}.
The homeomorphism $\Delta_{\spec{A}[\oid{G}{S}]}$ in~\eqref{eq:Delta-map} does restrict to a map between the $\CF$-parts.
But this restricted map is not surjective without further conditions on~$G$ and~$\CF$.
For example, take $\spec{A}=\spec{S}$, $S=\pt$, $n=1$, and~$q=0$.
Identify~$\oid{G}{\pt}=G$.
Then the map~$\Delta_{\spec{S}[G]}$ restricts to
\begin{equation}
\label{eq:warning}
cn_{[0],\CF}\spec{S}[G](0)
\TO
\bigl(cn_{p[0],\CF}\spec{S}[G](0,\dotsc,0)\bigr)^{C_p}
\,.
\end{equation}
Under the identification
\[
cn_{p[0],\CF}\spec{S}[G](0,\dotsc,0)
\cong
\SET{(g_1,\dotsc,g_p)\in G\times\dotsb\times G}{\langle g_1\dotsb g_p\rangle\in\CF}_+
\,,
\]
in the case~$\CF=1$ the map~\eqref{eq:warning} corresponds to the inclusion
\[
S^0\TO\SET{g\in G}{g^p=1}_+
\,.
\]
\end{warning}
\medskip

Next we identify the homotopy fiber of~$R$.
The first ingredient is the following.

\begin{lemma}
\label{pre-hofib-R}
Let $E$ be a model for~$EC_{p^n}$, i.e., a contractible free $C_{p^n}$-CW-space.
Let $X$ be a pointed $C_{p^n}$-CW-space and~$Y$ a pointed $C_{p^n}$-space.
Then there is a natural homotopy fibration sequence
\[
\map\bigl(X, E_+ \sma Y\bigr)^{C_{p^n}}
\xrightarrow{\pr_*}
\map\bigl(X, Y\bigr)^{C_{p^n}}
\TO[r]
\map\bigl(X^{C_p}, Y^{C_p}\bigr)^{C_{p^n}/C_p}
\,,
\]
where $\pr_*$ is induced by the projection $\MOR{\pr}{E_+ \sma Y}{Y}$, and $r$ takes each $C_{p^n}$-equivariant map $\MOR{f}{X}{Y}$ to its restriction $\MOR{f^{C_p}}{X^{C_p}}{Y^{C_p}}$ to the $C_p$-fixed parts.
\end{lemma}

\begin{proof}
By mapping the $C_{p^n}$-equivariant cofibration sequence $X^{C_p} \to X \to X/X^{C_p}$ into the $C_{p^n}$-equivariant map $\MOR{\pr}{E_+\sma Y}{Y}$ we get the following map of vertical fibration sequences.
\[
\begin{tikzcd}[row sep=scriptsize]
\map\bigl(X/X^{C_p}, E_+ \sma Y\bigr)^{C_{p^n}}
\arrow[d, "\cong"]
\arrow[r, "\simeq"]
&
\map\bigl(X/X^{C_p}, Y\bigr)^{C_{p^n}}
\arrow[d]
\\
\map\bigl(X, E_+ \sma Y\bigr)^{C_{p^n}}
\arrow[d]
\arrow[r, "\pr"]
&
\map\bigl(X, Y\bigr)^{C_{p^n}}
\arrow[d]
\\
\map\bigl(X^{C_p}, E_+ \sma Y\bigr)^{C_{p^n}}
\arrow[r]
&
\map\bigl(X^{C_p}, Y\bigr)^{C_{p^n}}
\end{tikzcd}
\]
The top map is a weak equivalence because the pointed $X_{p^n}$-CW-space $X/X^{C_p}$ is $C_{p^n}$-free and $E_+ \sma Y \TO Y$ is a non-equivariant equivalence.
The upper left-hand vertical map is an isomorphism because the lower left-hand space is a single point.

The desired homotopy fibration sequence appears as the middle horizontal map, followed by the lower right-hand vertical map composed with the identification $\map(X^{C_p}, Y)^{C_{p^n}} \cong \map(X^{C_p}, Y^{C_p})^{C_{p^n}/C_p}$.
\end{proof}

Recall the definition of $\sh^E\WT{X}\in\obj\CW\CT^C$ for $E\in\obj\CW_C$ and $\WT{X}\in\obj\CW\CT^C$; compare \autoref{shift}.

\begin{corollary}
\label{cor-pre-hofib-R}
Let $E$ be a model for~$EC_{p^n}$, $X$ a pointed $C_{p^n}$-CW-space, and~$Y$ a pointed $C_{p^n}$-space.
Consider the continuous functor $\MOR{\spec{M}=\map(X,-\sma Y)}{\CW}{\CT^{C_{p^n}}}$.
Then there is a levelwise homotopy fibration sequence in~$\CW\CT$
\[
\bigl(\sh^{E_+}\spec{M}\bigr)^{C_{p^n}}
\xrightarrow{\pr_*}
\spec{M}^{C_{p^n}}
\TO[r]
\map(X^{C_p}, -\sma Y^{C_p})^{C_{p^n}/C_p}
\,.
\]
\end{corollary}

\begin{proof}
This follows at once from \autoref{pre-hofib-R}.
\end{proof}

\begin{corollary}
\label{cor-cor-pre-hofib-R}
Let $n\geq1$ and let $E$ be a model for~$EC_{p^n}$.
Then for all~$[q]\in\obj\Lambda^\op$ and all~$\vec{x}\in\obj\CI^{[q]}$ there is a levelwise homotopy fibration sequence in~$\CW\CT$
\[
\bigl(\sh^{E_+}M_{p^{n  }[q]}\spec{D}(p^{n  }\vec{x})\bigr)^{C_{p^{n  }}}
\xrightarrow{\pr_*}
\bigl(         M_{p^{n  }[q]}\spec{D}(p^{n  }\vec{x})\bigr)^{C_{p^{n  }}}
\TO[r]
\bigl(         M_{p^{n-1}[q]}\spec{D}(p^{n-1}\vec{x})\bigr)^{C_{p^{n-1}}}
\!.
\]
\end{corollary}

\begin{proof}
This follows from the previous \autoref{cor-pre-hofib-R}, identifying the term on the right using the homeomorphism~$\Delta$ in~\eqref{eq:res-to-fix}.
\end{proof}

We are now ready to prove the main result of this section.

\begin{proof}[Proof of \autoref{hofib-R}]
Consider $ES^1$ as a $\Cp{n}$-space by restricting the action, and notice that it is a model for~$E\Cp{n}$.
Recall that in~$\CW\CT$ stable homotopy fibration sequences and stable homotopy cofibration sequences are the same, and therefore they are preserved by homotopy colimits.
So by applying $\hocolim_{\CI^{[q]}}$ to the (levelwise and hence) stable homotopy fibration sequence from \autoref{cor-cor-pre-hofib-R} 
we get a sequence in~$\Delta^\op(\CW\CT)$
\[
\bigl(\sh^{ES^1_+}\sd_{p^{n}}THH_\bullet(\spec{D})\bigr)^{C_{p^{n  }}}
\xrightarrow{\pr_*}
\bigl(\sd_{p^{n}}            THH_\bullet(\spec{D})\bigr)^{C_{p^{n  }}}
\TO[R]
\bigl(\sd_{p^{n-1}}          THH_\bullet(\spec{D})\bigr)^{C_{p^{n-1}}}
\]
which in every simplicial degree is a stable homotopy fibration sequence in~$\CW\CT$.
To identify the term on the left here we used the fact that~\eqref{eq:sd-THH-fix} generalizes to
\[
\big(\sh^E\sd_c THH_\bullet(\spec{D})\bigr)^C
\cong
\hocolim\Bigl(\CI^{[\bullet]},\bigl(\sh^E M_{c[\bullet]}\spec{D}(c\,-)\bigr)^C\Bigr)
\,.
\]
Since we are assuming that $\spec{D}$ is very well pointed, \autoref{THH-Segal-good}\ref{i:THH-fix-Segal-good} ensures that we can apply \autoref{real-hofib}.
The fact that~\eqref{eq:real-THH-fix} generalizes to
\[
\bigl(\sh^E\THH(\spec{D})\bigr)^C
\cong
\real*{\bigl(\sh^E\sd_c THH_\bullet(\spec{D})\bigr)^C}
\]
finishes the proof.
\end{proof}


\section{Adams isomorphism and the fundamental fibration sequence}
\label{ADAMS}

The main result of this section is the following theorem.
Together with \autoref{hofib-R}, it identifies up to natural $\pi_*$-isomorphisms the homotopy fiber of the restriction map.
We emphasize that this identification is natural before passing to the stable homotopy category.
This subtlety is essential when dealing with assembly maps.

Denote by $\MOR{\forget}{\CW\CT^C}{\Th\CT^C}$ the forgetful functor from $C$-equivariant $\CW$-spaces to orthogonal $C$-spectra.

\begin{theorem}
\label{THH-Adams}
Let $C$ be a finite cyclic subgroup of~$S^1$.
\begin{enumerate}
\item\label{i:THH-Adams}
If $\spec{D}$ is an objectwise strictly connective, objectwise convergent, and very well pointed symmetric spectral category, then there is a zig-zag of natural $\pi_*^1$-isomorphisms of orthogonal spectra between
\[
\bigl(\forget\THH(\spec{D})\bigr)_{hC}
\AND
\bigl(\forget\sh^{ES^1_+}\THH(\spec{D})\bigr)^{C}
.
\]
For each $n\ge1$, in the stable homotopy category there is an exact triangle
\[
     \bigl(\forget\THH(\spec{D})\bigr)_{hC_{p^n}}
\TO
           \forget\THH(\spec{D})^{C_{p^{n  }}}
\TO[R]
           \forget\THH(\spec{D})^{C_{p^{n-1}}}
\,,
\]
which is referred to as the fundamental fibration sequence.
\item\label{i:THH-Adams-F-part}
Let $\spec{D}=\spec{A}[\oid{G}{S}]$.
Assume that the symmetric ring spectrum~$\spec{A}$ is \cplus.
Then there is a zig-zag of $\pi_*^1$-isomorphisms of orthogonal spectra between
\[
\bigl(\forget\THH_\CF(\spec{A}[\oid{G}{S}])\bigr)_{hC}
\AND
\bigl(\forget\sh^{ES^1_+}\THH_\CF(\spec{A}[\oid{G}{S}])\bigr)^{C}
\]
that is natural in~$S$ and commutes with the projection and inclusion maps induced by~\eqref{eq:retraction-THH}.
\end{enumerate}
\end{theorem}

The proof of \autoref{THH-Adams} is based on the following key result, which in turn uses two main ingredients.
The first is the natural model for the Adams isomorphism developed by the second and last authors in~\cite{RV}, and explained in the course of the proof below.
The second is a result by Blumberg~\cite{Blumberg}, giving an equivariant generalization of an unpublished result of Lydakis~\cite{Lydakis}*{Proposition~11.7 on page~36}; see also~\cite{MMSS}*{Proposition~17.6 on page~494}.
Notice that the assumption on $\forget(\sh^{E_+}\WT{X})$ below is satisfied whenever $\WT{X}$ is fibrant in the absolute stable model structure on~$\CW_C\CT_C$ by \cite{Blumberg}*{Theorem~1.3(6) on page~2263}.

\begin{theorem}
\label{Adams}
Let $C$ be any finite group.
Let $\WT{X}\in\obj\CW\CT^C$ and let $E$ be a finite free $C$-CW-space.
Assume that $\forget\sh^{E_+}\WT{X}$ is an orthogonal $C$-$\Omega$-spectrum.
Then there is a zig-zag of natural $\pi_*$-isomorphisms of orthogonal spectra between
\[
E_+\sma_C \forget\WT{X}
\AND
\bigl(\forget\sh^{E_+}\WT{X}\bigr)^C
.
\]
\end{theorem}

\begin{proof}
Choose a complete $C$-universe~$\CU$.
In the category~$\Th\CT^C$ of orthogonal $C$-spectra, consider the bifunctorial replacement construction $\Q{\CU}$ and the natural map $\MOR{r}{\spec{Y}}{\Q{\CU}(\spec{Y})}$ from~\cite{RV}*{Theorem~1.1 on pages~1494--1495}.
When $\spec{Y}$ is good in the sense of~\cite{RV}*{Definition~5.5 on pages~1511--1512} (compare also~\cite{HM-top}*{Appendix~A on pages~96--98}), then $\Q{\CU}(\spec{Y})$ is a $C$-$\Omega$-spectrum and $r$ is a \piiso.
In particular, if $\MOR{f}{\spec{Y}}{\spec{Y}'}$ is a \piiso{} of good orthogonal $C$-spectra, then $\Q{\CU}(f)$ is a \piiso{} and therefore a level equivalence of $C$-$\Omega$-spectra.
Recall also that \piiso s of $C$-$\Omega$-spectra induce $\pi_*$-isomorphisms between the $C$-fixed points; see for example~\cite{RV}*{Lemma~6.20(iii) on page~1522}.
Here a \piiso{} is a map $\MOR{f}{\spec{Y}}{\spec{Y}'}$ such that $\MOR{\pi_n^H(f)}{\pi_n^H(\spec{Y})}{\pi_n^H(\spec{Y}')}$ is an isomorphism for all $n\in\IZ$ and all subgroups $H$ of~$C$.

Consider also a cofibrant replacement functor~$\cofr$ in the stable model category of orthogonal $C$-spectra (compare~\cite{MM}*{Section~III.4 on pages~47-51}), and let $\MOR{\gamma}{\cofr\spec{Y}}{\spec{Y}}$ be the corresponding \piiso{} and fibration.
Notice that cofibrant spectra are good by \cite{RV}*{Lemma~5.10(i) on page~1514}.
Recall that the fibrant objects are the $C$-$\Omega$-spectra.
In particular, if $\spec{Y}$ is a $C$-$\Omega$-spectrum, then so is~$\cofr\spec{Y}$.

We have now the following zig-zag of natural maps of orthogonal spectra, which we explain below.
\[
\hspace{-1.1ex}
\begin{tikzcd}[row sep=small, column sep=1.7em]
\ds E_+\sma_C\forget\WT{X}
&
\ds E_+\sma_C\cofr\forget\WT{X}
\arrow[l, "\ts\one"']
&
\ds EC_+\sma_C(E_+\sma\cofr\forget\WT{X})
\arrow[l, "\ts\two"']
\arrow[r, "\ts\bthree"]
&
\ds\Q{\CU}\bigl(E_+\sma \cofr\forget\WT{X}\bigr)^C
\\
&
&
&
\mathclap{\ds\Q{\CU}\bigl(\cofr(E_+\sma\cofr\forget\WT{X})\bigr)^C}
\arrow[u, "\ts\,\four"', shorten <=-.5ex, shorten >=-.75ex]
\arrow[d, "\ts\,\five",  shorten <=-.5ex, shorten >=-1ex]
\\
\ds\bigl(\forget\sh^{E_+}\WT{X}\bigr)^C
&
\ds\bigl(\cofr\forget\sh^{E_+}\WT{X}\bigr)^C
\arrow[l, "\ts\eight"]
\arrow[r, "\ts\seven"']
&
\ds\Q{\CU}\bigl(\cofr\forget\sh^{E_+}\WT{X}\bigr)^C
&
\ds\Q{\CU}\bigl(\cofr(E_+\sma\forget\WT{X})\bigr)^C
\arrow[l, "\ts\bsix"]
\end{tikzcd}
\]
The maps labeled $\one$, $\four$, $\five$, and~$\eight$ are all induced by~$\gamma$, and the map~$\seven$ by~$r$.
The map~$\two$ comes from the isomorphism $EC_+\sma_C(E_+\sma\cofr\forget\WT{X})\cong(EC\times E)_+\sma_C\cofr\forget\WT{X}$ and the projection~$EC\TO\pt$.
Since $E$ is free, $\two$ is a homotopy equivalence and $E_+\sma\cofr\forget\WT{X}$ is $C$-free in the sense of~\cite{RV}*{(1.6) on page~1495}.
Moreover, since $\cofr\forget\WT{X}$ is good, also $E_+\sma\cofr\forget\WT{X}$ is good by~\cite{RV}*{Lemma~5.10(v) on page~1515}.

Using the facts recalled in the first two paragraphs of the proof, we see that $\four$, $\five$, $\seven$, and $\eight$ are all $\pi_*$-isomorphisms.
Since $E$ is free, $\one$ is a $\pi_*$-isomorphism by~\cite{RV}*{Lemma~5.13(iv) on page~1517}.
It remains to describe the maps $\bthree$ and~$\bsix$, which are the key steps in the argument.

By~\cite{RV}*{Main Theorem~1.7 on page~1496}, for every good and $C$-free orthogonal $C$-spectrum~$\spec{Y}$ there is natural $\pi_*$-isomorphism
\[
\MOR{A}{EC_+\sma_C\spec{Y}}{\Q{\CU}(\spec{Y})^C}
\,.
\]
The map~$A$ induces in the stable homotopy category the Adams isomorphism of Lewis, May, and Steinberger~\cite{LMS}*{Theorem~II.7.1 on page~97}, which in the special case of suspension spectra was proved by Adams~\cite{Adams}*{Theorem~5.4 on page~500}.
We emphasize that the map~$A$ is natural before passing to the homotopy category.
This explains the natural $\pi_*$-isomorphism~$\bthree$.

By~\cite{Blumberg}*{Proposition~3.6 on page~2275}, the natural assembly map
\[
E_+\sma\WT{X}\TO\sh^{E_+}\WT{X}
\]
is a \piiso{} for every $\WT{X}\in\obj\CW\CT^C$ and for every finite $C$-CW-space~$E$.
This result in fact holds for any pointed finite $C$-CW-space~$B$ instead of~$E_+$.
This gives a \piiso{} $E_+\sma\forget\WT{X}\cong\forget(E_+\sma\WT{X})\TO\forget\sh^{E_+}\WT{X}$.
Using the facts above, this induces the natural $\pi_*$-isomorphism~$\bsix$, thus completing the proof.
\end{proof}

In order to apply \autoref{Adams} to the case $\WT{X}=\THH(\spec{D})$ we need the following result proven by Hesselholt and Madsen~\cite{HM-top}.

\begin{proposition}
\label{THH-Omega}
Let $C$ be a finite cyclic subgroup of~$S^1$, and let $E\in\obj\CW_C$.
\begin{enumerate}
\item\label{i:THH-Omega}
Assume that the symmetric spectral category~$\spec{D}$ is objectwise strictly connective, objectwise convergent, and very well pointed.
Then $\forget\sh^E\THH(\spec{D})$ is an orthogonal $C$-$\Omega$-spectrum.
\item\label{i:THH-Omega-F-parts}
The same is true for $\forget\sh^E\THH_\CF(\spec{A}[\oid{G}{S}])$ provided that the symmetric ring spectrum~$\spec{A}$ is \cplus.
\end{enumerate}
\end{proposition}

\begin{proof}
A version of part~\ref{i:THH-Omega} is proved in~\cite{HM-top}*{proof of Proposition~2.4 on page~40} in the setting of functors with smash products, or FSPs.
An inspection of their argument, together with \autoref{THH-Segal-good}\ref{i:THH-fix-Segal-good}, shows that the proof works also under the stated assumptions on~$\spec{D}$.

If $\spec{A}$ is \cplus, then $\spec{A}[\CC]$ is objectwise strictly connective, objectwise convergent, and very well pointed for any small category~$\CC$.
To verify the second property, notice that if $\spec{A}_x$ is $(x-1)$-connected and $\spec{A}_x\TO\Omega\spec{A}_{x+1}$ is $(x+\lambda(x))$-connected, then $\bigvee\spec{A}_x\TO\Omega\bigvee\spec{A}_{x+1}$ is roughly $\min(x+\lambda(x),2x)$-connected.
Here the wedge sum is indexed by any morphism set in~$\CC$.
The other properties are clear.

Therefore part~\ref{i:THH-Omega} applies to~$\spec{D}=\spec{A}[\oid{G}{S}]$, and since retracts of $\Omega$-spectra are again $\Omega$-spectra, part~\ref{i:THH-Omega-F-parts} follows.
\end{proof}

We can now conclude the proof of \autoref{THH-Adams}.

\begin{proof}[Proof of \autoref{THH-Adams}]
We first prove part~\ref{i:THH-Adams}.
Notice that the second claim in~\ref{i:THH-Adams} follows immediately from the first claim and \autoref{hofib-R}.
Consider $ES^1$ as a $C$-space by restricting the action.
Decompose $ES^1\cong\colim_{i\in\IN}E_i$ as a colimit along closed embeddings of finite free $C$-CW-complexes~$E_i$.
For each~$i$ take the zig-zag of natural $\pi_*$-isomorphisms given by \autoref{Adams} and \autoref{THH-Omega}, and then pass to the homotopy colimit to obtain a zig-zag of natural $\pi_*$-isomorphisms between
\[
\hocolim_{i\in\IN}\Bigl(E_{i+}\sma_C\forget\THH(\spec{D})\Bigr)
\AND
\hocolim_{i\in\IN}\Bigl(\bigl(\forget\sh^{E_{i+}}\THH(\spec{D})\bigr)^C\Bigr)
\,.
\]
Using well known properties of closed embeddings (e.g., see~\cite{RV}*{Facts~5.7 on pages~1512--1513}), we see that on both sides the homotopy colimit is taken levelwise along closed embeddings.
Therefore the natural projection $\hocolim\TO\colim$ is on both sides a $\pi_*$-isomorphism.
Finally, notice that we have isomorphisms
\begin{align*}
\colim_{i\in\IN}\Bigl(E_{i+}\sma_C\forget\THH(\spec{D})\Bigr)
&\cong
ES^1_+\sma_C\forget\THH(\spec{D})
,
\\
\colim_{i\in\IN}\Bigl(\bigl(\forget\sh^{E_{i+}}\THH(\spec{D})\bigr)^C\Bigr)
&\cong
\bigl(\forget\sh^{ES^1_+}\THH(\spec{D})\bigr)^C
\,.
\end{align*}
This proves part~\ref{i:THH-Adams}.
As explained in the proof of \autoref{THH-Omega}\ref{i:THH-Omega-F-parts}, the assumption on~$\spec{A}$ implies that part~\ref{i:THH-Adams} applies to~$\spec{D}=\spec{A}[\oid{G}{S}]$.
Since retracts of $\pi_*^1$-isomorphisms are $\pi_*^1$-isomorphisms, and using naturality of the zig-zag in \autoref{Adams}, part~\ref{i:THH-Adams-F-part} follows.
\end{proof}

We close this section with the following result, which is an immediate consequence of \autoref{split-THH-hC} and \autoref{THH-Adams}\ref{i:THH-Adams-F-part}.

\begin{corollary}
\label{split-THH-sh-fix}
Assume that the symmetric ring spectrum $\spec{A}$ is \cplus.
Let~$C$ be a finite cyclic subgroup of~$S^1$.
Consider the following commutative diagram in~$\CW\CT$.
\[
\begin{tikzcd}
\ds EG(\CF)_+\sma_{\Or G}\bigl(\sh^{ES^1_+}\THH    (\spec{A}[\oid{G}{-}])\bigr)^C
\arrow[r, "\asbl"]
\arrow[d, "\id\sma\pr_\CF"']
&
\bigl(\sh^{ES^1_+}\THH    (\spec{A}[G])\bigr)^C
\arrow[d, "\pr_\CF"]
\\
\ds EG(\CF)_+\sma_{\Or G}\bigl(\sh^{ES^1_+}\THH_\CF(\spec{A}[\oid{G}{-}])\bigr)^C
\arrow[r, "\asbl"']
&
\bigl(\sh^{ES^1_+}\THH_\CF(\spec{A}[G])\bigr)^C
\end{tikzcd}
\]
The left-hand map and the bottom map are $\pi_*$-isomorphisms.
If $\CF$ contains all cyclic groups, then the right-hand map is an isomorphism.
\end{corollary}


\section{\BHM's functor \texorpdfstring{$C$}{C}}
\label{FUNCTOR-C}

B\"okstedt, Hsiang, and Madsen define in~\cite{BHM}*{(5.14) on page~497} an infinite loop space~$C(X;p)$ for any space~$X$ and any prime~$p$.
In this section we define a $\CW$-space $\C(\spec{D};p)$ for any symmetric spectral category~$\spec{D}$, which generalizes their original construction in the following sense:
if $X=BG$ for a discrete group~$G$ and~$\spec{D}=\spec{S}[G]$, then $C(BG;p)$ and $\C(\spec{S}[G];p)$ are isomorphic in the stable homotopy category.

Consider an $S^1$-equivariant $\CW$-space $\WT{X}\in\obj\CW\CT^{S^1}$.
We denote by~$\F(\WT{X};p)$ the $\CW$-space defined as the homotopy limit in~$\CW\CT$ of the $C_{p^n}$-fixed points of~$\WT{X}$ along the inclusions~$F$:
\[
\F(\WT{X};p)=\holim_{n\in\IN}\bigl(\:
\dotsb\TO
\WT{X}^{C_{p^{n}}}
\TO[F]
\WT{X}^{C_{p^{n-1}}}
\TO
\dotsb
\TO
\WT{X}^{C_{p}}
\TO[F]
\WT{X}
\:\bigr)
\,.
\]
Recall from~\autoref{SPECTRA} that homotopy limits in~$\CW\CT$ are defined by first applying a fibrant replacement functor and then taking levelwise homotopy limits.

\begin{definition}[\BHM's functor~$C$]
\label{C}
Given a symmetric spectral category~$\spec{D}$ and a prime~$p$, we define
\[
\C(\spec{D};p)
=
\F\Bigl(\sh^{ES^1_+}\THH(\spec{D});p\Bigr)
\,.
\]
If $\spec{D}=\spec{A}[\oid{G}{S}]$ we define
\[
\C_\CF(\spec{A}[\oid{G}{S}];p)
=
\F\Bigl(\sh^{ES^1_+}\THH_\CF(\spec{A}[\oid{G}{S}]);p\Bigr)
\,.
\]
\end{definition}

Thus we obtain from~\eqref{eq:retraction-THH} maps
\begin{equation}
\label{eq:retraction-C}
\begin{tikzcd}
\C    (\spec{A}[\oid{G}{S}];p)
\arrow[r, shift left, "\pr_{\CF}"]
&
\C_\CF(\spec{A}[\oid{G}{S}];p)
\arrow[l, shift left, "\ii_{\CF}"]
\end{tikzcd}
\end{equation}
satisfying $\pr_{\CF} \circ \ii_{\CF} = \id$.

We also use the~$\F$-construction to define
\begin{equation}
\label{eq:TF}
\TF(\spec{D};p)
=
\F\bigl(\THH(\spec{D});p\bigr)
\,;
\end{equation}
compare~\cite{HM-top}*{(20) on page~56}.
The projection $\MOR{\pr}{ES^1}{\pt}$ induces a map
\begin{equation*}
\label{eq:C->TF}
 \C(\spec{D};p)
\TO
\TF(\spec{D};p)
\,.
\end{equation*}

\begin{remark}
By \autoref{THH-Omega}\ref{i:THH-Omega-F-parts}, $\forget\sh^E\THH_\CF(\spec{A}[\oid{G}{S}])$ is an orthogonal $C_{p^n}$-$\Omega$-spectrum for each~$n$, provided that $\spec{A}$ is \cplus.
Therefore its $C_{p^n}$-fixed points are orthogonal $\Omega$-spectra; e.g., see~\cite{RV}*{Lemma~6.20(ii) on page~1522}.
It follows that also $(\sh^E\THH(\spec{A}[\oid{G}{S}]))^{C_{p^{n}}}$ is a fibrant $\CW$-space by~\cite{Blumberg}*{Theorem~1.3(2) on page~2263},
and therefore no fibrant replacement would be needed in defining the homotopy limit.
\end{remark}

We are now ready to prove our main splitting result, \autoref{THH-and-C}\ref{i:C} and \autoref{splitting}.
The deduction of this result from~\autoref{split-THH-sh-fix} is based on the main theorem of~\cite{LRV}.

\begin{theorem}
\label{split-C}
Assume that the symmetric ring spectrum $\spec{A}$ is \cplus.
Let $\CF\subseteq\Fin$ be a family of finite subgroups of~$G$ and let $N\geq0$.
Assume that the following condition holds.
\begin{enumerate}[label=\({[}\Alph*'_{\CF,\,\le N+1}{]}\)]
\item
For every~$C\in\CF$ and for every $1\leq s\leq N+1$, the integral group homology $H_s(BZ_GC;\IZ)$ of the centralizer of~$C$ in~$G$ is an almost finitely generated abelian group.
\end{enumerate}
Then the assembly maps
\[
EG(\CF)_+\sma_{\Or G}\C(\spec{A}[\oid{G}{-}];p)
\xrightarrow{\ \asbl\ }
\C(\spec{A}[G];p)
\]
for all primes~$p$ and
\[
EG(\CF)_+\sma_{\Or G}\Bigl(\THH(\spec{A}[\oid{G}{-}])\!\times\!\!\prodp\!\!\C(\spec{A}[\oid{G}{-}];p)\Bigr)
\xrightarrow{\asbl}
\THH(\spec{A}[G])\!\times\!\!\prodp\!\!\C(\spec{A}[G];p)
\]
are $\pi_n^\IQ$-injective for all~$n\leq N$.
\end{theorem}

\begin{proof}
We first prove the theorem under the additional assumption that $\CF$ contains only finitely many conjugacy classes.
We begin with~$\C(\spec{A}[G];p)$.
Consider the following commutative diagram in~$\CW\CT$.
\[
\begin{tikzcd}[column sep=4em]
\ds EG(\CF)_+\sma_{\Or G}\C(\spec{A}[\oid{G}{-}];p)
\arrow[r, "\asbl"]
\arrow[d, "t"']
&
\C    (\spec{A}[G];p)
\arrow[d, equal]
\\
\ds\holim_{m\in\IN}\Bigl(
EG(\CF)_+\sma_{\Or G}\bigl(\sh^{ES^1_+}\THH    (\spec{A}[\oid{G}{-}])\bigr)^{C_{p^m}}
\Bigr)
\arrow[r, "\holim(\asbl)", "\ts\three"']
\arrow[d, "\holim(\id\sma\pr_\CF)"', "\ts\one"]
&
\C    (\spec{A}[G];p)
\arrow[d, "\pr_\CF"]
\\
\ds\holim_{m\in\IN}\Bigl(
EG(\CF)_+\sma_{\Or G}\bigl(\sh^{ES^1_+}\THH_\CF(\spec{A}[\oid{G}{-}])\bigr)^{C_{p^m}}
\Bigr)
\arrow[r, "\holim(\asbl)"', "\ts\two"]
&
\C_\CF(\spec{A}[G];p)
\end{tikzcd}
\]
The bottom square is obtained by taking the homotopy limit of the diagrams from \autoref{split-THH-sh-fix}.
Therefore the maps $\one$ and~$\two$ are $\pi_*$-isomorphisms, and so $\three$ is $\pi_*$-injective.

The map~$t$ is the natural map that interchanges the order of smashing over~$\Or G$ and taking $\holim_\IN$.
We proceed to verify the assumptions (A) to~(D) of~\cite{LRV}*{Addendum~1.3 on page~140}, which imply that~$\pi_n(t)$ is an almost isomorphism, and so in particular a rational isomorphism, for all~$n\leq N$.
For assumption~(A), the category~$\CC=\IN$ has an obvious $1$-dimensional model for~$E\IN$; compare~\cite{LRV}*{Example~7.1 on page~162}.
For~(B), \autoref{THH-Adams} and the fact that homotopy orbits preserve connectivity imply that $(\sh^{ES^1_+}\THH(\spec{A}[\oid{G}{-}]))^{C_{p^m}}$ is always $(-1)$-connected.
For $X=EG(\CF)$, (C) is implied by our assumptions on~$\CF$.
And finally (D) follows from our assumptions on~$H_s(BZ_GH;\IZ)$ combined with~\cite{LRV}*{Proposition~1.7 on page~142}.

A similar argument applies to
\[
\MOR{t}
{EG(\CF)_+\sma_{\Or G}\smallprod_p\C(\spec{A}[\oid{G}{-}];p)}
{\smallprod_p\Bigl(EG(\CF)_+\sma_{\Or G}\C(\spec{A}[G];p)\Bigr)}
\]
This time the category~$\CC$ is the set of all prime numbers considered as a discrete category, and the constant functor~$\pt$ clearly satisfies assumption~(A).

Next we explain how to drop the assumption that $\CF$ contains only finitely many conjugacy classes.
Consider the set~$\CJ$ of subfamilies $\CE\subset\CF$ that are finite up to conjugacy.
Notice that $\CJ$ is a directed poset with respect to inclusion, and that
$\CF=\bigcup_\CJ \CE$.
Appropriate models for~$EG(\CE)$ yield a functor from $\CJ$ to~$G$-spaces and a $G$-equivalence
\begin{equation}
\label{eq:EGF-as-hocolim}
\hocolim_{\CE\in\CJ}EG(\CE)
\TO[\simeq]
EG(\CF)
\,.
\end{equation}
Now there is the following commutative diagram, as well as a corresponding diagram for $\THH(\spec{A}[G])\!\times\!\smallprod_p\!\C(\spec{A}[G];p)$.
\[
\begin{tikzcd}[column sep=-2em, row sep=large]
&
\ds\hocolim_\CJ\Bigl(EG(\CE)_+\sma_{\Or G}\C(\spec{A}[\oid{G}{-}];p)\Bigr)
\arrow[rd, "\hocolim(\asbl_\CE)"]
\arrow[ld, "\ts\four"']
\\
\ds EG(\CF)_+\sma_{\Or G}\C(\spec{A}[\oid{G}{-}];p)
\arrow[rr, "\asbl_\CF"']
&&
\ds\C(\spec{A}[G];p)
\end{tikzcd}
\]
Here~$\four$ is induced by the equivalence in~\eqref{eq:EGF-as-hocolim} and the fact that homotopy colimits commute up to isomorphism with smash products over the orbit category, and therefore $\four$ is a $\pi_*$-isomorphism.
Since the functors $\pi_*(-)$ and $-\otimes_\IZ\IQ$ both commute with directed (homotopy) colimits, and directed colimits are exact, it follows from the argument above that $\hocolim(\asbl_\CE)$ is injective on~$\pi_n(-)\tensor_\IZ\IQ$ for all $n\leq N$.
Therefore the same is true for~$\asbl_\CF$, as claimed.
\end{proof}

The proof of the Splitting \autoref{splitting} is now complete.
In fact, Theorems~\ref{split-THH} and~\ref{split-C} establish more general splitting results than what is needed for the \autoref{main-technical}.
In \autoref{no-R-revisited} below we explain why our method of proof cannot be used to obtain similar splitting results for topological cyclic homology without additional assumptions on the group and the family.


\section{Topological cyclic homology}
\label{TC}

Topological cyclic homology of a symmetric spectral category~$\spec{D}$ at a prime~$p$ is defined as
\[
\TC(\spec{D};p)=\holim_{n\in\obj\RFcat}\THH(\spec{D})^{C_{p^n}}
.
\]
Here $\RFcat$ is the category with set of objects~$\IN$ and where the morphisms from $m$ to~$n$ are the pairs $(i,j)\in\IN\times\IN$ with $i+j=m-n$.
The functoriality in $\RFcat$ comes from the fact that the maps $\MOR{R,F}{\THH(\spec{D})^{C_{p^{n}}}}{\THH(\spec{D})^{C_{p^{n-1}}}}$ commute; see diagram~\eqref{eq:R-and-F-commute}.
So we get a functor by sending the morphisms $(0,1)$ to the $R$-maps, and the morphisms~$(1,0)$ to the $F$-maps.
By forgetting the $R$-maps, we get a natural map
\begin{equation}
\label{eq:TC->TF}
\TC(\spec{D};p)\TO\TF(\spec{D};p)
\,,
\end{equation}
whose target is defined in~\eqref{eq:TF}.

\begin{remark}
\label{no-R-revisited}
In general, for $\spec{D}=\spec{A}[\oid{G}{S}]$, there is no $R$-map between the fixed points of $\THH_\CF(\spec{A}[\oid{G}{S}])$; see~\autoref{no-R}.
Therefore, without additional assumptions on $G$ and~$\CF$, we cannot define $\TC_\CF(\spec{A}[\oid{G}{S}];p)$ and retraction maps as in~\eqref{eq:retraction-C}, and prove a splitting result for~$\TC$ analogous to \autoref{split-C}.
To see why our method works for $\C$ but not for~$\TC$, consider the following diagram, where we use the abbreviation $\BT=\THH(\spec{A}[\oid{G}{-}])$ and $E=ES^1_+$.
\[
\begin{tikzcd}[column sep=0]
&
\BT_{\CF,h\Cp{n}}
\arrow[leftrightarrow, rr, "\cdot" description, pos=.3, shorten >=1.75em]
&
&
\mathllap{\sh^E}{\BT_\CF}^{\Cp{n}}
\arrow[rr]
\arrow[dd]
&
&
{\BT_\CF}^{\Cp{n}}
\arrow[dotted, rr, "\text{\raisebox{1.5ex}{\tiny\dbend}}" description]
\arrow[dd]
&
&
{\BT_\CF}^{\Cp{n-1}}
\arrow[dd]
\\
\mspace{20mu}\BT_{h\Cp{n}}
\arrow[leftrightarrow, rr, "\cdot" description, pos=.32, shorten >=1.75em]
\arrow[ur] 
&
&
\mathllap{\sh^E}{\BT}^{\Cp{n}}
\arrow[rr, crossing over]
\arrow[ur] 
\arrow[phantom, ul, "\ts\bone"]
\arrow[phantom, dr, "\ts\btwo"]
&
&
\BT^{\Cp{n}}
\arrow[rr, crossing over]
\arrow[ur] 
&
&
\BT^{\Cp{n-1}}
\arrow[ur] 
\\
&
&
&
\mathllap{\sh^E}{\BT_\CF}^{\Cp{n-1}}
\arrow[rr]
&
&
{\BT_\CF}^{\Cp{n-1}}
\arrow[dotted, rr]
&
&
{\BT_\CF}^{\Cp{n-2}}
\\
&
&
\mathllap{\sh^E}{\BT}^{\Cp{n-1}}
\arrow[rr]
\arrow[ur] 
\arrow[leftarrow, uu]
&
&
\BT^{\Cp{n-1}}
\arrow[rr]
\arrow[ur] 
\arrow[leftarrow, uu, crossing over]
&
&
\BT^{\Cp{n-2}}
\arrow[ur] 
\arrow[leftarrow, uu, crossing over]
\end{tikzcd}
\]
The diagonal maps are the projections to the $\CF$-parts.
The vertical maps are the inclusions of fixed points~$F$.
The commutative diagram~$\bone$ comes from \autoref{THH-Adams}, where the symbol $\leftarrow\cdot\rightarrow$ indicates a zig-zag of natural $\pi_*$-isomorphisms, and it allows us to deduce splitting results for~$(\sh^E\BT)^{\Cp{n}}$ from the ones for~$\BT$; see \autoref{split-THH-sh-fix}.

On the front face, the other horizontal maps are the homotopy fibration sequences from \autoref{hofib-R}.
However, notice that in the back, for the $\CF$-parts, the $R$-maps are missing, and we do not have homotopy fibration sequences.
Therefore we cannot argue by induction and obtain splitting results for the fixed points~$\BT^{\Cp{n}}$, and then pass to the homotopy limit to get results for~$\TC$.

On the other hand, since $\btwo$ obviously commutes, we can pass to the homotopy limit over~$n$ with respect to the $F$-maps, and (using the results from~\cite{LRV}) deduce splitting results for~$\C$ from the ones for~$(\sh^E\BT)^{\Cp{n}}$; see \autoref{split-C}.
Notice that we do not need any compatibility between the squares $\bone$ and~$\btwo$, as they enter the proof at separate stages.
The missing cube on the left and its commutativity are not needed in the proof.

Notice also that if we simply took $\hofib R$ instead of~$(\sh^E\BT)^{\Cp{n}}$, we would not even obtain the existence, let alone the commutativity, of the corresponding diagram~$\btwo$, since there are no $R$-maps for the $\CF$-parts.
This shows the importance of the point-set model for $\hofib R$ given in \autoref{hofib-R}.

In the case of spherical coefficients~$\spec{A}=\spec{S}$, the $R$-map splits and the splitting induces a map between the $\CF$-parts, but other difficulties arise, as explained in \autoref{S_F}.

Finally, we mention that in a separate paper~\cite{tc} we show that under additional assumptions on $G$ and~$\CF$ it is possible to fill in the missing $R$-maps in the diagram above, thus obtaining splitting results for~$\TC$, too.

This remark represents years of discussions among the authors.
\end{remark}


\section{Spherical coefficients}
\label{A=S}

We now analyze the special case~$\spec{D}=\spec{S}[\CC]$ for a small category~$\CC$.
In this case, the key observation is that the $R$-maps split.
This result is well known; compare \cite{BHM}*{proof of Theorem~5.17, pages~500--501} and~\cite{Dundas}*{Lemma~6.2.5.1 on page~245}.
We give a proof below to highlight that the splitting is natural in~$\CC$.
We then use this fact to construct a zig-zag of natural transformations from $\TC(\IS[\CC];p)$ to $\C(\IS[\CC];p)$ where the wrong-way maps are all $\pi_*$-isomorphisms.

\begin{proposition}
\label{R-and-S}
If $\spec{D}=\spec{S}[\CC]$ for a small category $\CC$, then the $R$-map splits, i.e., there is a map of $\CW$-spaces
\[
\MOR{S}{\THH(\spec{S}[\CC])^{C_{p^{n-1}}}}{\THH(\spec{S}[\CC])^{C_{p^{n}}}}
\]
such that~$R \circ S=\id$.
The map $S$ is natural in $\CC$.
The maps $S$ and $F$ commute.
\end{proposition}

\begin{proof}
We split the map $r$ in~\eqref{eq:res-to-fix} by a map $s$ that is natural in $A\in\obj\CW$ and $\vec{x}\in\obj\CI^{[q]}$.
As $[q]$ varies, these maps~$s$ assemble to a cyclic map that after homotopy colimits and realizations gives rise to $S$.

For the tuple of finite sets $\vec{x}=(x_0 , \ldots , x_q) \in\obj\CI^{[q]}$ let $|\vec{x}|=x_0 \sqcup \ldots \sqcup x_q$ denote the concatenation.
There is an orthogonal decomposition of $C_{p^n}$-representations
\[
\IR[C_{p^n}] = U \oplus V,
\]
where $\IR[C_{p^n}]$ denotes the regular $C_{p^n}$-representation, $V = \IR[ C_{p^n} ]^{C_p} $ its $C_{p}$-fixed points and $U$ the orthogonal complement.
Since there is an isomorphism of $C_{p^n}$-sets $|p^{n} \vec{x}| \cong C_{p^n} \times \real{\vec{x}}$,
the obvious isomorphism
\[
cn_{[q]} \spec{S} ( \vec{x} ) = S^{\IR[x_0]} \sma \ldots \sma S^{\IR[x_q]} \cong
S^{\IR [ \real{\vec{x}} ]}
\]
generalizes to a $C_{p^n}$-equivariant isomorphism
\[
cn_{p^n [q]} \spec{S} ( p^n \vec{x} )
\cong S^{U \otimes \IR [ \real{\vec{x}} ]} \sma S^{V \otimes \IR [ \real{\vec{x}} ]}.
\]
Using \eqref{eq:cn(A[C])} the inclusion of $C_p$-fixed points
\begin{align*}
\left( cn_{p^n[q]} \spec{S} (p^n \vec{x}) \right)^{C_p}
&\TO
cn_{p^n[q]} \spec{S} (p^n \vec{x} )
\\
\left( A \sma cn_{p^n[q]} \spec{S} [\CC](p^n \vec{x}) \right)^{C_p}
&\TO
A \sma cn_{p^n[q]} \spec{S} [\CC](p^n \vec{x} )
\end{align*}
can be identified with the inclusions
\begin{align*}
S^{V \otimes \IR [ \real{\vec{x}} ]}
&\TO
S^{U \otimes \IR [ \real{\vec{x}} ]} \sma S^{V \otimes \IR [ \real{\vec{x}} ]}
\\
S^{V \otimes \IR [ \real{\vec{x}} ]} \sma CN_{p^nq} (\CC)_+^{C_p} \sma A
&\TO
S^{U \otimes \IR [ \real{\vec{x}} ]} \sma S^{V \otimes \IR [ \real{\vec{x}} ]} \sma CN_{p^n q} (\CC)_+ \sma A
\, .
\end{align*}
Now in order to split the upper map in~\eqref{eq:res-to-fix} we  extend each $C_{p^n} / C_p$-equivariant map
\[
S^{V \otimes \IR [ \real{\vec{x}} ]} \TO
S^{V \otimes \IR [ \real{\vec{x}} ]} \sma  CN_{p^nq} (\CC)_+^{C_p} \sma A
\]
to the $C_{p^n}$-equivariant map
\[
S^{U \otimes \IR [ \real{\vec{x}} ]}  \sma S^{V \otimes \IR [ \real{\vec{x}} ]} \TO
S^{U \otimes \IR [ \real{\vec{x}} ]}  \sma S^{V \otimes \IR [ \real{\vec{x}} ]} \sma CN_{p^nq} (\CC)_+ \sma A
\]
obtained by smashing with the identity on $S^{U \otimes \IR [ \real{\vec{x}} ]}$ and composing with the map induced from the inclusion
$CN_{p^nq} (\CC)^{C_p} \TO CN_{p^nq} (\CC) $.
Composing with the homeomorphism~$\Delta$ in~\eqref{eq:res-to-fix} we obtain the desired map~$s$.
\end{proof}

\begin{warning}
\label{S_F}
As opposed to the case of the $R$-maps explained in \autoref{no-R}, it is in fact possible to define a map
\[
\MOR{S_\CF}{\THH_\CF(\spec{S}[\oid{G}{S}])^{C_{p^{n-1}}}}{\THH_\CF(\spec{S}[\oid{G}{S}])^{C_{p^{n}}}}
\,,
\]
and we have $S_\CF\circ F=F\circ S_\CF$.
However, the diagrams
\[
\begin{tikzcd}[row sep=scriptsize]
\THH(\spec{S}[\oid{G}{S}])^{C_{p^{n-1}}}
\arrow[r, "S"]
\arrow[d, "\pr_\CF"']
&
\THH(\spec{S}[\oid{G}{S}])^{C_{p^{n}}}
\arrow[d, "\pr_\CF"]
\\
\THH_\CF(\spec{S}[\oid{G}{S}])^{C_{p^{n-1}}}
\arrow[r, "S_\CF"']
&
\THH_\CF(\spec{S}[\oid{G}{S}])^{C_{p^{n}}}
\end{tikzcd}
\]
and
\[
\begin{tikzcd}[row sep=scriptsize]
\THH(\spec{S}[\oid{G}{S}])
\arrow[r, "F\circ S"]
\arrow[d, "\pr_\CF"']
&
\THH(\spec{S}[\oid{G}{S}])
\arrow[d, "\pr_\CF"]
\\
\THH_\CF(\spec{S}[\oid{G}{S}])
\arrow[r, "F\circ S_\CF"']
&
\THH_\CF(\spec{S}[\oid{G}{S}])
\end{tikzcd}
\]
do not commute without additional assumptions on $G$ and~$\CF$, not even in the stable homotopy category.
It seems to us that the commutativity of the latter diagram in the case $\CF=1$ is used implicitly in the survey~\cite{Madsen}*{proof of Theorem~4.5.2 on pages 283--284}.
The original proof of that theorem~\cite{BHM}*{Theorem~9.13 on page~535} avoids this problem.
\end{warning}

Now consider the following diagram for $n\ge2$.
\[
\hspace{-1ex}\begin{tikzcd}[row sep=0ex]
\ds\bigl(\sh^{ES^1_+}\THH(\spec{S}[\CC])\bigr)^{C_{p^{n  }}}
\arrow[rrd, "\pr_*", shorten <=-1.75em]
\arrow[dddd, "F_{n  }", shift right=1em]
\\
&
&
\ds                  \THH(\spec{S}[\CC])^{C_{p^{n  }}}
\arrow[r, "R_{n  }", shorten <=-.25em]
\arrow[ld, "q_{n  }", shift right]
\arrow[dddd, "F_{n  }", shift right=1em]
&
\ds                  \THH(\spec{S}[\CC])^{C_{p^{n-1}}}
\arrow[l, "S_{n  }", pos=.54, shorten >=-.5em, bend left=18, shorten >=-.5em]
\arrow[dddd, "F_{n-1}"]
\\
&
\mathllap{\ds\hocofib(S_{n  })}
\\
\vphantom{x}
\\
\ds\bigl(\sh^{ES^1_+}\THH(\spec{S}[\CC])\bigr)^{C_{p^{n-1}}}
\arrow[rrd, "\pr_*", shorten <=-1.75em]
\\
&
&
\ds                  \THH(\spec{S}[\CC])^{C_{p^{n-1}}}
\arrow[r, "R_{n-1}", shorten <=-.25em, pos=.4]
\arrow[ld, "q_{n-1}", shift right]
&
\ds                  \THH(\spec{S}[\CC])^{C_{p^{n-2}}}
\arrow[l, "S_{n-1}", pos=.54, shorten >=-.5em, bend left=18, shorten >=-.5em]
\\
&
\mathllap{\ds\hocofib(S_{n-1})}
\arrow[leftarrow, uuuu, crossing over, shift left=2.5em, pos=.6, "f_{n  }"']
\end{tikzcd}
\]
Here $q_n$ is the natural map to the homotopy cofiber.
By \autoref{hofib-R}, the sequence with $\pr_*$ and~$R_n$ is a stable homotopy fibration sequence.
So \autoref{R-and-S} implies that the composition~$q_n\circ\pr_*$ is a $\pi_*$-isomorphism.
As $n$ varies, the maps $R_n$ and~$S_n$ commute with the inclusions~$F_n$.
Therefore there are induced vertical maps~$f_n$ and all the squares in the diagram above commute.
Now let
\[
\C'(\spec{S}[\CC];p)=\holim_{n\ge2}\bigl(\:
\dotsb
\TO
\hocofib(S_{n  })
\TO[f_{n  }]
\hocofib(S_{n-1})
\TO
\dotsb
\:\bigr)
\,.
\]
Then the maps~$q_n$ induce a map
\(
\TF(\spec{S}[\CC];p)
\TO
\C'(\spec{S}[\CC];p)
\)
such that the composition
\begin{equation}
\label{eq:C->C'}
\C (\spec{S}[\CC];p)
\TO
\TF(\spec{S}[\CC];p)
\TO
\C'(\spec{S}[\CC];p)
\end{equation}
is a natural $\pi_*$-isomorphism.

Hence we get a zig-zag of natural transformations of functors $\Cat\TO\CW\CT$
\begin{equation}
\label{eq:alpha}
\TC(\spec{S}[\CC];p)
\TO
\TF(\spec{S}[\CC];p)
\TO
\C'(\spec{S}[\CC];p)
\xleftarrow{\:\simeq\:}
\C (\spec{S}[\CC];p)
\,,
\end{equation}
where the first map comes from~\eqref{eq:TC->TF} and the last map is the natural $\pi_*$\=/isomorphism~\eqref{eq:C->C'} above.
The natural map induced in the stable homotopy category by the zig-zag~\eqref{eq:alpha} is denoted~$\alpha$.

The rest of this section is devoted to the following result, which generalizes from groups to arbitrary small categories a fundamental theorem  due to~\cite{BHM}*{Theorem~5.17 on page~500}; see also~\cite{Rognes-2primary}*{Theorem~1.16 on page~885} and~\cite{Dundas}*{Section~6.4.2.3 on pages~259--261}.

\begin{theorem}
\label{TC(SC)}
For any small category~$\CC$ and any prime~$p$, in the stable homotopy category there is a homotopy cartesian square
\begin{equation}
\label{eq:TC(SC)}
\begin{tikzcd}[column sep=large]
 \TC(\spec{S}[\CC];p)
\arrow[r, "\alpha"]
\arrow[d, "\pr"']
&
  \C(\spec{S}[\CC];p)
\arrow[d, "\pr"]
\\
\THH(\spec{S}[\CC])
\arrow[r, "1-F_1\circ S_1"']
&
\THH(\spec{S}[\CC])
\mathrlap{\,.}
\end{tikzcd}
\end{equation}
\end{theorem}

\begin{proof}
The proof in~\cite{BHM}*{Section~5} can be made more precise with our point-set models for all the relevant spectra and maps.
In detail, from \autoref{R-and-S} and \autoref{hofib-R} we get inductively $\pi_*$-isomorphisms
\[
\MOR{\rho_n}
{\bigl(\sh^E\T\bigr)^{\Cp{n}}\vee\bigl(\sh^E\T\bigr)^{\Cp{n-1}}\vee\!\dotsb\!\vee\bigl(\sh^E\T\bigr)^{\Cp{}}\vee\THH(\spec{S}[\CC])}
{\THH(\spec{S}[\CC])^{\Cp{n}}}
,
\]
where we use the abbreviation $\sh^E\T=\sh^{ES^1_+}\THH(\spec{S}[\CC])$.
Here $\rho_n$ is the sum of the maps
$\MOR{S^{n-m}\circ \pr_*}{(\sh^E\T)^{C_{p^m}}}{\THH(\spec{S}[\CC])^{C_{p^n}}}$
for $1 \le m \le n$ and
$\MOR{S^n}{\THH(\spec{S}[\CC])}{\THH(\spec{S}[\CC])^{C_{p^n}}}$.

Moreover, the following diagrams commute.
\[
\begin{tikzcd}[column sep=0]
\ds\bigl(\sh^E\T\bigr)^{\Cp{n}}
\arrow[phantom, r, "\vee"]
\arrow[dr, "F_n" description]
&
\ds\bigl(\sh^E\T\bigr)^{\Cp{n-1}}
\arrow[phantom, r, "\vee"]
&
\dotsb
\arrow[phantom, r, "\vee"]
\arrow[phantom, d, "\ddots" pos=.4]
&
\ds\bigl(\sh^E\T\bigr)^{\Cp{}}
\arrow[phantom, r, "\vee"]
\arrow[dr, "\pr_*\!\circ F_1" description]
&
\THH(\spec{S}[\CC])
\arrow[rr, "\rho_n"]
\arrow[d, "F_1\circ S_1" description]
&
\mspace{25mu}
&
\ds\THH(\spec{S}[\CC])^{\Cp{n}}
\arrow[d, "F_n"', shift right=3]
\\
&
\ds\bigl(\sh^E\T\bigr)^{\Cp{n-1}}
\arrow[phantom, r, "\vee"]
&
\dotsb
\arrow[phantom, r, "\vee"]
&
\ds\bigl(\sh^E\T\bigr)^{\Cp{}}
\arrow[phantom, r, "\vee"]
&
\THH(\spec{S}[\CC])
\arrow[rr, "\rho_{n-1}"']
&
&
\ds\THH(\spec{S}[\CC])^{\Cp{n-1}}
\\
\ds\bigl(\sh^E\T\bigr)^{\Cp{n}}
\arrow[phantom, r, "\vee"]
&
\ds\bigl(\sh^E\T\bigr)^{\Cp{n-1}}
\arrow[phantom, r, "\vee"]
\arrow[d, "\id" description, shift right=1em, pos=.535]
&
\dotsb
\arrow[phantom, r, "\vee"]
\arrow[phantom, d, "\vdots" pos=.4]
&
\ds\bigl(\sh^E\T\bigr)^{\Cp{}}
\arrow[phantom, r, "\vee"]
\arrow[d, "\id" description, shift right=.5em, pos=.535]
&
\THH(\spec{S}[\CC])
\arrow[rr, "\rho_n"]
\arrow[d, "\id" description]
&
\mspace{25mu}
&
\ds\THH(\spec{S}[\CC])^{\Cp{n}}
\arrow[d, "R_n"', shift right=3]
\\
&
\ds\bigl(\sh^E\T\bigr)^{\Cp{n-1}}
\arrow[phantom, r, "\vee"]
\arrow[u, bend right, shift left=1em, "\id" description]
&
\dotsb
\arrow[phantom, r, "\vee"]
&
\ds\bigl(\sh^E\T\bigr)^{\Cp{}}
\arrow[phantom, r, "\vee"]
\arrow[u, bend right, shift left=.45em, "\id" description]
&
\THH(\spec{S}[\CC])
\arrow[rr, "\rho_{n-1}"']
\arrow[u, bend right, shift right=.07em, "\id" description]
&
&
\ds\THH(\spec{S}[\CC])^{\Cp{n-1}}
\arrow[u, bend right, shift left=2, "S_n"']
\end{tikzcd}
\]
The downward vertical map on the left in the second diagram projects the first summand to the base point.

For each $n\ge1$ consider the following commutative diagram in $\CW\CT$, with horizontal stable homotopy (co)fibration sequences.
\[
\begin{tikzcd}[column sep=-.5em]
\THH(\spec{S}[\CC])
\arrow[rr, "S^{n+1}"]
\arrow[equal, dr]
\arrow[dd, "F_1\circ S_1" description, near start]
&&
\THH(\spec{S}[\CC])\mathrlap{^{C_{p^{n+1}}}}
\arrow[rr, shorten <=2em]
\arrow[dr, "R_{n+1}", pos=.66]
\arrow[dd, "F_{n+1}" description, near start]
&&
\hocofib{(S^{n+1})}\hspace{-2.5em}
\arrow[dd, "\overline{F}_{n+1}" description, near start]
\arrow[dr, "\overline{R}_{n+1}", pos=.66]
\\
&
\THH(\spec{S}[\CC])
\arrow[rr, crossing over, "S^n"', near start]
&&
\THH(\spec{S}[\CC])\mathrlap{^{C_{p^n}}}
\arrow[rr, crossing over, shorten <=1.1em]
&&
\hocofib{(S^{n})}
\arrow[dd, "\overline{F}_n" description, near start]
\\
\THH(\spec{S}[\CC])
\arrow[rr, "S^n"', near start]
\arrow[equal, dr]
&&
\THH(\spec{S}[\CC])\mathrlap{^{C_{p^n}}}
\arrow[rr, shorten <=1.1em]
\arrow[dr, "R_n", pos=.66]
&&
\hocofib{(S^{n})}\hspace{-2.5em}
\arrow[dr, "\overline{R}_n", pos=.66]
\\
&
\THH(\spec{S}[\CC])
\arrow[rr, "S^{n-1}"']
\arrow[uu, leftarrow, crossing over, "F_1\circ S_1" description, near end]
&&
\THH(\spec{S}[\CC])\mathrlap{^{C_{p^{n-1}}}}
\arrow[rr, shorten <=2em]
\arrow[uu, leftarrow, crossing over, "F_n" description, near end]
&&
\hocofib{(S^{n-1})}
\end{tikzcd}
\]
Passing to homotopy limits over the $R$-maps, we get another commutative diagram with horizontal stable homotopy fibration sequences.
\[
\begin{tikzcd}
\THH(\spec{S}[\CC])
\arrow[r, "S^\infty"]
\arrow[d, "F_1\circ S_1"']
&
\TR(\spec{S}[\CC]; p)
\arrow[r]
\arrow[d, "F"']
&
\hocofib(S^\infty)
\arrow[d, "\overline{F}"']
\\
\THH(\spec{S}[\CC])
\arrow[r, "S^\infty"]
&
\TR(\spec{S}[\CC]; p)
\arrow[r]
&
\hocofib(S^\infty)
\end{tikzcd}
\]
Here
\[
\TR(\spec{S}[\CC];p)=\holim_{n\in\IN}\bigl(\:
\dotsb
\TO
\THH(\spec{S}[\CC])^{C_{p^n}}
\TO[\!R_n]
\THH(\spec{S}[\CC])^{C_{p^{n-1}}}
\TO
\dotsb
\:\bigr)
\]
is defined analogously to $\TF(\spec{S}[\CC]; p)$, replacing the Frobenius $F$-maps with the restriction $R$-maps; compare~\eqref{eq:TF} and~\cite{HM-top}*{(20) on page~56}.
Taking homotopy fixed points for the vertical self-maps, or equivalently their homotopy equalizers with the identity maps, gives the stable homotopy fibration sequence
\[
\hoeq(\id, F_1\circ S_1)
\TO[S^\infty]
\TC(\spec{S}[\CC]; p)
\TO[\alpha]
\hoeq(\id, \overline{F})
\,.
\]
Here we use the well known equivalent ways to compute homotopy limits over~$\RFcat$; e.g., see~\cite{Dundas}*{Section~6.4.2.1 on pages~255--256}.

The $\pi_*$-isomorphisms $\rho_n$ induce compatible $\pi_*$-isomorphisms
\[
\MOR{\overline{\rho}_n}
{(\sh^E\T)^{C_{p^n}} \vee \dots \vee (\sh^E\T)^{C_p}}
{\hocofib(S^n)}
\]
with homotopy limit
\[
\MOR{\overline{\rho}_\infty}
{\holim_{n \in \IN} \bigl(
(\sh^E\T)^{C_{p^n}} \vee \dots \vee (\sh^E\T)^{C_p}
\bigr)}
{\hocofib(S^\infty)}
\,,
\]
and the inclusion of wedge sums into products induces a $\pi_*$-isomorphism
\[
\holim_{n \in \IN} \bigl(
(\sh^E\T)^{C_{p^n}} \vee \dots \vee (\sh^E\T)^{C_p}
\bigr)
\TO
\smallprod_{n=1}^\infty (\sh^E\T)^{C_{p^n}}
\,,
\]
such that $\overline{F}$ corresponds to the product of the $\MOR{F_n}{(\sh^E\T)^{C_{p^n}}}{(\sh^E\T)^{C_{p^{n-1}}}}$.
Hence we have a natural zig-zag of $\pi_*$-isomorphisms between the target $\hoeq(\id, \overline{F})$ of $\alpha$ and
\[
\hoeq\bigl(\id, {\smallprod}_n F_n\bigr)
  = \holim_{n\ge1} \, (\sh^E\T)^{C_{p^n}} \simeq \C(\spec{S}[\CC]; p) \,,
\]
by the usual model for a sequential homotopy limit.

So far, this argument follows \cite{BHM}*{Section~5}, and establishes the left-hand vertical homotopy fiber sequence of~(5.19) on page~502 in that paper.
To construct the homotopy cartesian square~\eqref{eq:TC(SC)}, an additional argument is needed, which we now describe.


Restricting the homotopy limit defining $\TC(\spec{S}[\CC]; p)$ to the full subcategory of $\RFcat$ generated by the objects $\{0,1\} \subset \IN$, we get the middle vertical map in the following commutative diagram in $\CW\CT$.
\begin{equation}
\label{jr-upper}
\begin{tikzcd}
\hoeq(\id, F_1\circ S_1)
\arrow[r, "S^\infty"]
\arrow[equal, d]
&
\TC(\spec{S}[\CC]; p)
\arrow[r, "\alpha"]
\arrow[d, "\res"']
&
\hoeq(\id, \overline{F})
\simeq\C(\spec{S}[\CC]; p)
\arrow[d, shift right=3em]
\\
\hoeq(\id, F_1\circ S_1)
\arrow[r, "S_1"']
&
\hoeq(R_1, F_1)
\arrow[r, "q_1"']
&
\hocofib(S_1)
\simeq(\sh^E\T)^{C_p}
\end{tikzcd}
\end{equation}
The right-hand vertical map factors in the canonical way through $\hocofib(S^\infty)$.
The left-hand lower map is induced by $\MOR{S_1}{\THH(\spec{S}[\CC])}{\THH(\spec{S}[\CC])^{C_p}}$ and the identity on $\THH(\spec{S}[\CC])$.
The right-hand lower map is induced by $\MOR{q_1}{\THH(\spec{S}[\CC])^{C_p}}{\hocofib(S_1)}$ and $\THH(\spec{S}[\CC])\TO\pt$.
Hence the lower row is a stable homotopy fibration sequence, and the right-hand square is homotopy cartesian in $\CW\CT$.

In the stable homotopy category, $\hoeq(R_1, F_1)$ is isomorphic to $\hofib(F_1-R_1)$, and $F_1-R_1$ corresponds under $\rho_1$ and $\rho_0$ to the difference of $\pr_*\!\circ F_1 + F_1\circ S_1$ and $0 + \id$, which equals the difference of $\pr_*\!\circ F_1 + 0$ and $0 + (\id - F_1\circ S_1)$, as maps $(\sh^E\T)^{C_p} \vee \THH(\spec{S}[\CC]) \TO \THH(\spec{S}[\CC])$.
Hence $\hofib(F_1-R_1)$ is also isomorphic to the homotopy pullback of these two maps, and we have a homotopy cartesian square
\begin{equation}
\label{jr-lower}
\begin{tikzcd}[column sep=large]
\hoeq(R_1, F_1)
\arrow[r, "q_1"]
\arrow[d, "R_1"']
&
(\sh^E\T)^{C_p}
\arrow[d, "\pr_*\!\circ F_1"]
\\
\THH(\spec{S}[\CC])
\arrow[r, "\id - F_1\circ S_1"']
&
\THH(\spec{S}[\CC])
\end{tikzcd}
\end{equation}
in the stable homotopy category.

Now combine the homotopy cartesian squares in \eqref{jr-upper} and~\eqref{jr-lower}.
The composition $R_1 \circ \res$ is the canonical projection $\MOR{\pr}{\TC(\spec{S}[\CC]; p)}{\THH(\spec{S}[\CC])}$, and the composition $\hoeq(\id, \overline{F}) \TO \hocofib(S_1) \simeq (\sh^E\T)^{C_p} \TO \THH(\spec{S}[\CC])$ corresponds to the canonical projection $\MOR{\pr}{\C(\spec{S}[\CC]; p)}{\THH(\spec{S}[\CC])}$.
So we obtain the asserted homotopy cartesian square~\eqref{eq:TC(SC)}.
\end{proof}

The following easy consequence of \autoref{TC(SC)} is very important for the proof of the Detection \autoref{detection}.

\begin{corollary}
\label{TC->TxC}
If $G$ is a finite group, then the map
\begin{equation}
\label{eq:(pr,alpha)}
\MOR{(\pr,\alpha)}
{\TC(\spec{S}[G];p)}
{\THH(\spec{S}[G])
\times
\C(\spec{S}[G];p)}
\end{equation}
is $\pi_0^\IQ$-injective and a $\pi_n^\IQ$-isomorphism for all~$n\ge1$.
\end{corollary}

\begin{proof}
Apply \autoref{TC(SC)} to $\CC=G$ and consider the Mayer-Vietoris sequence coming from the homotopy cartesian square~\eqref{eq:TC(SC)}.
From \autoref{THH-free-loop} we know that $\pi_n(\THH(\spec{S}[G]))\tensor_\IZ\IQ=0$ for all~$n\geq1$ if $G$ is finite, and the claim follows.
\end{proof}

\begin{remark}
In fact, it can be shown that the map~\eqref{eq:(pr,alpha)} is not $\pi^\IQ_{-1}$-injective.
\end{remark}


\section{Equivariant homology theories and Chern characters}
\label{CHERN}

In this section we review the notion of an equivariant homology theory and a very general way of producing examples; see \autoref{equiv-homolgy-theory}.
Next we recall under which conditions equivariant Chern characters can be used to compute rationalized equivariant homology theories; see \autoref{Chern} and \autoref{rel-Chern}.
These results are key to the proof of the Detection \autoref{detection} in \autoref{PROOF-DETECTION}.

Given a group~$G$, a $G$-homology theory~$\CH^G_*$ is a collection of functors $\CH^G_n$, one for every $n\in\IZ$, from the category of $G$-CW-complex pairs to abelian groups, satisfying the obvious $G$-equivariant analogues of the familiar axioms in the non-equivariant setting; see~\cite{L-Chern}*{Section~1, pages~197--198}.
An \emph{equivariant homology theory}~$\CH^?_*$ is a collection of $G$-homology theories~$\CH^G_*$, one for every group~$G$, together with induction isomorphisms
\begin{equation}
\label{eq:ind_alpha}
\MOR[\cong]{\ind_\alpha}{\CH^H_*(X,A)}{\CH^G_*(\ind_\alpha X,\ind_\alpha A)}
\end{equation}
for every group homomorphism $\MOR{\alpha}{H}{G}$ and every $H$-CW-complex pair $(X,A)$ on which $\ker\alpha$ acts freely.
The induction isomorphisms are required to be functorial in~$\alpha$ and to be compatible with conjugation and boundary homomorphisms; see~\cite{L-Chern}*{Section~1, pages~198--199}.

The following result provides a framework for constructing equivariant homology theories.
Recall the action groupoid functor $\MOR{\oid{G}{-}}{\Sets^G}{\Groupoids}$ from \autoref{ACTION-GROUPOID}.
Recall also that, given a functor $\MOR{\BE}{\Or G}{\IN\Sp}$, its continuous left Kan extension $\MOR{\BE_\%}{\Top^G}{\IN\Sp}$ along the inclusion $\Or G\hookrightarrow\Top^G$ is defined as follows.
For any $G$-space~$X$ let $\BE_\%(X)=X_+\sma_{\Or G}\BE$ be the coend of the functor
\[
(\Or G)^\op\times\Or G
\TO\IN\Sp
\,,
\quad
(G/H,G/K)\longmapsto X^H_+\sma\BE(G/K)
\,;
\]
compare \autoref{KAN}.

\begin{theorem}
\label{equiv-homolgy-theory}
Let $\MOR{\BE}{\Groupoids}{\IN\Sp}$ be a functor.
Assume that $\BE$ preserves equivalences, in the sense that $\BE$ sends equivalences of groupoids to $\pi_*$-isomorphisms.
For any group~$G$, consider the composition
\[
\Or G
\xrightarrow{\oid{G}{-}}
\Groupoids
\xrightarrow{\ \BE\ }
\IN\Sp
\,,
\]
and for any $G$-CW-pair $(X,A)$ define
\[
H^G_*(X,A;\BE)
=
\pi_*\bigl(\BE(\oid{G}{-})_\%(X/A)\bigr)
=
\pi_*\Bigl(X/A_+\sma_{\Or G}\BE(\oid{G}{-})\Bigr)
\,.
\]
Then $H^?_*(-;\BE)$ is an equivariant homology theory.
\end{theorem}

\begin{proof}
This is proved in~\cite{LR-survey}*{Proposition~157 on page~796}.
We explain here in a more conceptual way how the induction homomorphisms~\eqref{eq:ind_alpha} are defined.
Let $\MOR{\alpha}{H}{G}$ be a group homomorphism.
Define a functor $\MOR{\ind_\alpha}{\Or H}{\Or G}$ by $H/K\longmapsto G/\alpha(K)$.
Consider the following diagram.
\[
\begin{tikzcd}[column sep=tiny]
\Or H
\arrow[hook, rrr]
\arrow[drr, "\ind_\alpha" description]
\arrow[ddr, "\oid{H}{-}"']
&
&
&
\Top^H
\arrow[drr, "\ind_\alpha" description]
\\
\null
\arrow[Rightarrow, rr, shorten <=1.95em, shorten >=.05em, near end, "\overline{\alpha}"']
&
&
\Or G
\arrow[hook,rrr]
\arrow[dl, near start, "\oid{G}{-}"]
\arrow[Rightarrow, ur, shorten <=.75em, shorten >=.75em, "\nu"]
&
&
&
\Top^G
\\
&
\mathclap{\Groupoids}
\arrow[rrr, shorten <=2.5em, shorten >=1em, "\BE"', pos=.58]
&
&
&
\mathclap{\IN\Sp}
\end{tikzcd}
\]
Here $\overline{\alpha}$ is the natural transformation defined in~\eqref{eq:over-alpha}, and $\nu$ is the natural isomorphisms $G\times_H H/K\cong G/\alpha(K)$.
Then we get the following natural transformations of functors $\Top^H\TO\IN\Sp$
\[
\BE(\oid{H}{-})_\%
\TO
\BE(\oid{G}{\ind_\alpha-})_\%
\TO
\BE(\oid{G}{-})_\%\circ\ind_\alpha
\,.
\]
The first natural transformation is induced by~$\overline{\alpha}$, the second one by the version of \autoref{Marco's-trick} for topological categories.
Passing to homotopy groups, the composition of these natural transformations defines the induction homomorphisms~\eqref{eq:ind_alpha}.
\end{proof}

Next we want to recall some of the main results of~\cite{L-Chern} about equivariant Chern characters.
We first introduce some more notation.

Let $\CH^?_*$ be an equivariant homology theory.
Given a group~$G$, consider the category~$\Sub G(\Fin)$ defined right after \autoref{Whitehead}.
Then for any $n\in\IZ$ we get a functor
\begin{equation}
\label{eq:SubG->Ab}
\MOR{\CH^G_n(G/-)}{\Sub G(\Fin)}{\Ab}
\,,\qquad
H\longmapsto\CH^G_n(G/H)
\,;
\end{equation}
compare~\cite{L-Chern}*{paragraph after Lemma~3.7 on page~207}.
If $\CH^?_*=H^?_*(-;\BE)$ for some functor $\MOR{\BE}{\Groupoids}{\IN\Sp}$ as in \autoref{equiv-homolgy-theory}, then this is also proved in~\cite{LRV}*{Lemma~3.11 on page~152}.
Notice that in this case we have $H^G_*(G/H;\BE)\cong\pi_*\BE(H)$, where we identify any group~$H$ with the groupoid~$\oid{H}{\pt}$.

We denote by $\FGI$ the category of finite groups and injective group homomorphisms.
Given any equivariant homology theory~$\CH^?_*$ and any $n\in\IZ$ we get a functor
\begin{equation}
\label{eq:cov-obj}
\MOR{M_*}{\FGI}{\Ab}
\,,\qquad
G\longmapsto\CH^G_n(\pt)
\,,
\end{equation}
where for every injective group homomorphism $\MOR{\alpha}{H}{G}$ we define~$M(\alpha)$ as the composition
\begin{equation}
\label{eq:cov-mor}
M_*(H)=\CH^H_n(\pt)
\xrightarrow{\ind_\alpha}
\CH^G_n(G/\alpha(H))
\xrightarrow{\CH^G_n(\pr)}
\CH^G_n(\pt)=M_*(G)
\,.
\end{equation}
Notice that, if $\CH^?_*=H^?_*(-;\BE)$ for some functor $\MOR{\BE}{\Groupoids}{\IN\Sp}$ as in \autoref{equiv-homolgy-theory}, then the functor~$M_*$ coincides with the composition
\begin{equation}
\label{cov-special}
\FGI
\xhookrightarrow{\quad}
\Groupoids
\xrightarrow{\ \BE\ }
\IN\Sp
\xrightarrow{\:\pi_n\:}
\Ab
\,,
\qquad
G\longmapsto\pi_n\BE(G)
\,.
\end{equation}

A \emph{Mackey functor} consists of a pair of functors
\[
\MOR{M_*}{\FGI}{\Ab}
\AND
\MOR{M^*}{(\FGI)^\op}{\Ab}
\]
that agree on objects, i.e., $M_*(G)=M^*(G)$ for every finite group~$G$, and that satisfy the following axioms.
\begin{enumerate}
\item
If $\MOR{\alpha}{H}{G}$ is an isomorphism, then $M^*(\alpha)$ is the inverse of~$M_*(\alpha)$.
\item
If $\MOR{\alpha=\conj_g}{G}{G}$, $x\longmapsto gxg^{-1}$, for some~$g\in G$, then \[M_*(\conj_g)=\id_{M(G)}\,.\]
\item (Double coset formula).
If $H$ and~$K$ are subgroups of~$G$, 
let $\CJ=K\backslash G/H$.
For every $j\in\CJ$, choose a representative~$g_j\in j=KgH$, and define $L_j=H\cap g_j^{\smash{-1}}Kg_j$ 
and $\MOR{c_j=\conj_{g_j}}{L_j}{K}$.
Then
\[
M^*(K\hookrightarrow G)\circ M_*(H\hookrightarrow G)
=
\sum_{j\in\CJ} M_*(c_j)\circ M^*(L_j\hookrightarrow H)
\,.
\]
\end{enumerate}

We say that an equivariant homology theory~$\CH^?_*$ \emph{has a Mackey structure} if, for every $n\in\IZ$, the functor defined in \eqref{eq:cov-obj} and~\eqref{eq:cov-mor} extends to a Mackey functor.

A fundamental example of a Mackey functor is the \emph{Burnside ring}~$A$; see for example~\cite{tD-transf}*{(2.18) on pages~18--19}.
For every finite group~$G$, $A(G)$ is the Grothendieck ring of the set of isomorphism classes of finite $G$-sets, with addition given by the disjoint union and multiplication given by the cartesian product.
As an abelian group, $A(G)$ is the free abelian group with basis given by the isomorphism classes of the orbits~$G/H$.
There is an injective ring homomorphism
\[
\MOR{\varphi_G}{A(G)}{\smallprod_{(H)}\IZ}
\,,
\]
where the product is indexed over the conjugacy classes of subgroups of~$G$.
The composition of~$\varphi_G$ with the projection to the factor indexed by the conjugacy class of a subgroup~$H$ maps $S$ to~$\#(S^H)$, the number of elements in the $H$-fixed point set.
Moreover, $\varphi_G\tensor_\IZ\IQ$ is an isomorphism.
Now let $e_G\in\prod_{(H)}\IZ$ be the idempotent that corresponds to the factor indexed by the conjugacy class~$(G)$.
We then define the idempotent
\[
\Theta_G=\Bigl(\varphi_G\tensor_\IZ\IQ\Bigr)^{-1}(e_G\tensor1)
\in A(G)\tensor_\IZ\IQ
\]
in the rationalized Burnside ring of~$G$.

The Burnside ring~$A$ has a natural action on any other Mackey functor~$M$, in the precise sense of~\cite{L-Chern}*{Section~7, pages~222--223}.
In particular, the idempotent $\Theta_G\in A(G)\tensor_\IZ\IQ$ acts on~$M(G)\tensor_\IZ\IQ$ for every~$G$.

We are now ready to formulate the following result.

\begin{theorem}
\label{Chern}
Keep the notation and assumption of \autoref{equiv-homolgy-theory}.
Assume that the equivariant homology theory $H^?_*(-;\BE)$ has a Mackey structure, i.e., for every $n\in\IZ$ the composition~\eqref{cov-special} extends to a Mackey functor.
Then for every proper $G$-CW-complex~$X$ there are isomorphisms
\begin{equation}
\label{eq:Chern}
\begin{tikzcd}
\ds
\adjustlimits
\bigoplus_{s+t=n}
\bigoplus_{(H)\in(\Fin)}
H_s\bigl(X^H     /Z_G H;\IQ\bigr)\tensor_{\IQ[W_G H]}S_t(H;\BE)
\arrow[d, "\mathit{h}\,"']
\\
\ds
\bigoplus_{s+t=n}
H_s\bigl(X^{\ts-}/Z_G -;\IQ\bigr)\tensor_{\IQ[\Sub G(\Fin)]}\Bigl(\pi_t\BE(-)\tensor_\IZ\IQ\Bigr)
\arrow[d, "\mathit{c}\,"']
\\
\ds
H^G_n(X;\BE)\tensor_\IZ\IQ
\mathrlap{\,.}
\end{tikzcd}
\end{equation}
Their composition $\mathit{ch}=\mathit{c}\circ\mathit{h}$ is called Chern character.
Both $\mathit{h}$ and~$\mathit{c}$ are natural in~$X$, and $\mathit{h}$ respects the outer sum decompositions.
Here, for every finite subgroup $H$ of~$G$,
\begin{equation}
\label{eq:S_t(H;E)}
S_t(H;\BE)
=
\coker\Biggl(\;
{\bigoplus_{K\lneqq H}
 \pi_t\BE(K)\tensor_\IZ\IQ}
\TO
{\pi_t\BE(H)\tensor_\IZ\IQ}
\Biggr)
.
\end{equation}
Moreover, for every finite cyclic subgroup $C$ of~$G$ there is a natural isomorphism
\begin{equation}
\label{eq:idempotent}
S_t(C;\BE)
\cong
\Theta_C\Bigl(
\pi_t\BE(C)\tensor_\IZ\IQ
\Bigr)
.
\end{equation}
\end{theorem}

\begin{proof}
The existence of a Mackey structure implies that $\pi_t\BE(-)\tensor_\IZ\IQ$ is a projective and hence flat $\IQ[\Sub G(\Fin)]$-module by~\cite{L-Chern}*{Theorem~5.2 on page~216}.
The same theorem together with~\cite{L-Chern}*{(2.12) on page~205} yields the isomorphism~$\mathit{h}$.
We remark that the definition of~$\mathit{h}$ depends on choices~\cite{L-Chern}*{(2.10) on page~204}, and so $\mathit{h}$ is a priori not natural in~$\BE$.

The map~$\mathit{c}$ is constructed in~\cite{L-Chern}*{Section~4, in particular Lemma~4.3(b) on page~211}.
It is defined for any equivariant homology theory, and its definition only depends on the induction structure (no choices).
The flatness condition above allows to interchange homology and tensor product in the source of~$\mathit{c}$.
Then \cite{L-Chern}*{Theorem~4.4 on page~213}, together with the isomorphism at the beginning of that proof, implies that $\mathit{c}$ is an isomorphism.

Finally, the isomorphism~\eqref{eq:idempotent} is proved in~\cite{LR-cyclic}*{Lemma~7.4 on page~622}.
\end{proof}

\begin{example}
\label{EG(F)-computation}
If $X=EG(\CF)$ for a family $\CF$ of finite subgroups of~$G$, then the maps
\[
EG(\CF)^H/Z_GH \FROM EZ_GH\timesd_{Z_GH} EG(\CF)^H \TO EZ_GH/Z_GH
\]
induce an isomorphism
\[
H_s\bigl(EG(\CF)^H/Z_GH;\IQ\bigr)
\cong
H_s(BZ_GH;\IQ)
\,;
\]
compare~\cite{LRV}*{proof of Lemma~8.1 on page~164}.
Under the Chern character isomorphisms~\eqref{eq:Chern}, an inclusion $\CF\subseteq\CF'$ of families of finite subgroups induces the inclusion of the summands indexed by $(H)\in\CF$, which in particular is always injective.
\end{example}

\begin{example}
\label{S_t(H;K)-computation}
If $R$ is an arbitrary discrete ring and $\BE=\K^{\ge0}(R[\oid{G}{-}])$, then $S_t(H;\K^{\ge0}(R[\oid{G}{-}]))=0$ for any $t\ge0$ and any non-cyclic finite subgroup~$H$ of~$G$ by~\cite{LR-cyclic}*{Lemma~7.2 on page~621}.
In particular, the inclusion $\FCyc\subseteq\Fin$ induces an isomorphism
\[
H^G_n\bigl(EG(\FCyc);\K^{\ge0}(R[\oid{G}{-}])\bigr)\tensor_\IZ\IQ
\TO[\cong]
H^G_n\bigl(EG(\Fin);\K^{\ge0}(R[\oid{G}{-}])\bigr)\tensor_\IZ\IQ
\,.
\]
\end{example}

Next we investigate the naturality of \autoref{Chern} in the functor~$\BE$.
Because of the subtleties explained in \autoref{tau-sigma-Mackey}, we also need to consider natural transformations that do not necessarily respect the Mackey structures.

\begin{addendum}
\label{rel-Chern}
Keep the notation and assumption of \autoref{equiv-homolgy-theory}.
Consider another equivalence-preserving functor $\MOR{\BE'}{\Groupoids}{\IN\Sp}$ and a natural transformation $\MOR{\tau}{\BE}{\BE'}$.
Assume that both $H^?_*(-;\BE)$ and~$H^?_*(-;\BE')$ have Mackey structures.
\begin{enumerate}
\item
\label{i:tau-ind}
The induced homomorphism
\[
\MOR{H^G_n(X;\tau)\tensor_\IZ\IQ}%
    {H^G_n(X;\BE )\tensor_\IZ\IQ}%
    {H^G_n(X;\BE')\tensor_\IZ\IQ}
\]
is compatible with the isomorphisms~$\mathit{c}$ in~\eqref{eq:Chern}, even when $\tau$ does not respect the Mackey structures.

\item
\label{i:tau-Chern}
If $\tau$ induces a natural transformation of Mackey functors, then there are choices of the maps~$\mathit{h}$  in~\eqref{eq:Chern} for which $H^G_n(X;\tau)\tensor_\IZ\IQ$ is compatible with the isomorphisms~$\mathit{ch}=\mathit{c}\circ\mathit{h}$ for any $\FCyc$-G-CW-complex~$X$.

\item
\label{i:tau-Chern-inj}
If $\tau$ induces a natural transformation of Mackey functors, and the homomorphism
\(
{\pi_t\BE (C)\tensor_\IZ\IQ}
\TO
{\pi_t\BE'(C)\tensor_\IZ\IQ}
\)
is injective for all finite cyclic subgroups $C$ of~$G$ and all $t\in\IZ$, then $H^G_n(X;\tau)\tensor_\IZ\IQ$ is injective for any $\FCyc$-G-CW-complex~$X$.
\end{enumerate}
\end{addendum}

\begin{proof}
First of all, notice that $\tau$ induces a natural transformation of equivariant homology theories $H^G(-;\BE)\TO H^G(-;\BE')$ which is compatible with the induction structures.
Therefore it induces a natural transformation between the functors $\Sub G(\Fin)\TO\Ab$ defined in~\eqref{eq:SubG->Ab}.
Since the definition of the map~$\mathit{h}$ only depends on the induction structure, \ref{i:tau-ind} follows.

Furthermore, since also definition~\eqref{eq:S_t(H;E)} only depends on the induction structure, $\tau$ always induces a map~$S_t(H;\BE)\TO S_t(H;\BE')$.
Now assume that $X$ is a $\FCyc$-G-CW-complex~$X$.
Then in the source of~$\mathit{h}$ we only have to consider finite cyclic subgroups $H=C$, and in this case we have the isomorphism~\eqref{eq:idempotent}.

If we additionally assume that $\tau$ induces a natural transformation of Mackey functors, then we get a commutative diagram
\begin{equation}
\label{eq:tau-Chern-inj}
\begin{tikzcd}
\ds\Theta_C\smash{\Bigl(
\pi_t\BE(C)\tensor_\IZ\IQ
\Bigr)}
\arrow[r, hook]
\arrow[d]
&
\ds\pi_t\BE(C)\smash{\tensor_\IZ\IQ}
\arrow[r, two heads]
\arrow[d]
&
S_t(C,\BE)
\arrow[d]
\\
\ds\Theta_C\smash{\Bigl(}
\pi_t\BE'(C)\tensor_\IZ\IQ
\smash{\Bigr)}
\arrow[r, hook]
&
\ds\pi_t\BE'(C)\tensor_\IZ\IQ
\arrow[r, two heads]
&
S_t(C,\BE')
\end{tikzcd}
\end{equation}
for any finite cyclic subgroup~$C$.
The composition of the horizontal maps is an isomorphism.
The inverse of this isomorphism, followed by the inclusion of the image of~$\Theta_C$, then gives compatible choices of the sections needed in~\cite{L-Chern}*{(2.10) on page~204} to define the map~$\mathit{h}$.
This proves~\ref{i:tau-Chern}.

Diagram~\eqref{eq:tau-Chern-inj} also shows that the two vertical maps on the outside are injective provided that the one in the middle is injective.
Then \ref{i:tau-Chern-inj} follows because the Weyl group~$W_G H$ is finite for every finite group~$H$, hence $\IQ[W_G H]$ is semisimple, and the functor $M\tensor_{\IQ[W_G H]}-$ preserves injectivity for every~$M$.
\end{proof}


\section{Enriched categories of modules over a category}
\label{MODULES}

The goal of this technical section is to show that the formalism of modules over a category (for example, see~\cite{L-LNM}*{Section~9, page~162}) extends to the enriched setting.
For our main applications in the following sections, it would be enough to consider spectral categories only.
But since we actually need two different models of spectral categories (with enrichments in either $\Sigma\Sp$ or~$\Gamma\CS_*$),
and since the proofs are not any easier in these special cases, we proceed in complete generality.
In the case of pre-additive categories, the constructions in this section can be simplified, because products and coproducts are the same.
Since this is not the case in categories of spectra, we resort to machinery from enriched category theory, which we now recall.
The standard modern reference for this material (and much more) is~\cite{Kelly}.

Fix a closed symmetric monoidal category $(\CV,\odot,\Ione)$; compare \autoref{ENRICHED}.
We always assume that $\CV$ is both complete and cocomplete, as in all the examples in~\ref{ENRICHED}.
We denote by~$\CV\D\Cat$ the 2-category of small $\CV$-categories, $\CV$-functors, and $\CV$-natural transformations.
To distinguish between $\CV$-categories and $\Sets$-categories, we call the latter ordinary categories and denote them $\CC$, $\CD$, \ldots, whereas we use the symbols $\SX$, $\SY$, \ldots\ to denote $\CV$-categories.

Let
\[
\MOR{V=\CV(\Ione,-)}{\CV}{\Sets}
\]
be the underlying set functor.
Notice that in some of our examples ($\Cat$, $\Sigma\Sp$, $\Gamma\CS_*$) the functor $V$ is neither faithful nor conservative.
Since we are assuming that $\CV$ is cocomplete, $V$ has a left adjoint
\[
\MOR{\Ione[-]}{\Sets}{\CV}
\,,\qquad
S\longmapsto\Ione[S]=\smallcoprod_S\Ione
\,.
\]
Both functors $\Ione$ and~$V$ are lax monoidal, and base change along them (see \autoref{ENRICHED}) induces a 2-adjunction
\[
\begin{tikzcd}[column sep=large]
\CV\D\Cat
\arrow[r, shift left, "V"]
&
\Cat\,;
\arrow[l, shift left, "\Ione{[-]}"]
\end{tikzcd}
\]
compare~\cite{Kelly}*{Section~2.5 on page~35}.

We make extensive use of the $\CV$-categories of functors $\CV\D\fun(\Ione[\CC],\SX)$, where $\CC$ is an ordinary small category, and $\SX$ is a $\CV$-category.
To simplify the notation we simply write $\CV\D\fun(\CC,\SX)$ in this case, and  this $\CV$-category can be explicitly described as follows.
The objects are precisely the functors from $\CC$ to the underlying category~$V\SX$, i.e., $\obj\CV\D\fun(\CC,\SX)=\obj\fun(\CC,V\SX)$.
Given two such functors $X$ and~$Y$, the $\CV$-object of morphisms from~$X$ to~$Y$ is defined as the end in~$V\CV$ of the functor
\[
\CC^\op\times\CC
\xrightarrow{X^\op\times Y}
V\SX^\op\times V\SX
\xrightarrow{\SX(-,-)}
V\CV
\,.
\]
Notice that
\begin{equation}
\label{eq:VVfun}
V\bigl(\CV\D\fun(\CC,\SX)\bigr)\cong\fun(\CC,V\SX)
\,.
\end{equation}

Given a functor~$\MOR{\alpha}{\CC}{\CD}$, the functor $\MOR{\res_\alpha}{\fun(\CD,V\SX)}{\fun(\CC,V\SX)}$ given by restriction along~$\alpha$ is clearly a $\CV$-functor
\begin{equation}
\label{res-enriched}
\MOR{\res_\alpha}{\CV\D\fun(\CD,\SX)}{\CV\D\fun(\CC,\SX)}
\,.
\end{equation}

Now assume that $\SX$ is $\CV$-complete and $\CV$-cocomplete, which is equivalent to assuming that $\SX$ is $\CV$-tensored and $\CV$-cotensored and that the underlying category~$V\SX$ is complete and cocomplete; see~\cite{Kelly}*{Theorem~3.73 on page~54}.
Then the assumptions on~$\SX$ imply that the left Kan extension along~$\alpha$
\[
\MOR{\Lan_\alpha}{\fun(\CC,V\SX)}{\fun(\CD,V\SX)}
\,,\qquad
X\longmapsto
\Bigl(d\mapsto\CD(\alpha(-),d)\tensor_\CC X(-)\Bigr)
\]
exists and agrees with the $\CV$-enriched left Kan extension, and therefore defines a $\CV$-functor
\[
\MOR{\Lan_\alpha}{\CV\D\fun(\CC,\SX)}{\CV\D\fun(\CD,\SX)}
\]
which is left 2-adjoint to the $\CV$-functor~$\res_\alpha$ in~\eqref{res-enriched}; compare~\cite{Kelly}*{second paragraph on page~50}.

Next, we recall some basic facts about Kan extensions that are needed later.
The first two are well-known, and the third is also used in~\cite{RV}*{proof of Proposition~8.3 on page~1528}.

\begin{fact}
\label{LanLan}
Given functors $\MOR{\alpha}{\CB}{\CC}$ and $\MOR{\beta}{\CC}{\CD}$, then there is a natural isomorphism $\Lan_{\beta\circ\alpha}\cong\Lan_\beta\circ\Lan_\alpha$.
\end{fact}

\begin{fact}
\label{Lan-equivalence}
If $\alpha$ is an equivalence of categories, then there are natural isomorphisms $\res_\alpha\Lan_\alpha\cong\id$ and $\Lan_\alpha\res_\alpha\cong\id$.
\end{fact}

\begin{fact}
\label{Marco's-trick}
Consider a diagram
\[
\begin{tikzcd}
\CA
\arrow[r, "\alpha"]
\arrow[d, "\gamma"']
&
\CB
\arrow[d, "\beta"]
\\
\CC
\arrow[r, "\delta"']
\arrow[d, "X"']
\arrow[Rightarrow, ru, shorten <=1em, shorten >=1em, "\nu"]
&
\CD
\\
V\SX
\end{tikzcd}
\]
where $\CA$, $\CB$, $\CC$, and~$\CD$ are ordinary small categories, and let $\nu$ be a natural transformation from $\delta\gamma$ to $\beta\alpha$.
Then $\nu$ induces a natural transformation
\[
\MOR{\nu_*}{\Lan_\alpha\res_\gamma{X}}{\res_\beta\Lan_\delta{X}}
\]
as we proceed to explain.
For a functor $\MOR{\delta}{\CC}{\CD}$ define
\begin{align*}
\MOR{\modu{}{\CD}{\delta}}{\CC^\op\times\CD}{\Sets}
,\quad
(c,d)&\mapsto\CD\bigl(\delta(c),d\bigr)
\,,
\\
\MOR{\modu{\delta}{\CD}{}}{\CD^\op\times\CC}{\Sets}
,\quad
(d,c)&\mapsto\CD\bigl(d,\delta(c)\bigr)
\,.
\end{align*}
Notice that by Yoneda Lemma, for any functor $\MOR{Y}{\CD}{V\SX}$, we have $\res_\delta{Y}\cong\modu{\delta}{\CD}{}\tensor_\CD{Y}$.
Recall that
\(
\Lan_\delta{X}=\CD_\delta\tensor_\CC{X}
\).

Therefore we have
\begin{align*}
\Lan_\alpha\res_\gamma{X}
&\cong
\modu{}{\CB}{\alpha} \tensor_\CA \bigl(\modu{\gamma}{\CC}{} \tensor_\CC{X}\bigr)
\cong
\bigl(\modu{}{\CB}{\alpha} \tensor_\CA \modu{\gamma}{\CC}{}\bigr) \tensor_\CC{X}
\,,\AND
\\
\res_\beta\Lan_\delta{X}
&\cong
\modu{\beta}{\CD}{} \tensor_\CD \bigl(\modu{}{\CD}{\delta} \tensor_\CC{X}\bigr)
\cong
\bigl(\modu{\beta}{\CD}{} \tensor_\CD \modu{}{\CD}{\delta}\bigr) \tensor_\CC{X}
\cong
\modu{\beta}{\CD}{\delta} \tensor_\CC{X}
\,.
\end{align*}
Now let $\overline{\nu}$ be the composition
\[
\overline{\nu}
\colon
\modu{}{\CB}{\alpha} \tensor_\CA \modu{\gamma}{\CC}{}
\TO
\modu{\beta}{\CD}{\beta\alpha} \tensor_\CA \modu{\beta\alpha}{\CD}{\delta}
\TO
\modu{\beta}{\CD}{} \tensor_\CD \modu{}{\CD}{\delta}
\cong
\modu{\beta}{\CD}{\delta}
\]
where the first map is induced by the functors $\beta$ and~$\delta$ together with the natural transformation~$\nu$.
Then define~$\nu_*$ by applying $-\tensor_\CC{X}$ to~$\overline{\nu}$ and using the isomorphisms above.
Furthermore:
\begin{enumerate}
\item\label{overline-nu-iso}
If $\overline{\nu}$ is a natural isomorphism of functors $\CC^\op\times\CB\TO\Sets$, then $\nu_*$ is a natural isomorphism of functors $\CB\TO V\SX$, natural in~${X}$.
\item\label{vertical-eq}
If $\alpha$ and~$\delta$ are equivalences of categories and $\nu$ is a natural isomorphism, then $\nu_*$ is an isomorphism.
\end{enumerate}
\end{fact}

\begin{proof}
\ref{overline-nu-iso} follows directly from the construction.
\ref{vertical-eq} follows from~\ref{overline-nu-iso} by observing that the assumptions imply that $\overline{\nu}$ is an isomorphism.
Notice also that \autoref{Lan-equivalence} follows from~\ref{vertical-eq}.
\end{proof}

Now fix also a monoid $\IA$ in~$\CV$, or in other words, a $\CV$-category with exactly one object.
Let $\IA\D\Mod$ be the $\CV$-category of modules over~$\IA$, and notice that $
\IA\D\Mod\cong\CV\D\fun(\IA,\CV)$.
Notice that $\IA\D\Mod$ is $\CV$-complete and $\CV$-cocomplete;
see for example~\cite{Kelly}*{Proposition~3.75 on page~54}.
Therefore the discussion above about left Kan extensions applies to functors with values in~$\IA\D\Mod$.

\begin{definition}[$\IA$-modules over~$\CC$]
\label{all-mod}
Given an ordinary small category~$\CC$, define
\[
\SM_\IA(\CC)=\CV\D\fun\bigl(\CC,\IA\D\Mod\bigr)
\]
to be the $\CV$-category of $\CV$-functors from~$\CC$ to the $\CV$-category of $\IA$-modules.
The objects of this category are called \emph{$\IA$-modules over~$\CC$}.
\end{definition}

Next, we want to define what it means for an $\IA$-module over~$\CC$ to be free and finitely generated.
Notice that for the underlying category of~$\SM_\IA(\CC)$ we have
\[
V\SM_\IA(\CC)\cong\fun\bigl(\CC,V(\IA\D\Mod)\bigr)
\]
by~\eqref{eq:VVfun}.
There are two pairs of adjoint functors
\[
\begin{tikzcd}[column sep=large]
V\IA\D\Mod
\arrow[r, shift left, "u"]
&
V\CV
\arrow[r, shift left, "V"]
\arrow[l, shift left, "\IA\odot-"]
&
\Sets
\arrow[l, shift left, "{\Ione[-]}"]
\end{tikzcd}
\]
where the left adjoints are displayed at the bottom, and they induce adjoint functors between the ordinary functor categories
\[
\begin{tikzcd}[column sep=large]
\fun(\CC,V\IA\D\Mod)
\arrow[r, shift left, "u_*"]
&
\fun(\CC,V\CV)
\arrow[r, shift left, "V_*"]
\arrow[l, shift left, "(\IA\odot-)_*"]
&
\fun(\CC,\Sets)\,.
\arrow[l, shift left, "{(\Ione[-])_*}"]
\end{tikzcd}
\]
We use the abbreviation $\IA[X]=\IA\odot\Ione[X]$.
Furthermore, the inclusion~$\iota$ of the discrete subcategory $\obj\CC$ into~$\CC$ induces another pair of adjoint functors
\[
\begin{tikzcd}[column sep=large]
\fun(\CC,\Sets)
\arrow[r, shift left, "\iota^*"]
&
\fun(\obj\CC,\Sets)\,,
\arrow[l, shift left, "L"]
\end{tikzcd}
\]
where the left adjoint~$L$ is defined as follows.
Any $X\in\obj\fun(\obj\CC,\Sets)$ is just a collection~$\{X(c)\}_{c\in\obj\CC}$ of sets indexed by the objects of~$\CC$.
The functor~$LX$ is then defined as
\[
LX=\adjustlimits\smallcoprod_{c\in\obj\CC}\smallcoprod_{X(c)}\CC(c,-)
\,.
\]
Composing these functors gives an adjunction
\[
\begin{tikzcd}[column sep=large]
\fun(\CC,V\IA\D\Mod)
\arrow[r, shift left, "U"]
&
\fun(\obj\CC,\Sets)\,.
\arrow[l, shift left, "F"]
\end{tikzcd}
\]
Notice that, given $X$ as above, then $FX$ is isomorphic to
\begin{equation}
\label{eq:fgf}
\adjustlimits\smallcoprod_{c\in\obj\CC}\smallcoprod_{X(c)}\IA[\CC(c,-)]
\,.
\end{equation}

\begin{definition}[Finitely generated and free $\IA$-modules over~$\CC$]
\label{fgf-mod}
If $M\in\obj\SM_\IA(\CC)$ is an $\IA$-module over~$\CC$, we say that $M$~is \emph{free} if it is isomorphic in $V\SM_\IA(\CC)=\fun(\CC,V\IA\D\Mod)$ to~$FX$ for some~$X\in\obj\fun(\obj\CC,\Sets)$, and it is \emph{finitely generated free} if $X(c)$ is finite for all~$c\in\obj\CC$ and nonempty only for finitely many~$c\in\obj\CC$.
We let
\[
\SF_\IA(\CC)
\]
be the full $\CV$-subcategory of~$\SM_\IA(\CC)$ whose objects are the finitely generated free modules.
\end{definition}

\begin{example}
\label{fgf-on-one-obj}
There is a $\CV$-functor
\(
\IA[\CC]\TO\SF_\IA(\CC)
\),
\(
c\longmapsto\IA[\CC(c,-)]
\).
\end{example}

Since $\CV$ has all coproducts, the same is true in~$\IA\D\Mod$, and therefore also in~$\SM_\IA(\CC)$, as coproducts are computed objectwise in functor categories.
The coproduct of two finitely generated free modules is again finitely generated free by inspection, and therefore also $\SF_\IA(\CC)$ has all finite coproducts.
Notice that if $\CC$ is small, then $\SF_\IA(\CC)$ has a small skeleton.

Now let $\MOR{\alpha}{\CC}{\CD}$ be a functor.
The left Kan extension along~$\alpha$ gives a $\CV$-functor
\(
\MOR{\Lan_\alpha}{\SM_\IA(\CC)}{\SM_\IA(\CD)}
\).
By inspection, one verifies that $\Lan_\alpha$ restricts to a functor
\begin{equation}
\label{eq:pseudo-fun-fgf-mod}
\MOR{\SF_\IA(\alpha)}{\SF_\IA(\CC)}{\SF_\IA(\CD)}
\,,
\end{equation}
and that it preserves coproducts.
In fact, if $M\in\obj\SF_\IA(\CC)$ is isomorphic to~\eqref{eq:fgf} for some $X\in\obj\fun(\obj\CC,\Sets)$, then
\begin{equation}
\label{eq:fgf-mor}
\SF_\IA(\alpha)(M)
\cong\adjustlimits\smallcoprod_{d\in\obj\CD}\smallcoprod_{c\in\alpha^{-1}(d)}\smallcoprod_{X(c)}\IA[\CD(d,-)]
\,.
\end{equation}
Since Kan extension are functorial up to natural isomorphisms, in the sense of \autoref{LanLan}, we get a pseudo-functor
\[
\MOR{\SF_\IA}{\Cat}{\CV\D\CAT^\copr}
,
\]
where we denote by $\CV\D\CAT^\copr$ the $2$-category of skeletally small $\CV$-categories with finite coproducts, coproduct preserving $\CV$-functors, and $\CV$-natural transformations.
For the definition of pseudo-functor, see for example~\cite{Borceux}*{Definition~7.5.1 on pages~296--297}.

For the applications in the next section, it is necessary to work with small categories, not just skeletally small ones, and with functors, not just pseudo-functors.
This can be achieved as follows.
First, choose a functorial coproduct in~$\IA\D\Mod$.
Then, let $\SF'_\IA(\CC)$ be the small full subcategory of~$\SF_\IA(\CC)$ whose objects are precisely the ones in~\eqref{eq:fgf}.
Finally, using~\eqref{eq:fgf-mor}, it is clear how to define~$\SF'_\IA(\alpha)$ in order to get a strict functor
\[
\MOR{\SF'_\IA}{\Cat}{\CV\D\Cat^\copr}
,
\]
where $\CV\D\Cat^\copr$ is the full subcategory of the small $\CV$-categories with finite coproducts.
In order to simplify the exposition, we do not distinguish between the equivalent constructions $\SF_\IA$ and~$\SF'_\IA$ in the next section.


\section{Constructing homology theories with Mackey structures}
\label{MACKEY}

The main result of this section, \autoref{constr-Mackey}, gives a very general procedure to construct equivariant homology theories with Mackey structures.
This is used in \autoref{TRC} to construct the main diagram~\eqref{eq:main-diagram} and show that the results of \autoref{CHERN} apply.
A special case of \autoref{constr-Mackey} for additive categories was proved in~\cite{LR-cyclic}*{Proposition~6.5 on page~620}.

As in the previous section, we fix a closed symmetric monoidal category $(\CV,\odot,\Ione)$ and a monoid~$\IA$ in~$\CV$.
We assume that $\CV$ is complete and cocomplete.
We denote by $\CV\D\Cat^\copr$ the category of small $\CV$-categories with finite coproducts.

Let $\MOR{\BT}{\CV\D\Cat^\copr}{\IN\Sp}$ be a functor; see \autoref{examples-of-T} below.
For any group~$G$, we consider the composition
\begin{equation}
\label{eq:OrG->VCat->NSp}
\Or G
\xhookrightarrow{\quad}
\Sets^G
\xrightarrow{\,\oid{G}{-}\,}
\Groupoids
\xrightarrow{\ \SF_\IA\ }
\CV\D\Cat^\copr
\xrightarrow{\ \BT\ }
\IN\Sp
\end{equation}
and the corresponding $G$-homology theory $H^G_*(-;\BT\SF_\IA)$ from \autoref{equiv-homolgy-theory}.
Here $\SF_\IA$ is the restriction to~$\Groupoids$ of the functor constructed at the end of the last section, which sends an ordinary small category to the $\CV$-category of finitely generated and free $\IA$-modules over it.
We can now state the main result of this section.

\begin{theorem}
\label{constr-Mackey}
Let $\MOR{\BT}{\CV\D\Cat^\copr}{\IN\Sp}$ be a functor.
Assume that $\BT$ preserves equivalences and products, in the sense that $\BT$ sends equivalences to $\pi_*$-isomorphisms and the natural map
\(
\BT(\SX\times\SY)\TO\BT(\SX)\times\BT(\SY)
\)
is a $\pi_*$\=/isomorphism for all $\SX$ and~$\SY$.
Then the equivariant homology theory~$H^?_*(-;\BT\SF_\IA)$ has a Mackey structure.
\end{theorem}

\begin{example}
\label{examples-of-T}
Examples of functors~$\T$ that satisfy the assumptions of \autoref{constr-Mackey} are given by the composition
\[
\Gamma\CS_*\D\Cat^\copr\TO\Sym\Mon\Gamma\CS_*\D\Cat\xrightarrow{\ \K\ }\IN\Sp
\,,
\]
where the first functor is the evident inclusion and the second is the algebraic $K$-theory functor from~\eqref{eq:Dundas-K-functor}, and by
\[
\Sigma\Sp_*\D\Cat^\copr\TO\Sigma\Sp_*\D\Cat\xrightarrow{\TC(-;p)}\IN\Sp
\,,
\]
where the first functor is the forgetful one.
\end{example}

In order to prove \autoref{constr-Mackey}, we analyze the composition
\[
\Sets^G
\xrightarrow{\,\oid{G}{-}\,}
\Groupoids
\xrightarrow{\ \SF_\IA\ }
\CV\D\Cat^\copr
\]
through which \eqref{eq:OrG->VCat->NSp} factors.
The following examples illustrate this construction and guide our intuition.

\begin{example}
For any subgroup $H$ of~$G$, $\SF_\IA(\oid{G}{\,(G/H)})$ is equivalent to the $\CV$-category of finitely generated free $\IA[H]$-modules.
\end{example}

\begin{example}
\label{Sets-pt}
Consider the special case $(\CV,\odot,\Ione)=(\Sets,\times,\pt)$ and $\IA=\pt$, and notice that then $\pt\D\Mod=\Sets$.
Then for any $G$-set~$S$ the categories
\[
\SF_{\pt}(\oid{G}{S})
\AND
\coffree{G}\downarrow S
\]
are equivalent, where $\coffree{G}\downarrow S$ is the comma category of cofinite and free $G$-sets over~$S$.

To see this, consider first the category~$\SM_{\pt}(\oid{G}{S})$.
There is an equivalence of categories
\begin{equation}
\label{eq:slice-cat}
\SM_{\pt}(\oid{G}S{})\TO\Sets^G\downarrow S
\,,
\end{equation}
sending a functor $\MOR{X}{\oid{G}{S}}{\Sets}$ to the $G$-set $\smallcoprod_{s\in S} X(s)$ together with the obvious $G$-map to~$S$ sending any $x\in X(s)$ to~$s$.
In the opposite direction, a $G$-map $\MOR{p}{T}{S}$ is sent to the functor $s\mapsto p^{-1}(\{s\})$.

It is then easy to see that the equivalence of categories given by~\eqref{eq:slice-cat} induces an equivalence between $\SF_{\pt}(\oid{G}{S})$ and $\coffree{G}$.
\end{example}

We now establish a couple of preliminary results.
The first one follows immediately from the definitions.

\begin{lemma}
\label{finite-add}
For any group~$G$ and any pair of cofinite $G$-sets $S_1$ and~$S_2$ there is a natural isomorphism of $\CV$-categories with coproducts
\[
\MOR{\varphi}{\SF_\IA\bigl(\oid{G}{\,(S_1\amalg S_2)\bigr)}}%
{\SF_\IA(\oid{G}{S_1})\times\SF_\IA(\oid{G}{S_2})}
\,.
\]
\end{lemma}

\begin{lemma}
\label{eval-on-S}
Under the assumptions of \autoref{constr-Mackey},
for any group~$G$ and any cofinite $G$-set~$S$ there is a natural isomorphism
\[
H^G_*(S;\BT\SF_\IA)
\cong
\pi_*\bigl(\BT\SF_\IA(\oid{G}{S})\bigr)
\,.
\]
\end{lemma}

\begin{proof}
Recall that, by construction,
\(
H^G_*(S;\BT\SF_\IA)=\pi_*\bigl((\BT\SF_\IA(\oid{G}{-}))_\%(S)\bigr)
\),
where $(\BT\SF_\IA(\oid{G}{-}))_\%$ is the continuous left Kan extension of~\eqref{eq:OrG->VCat->NSp} along the inclusion $\Or G\hookrightarrow\Top^G$.
Choose an isomorphism $\MOR{f}{S}{G/H_1\amalg\dotsb\amalg G/H_\ell}$.
We then get the following maps.
\[
\begin{tikzcd}[column sep=tiny]
\ds\bigl(\BT\SF_\IA(\oid{G}{-})\bigr)_\%(S)
\arrow[r, "\cong"]
&
\ds\bigl(\BT\SF_\IA(\oid{G}{-})\bigr)_\%(
G/H_1
\amalg\dotsb\amalg
G/H_\ell)
\arrow[leftarrow, d, "\ts\,\one", "\cong\,"']
\\
&
\ds
\BT\SF_\IA\bigl(\oid{G}{\,(G/H_1)}\bigr)
\vee\dotsb\vee
\BT\SF_\IA\bigl(\oid{G}{\,(G/H_\ell)}\bigr)
\arrow[d, "\ts\,\two"]
\\
\ds\BT\SF_\IA(\oid{G}{S})
\arrow[r, "\cong"]
&
\ds
\BT\SF_\IA\bigl(\oid{G}{\,(G/H_1
\amalg\dotsb\amalg
                      G/H_\ell)}\bigr)
\arrow[d, "\ts\,\three", "\cong\,"']
\\
&
\ds
\BT\Bigl(
   \SF_\IA\bigl(\oid{G}{\,(G/H_1)}\bigr)
\times\dotsb\times
   \SF_\IA\bigl(\oid{G}{\,(G/H_\ell)}\bigr)
\Bigr)
\arrow[d, "\ts\,\four"]
\\
&
\ds
\BT\SF_\IA\bigl(\oid{G}{\,(G/H_1)}\bigr)
\times\dotsb\times
\BT\SF_\IA\bigl(\oid{G}{\,(G/H_\ell)}\bigr)
\end{tikzcd}
\]
The horizontal isomorphisms are induced by~$f$.
The isomorphism~$\one$ follows from the fact that the discrete category with $\ell$ objects given by the inclusions $G/H_i\TO{G/H_1\amalg\dotsb\amalg G/H_\ell}$ is cofinal in~$\Or G\downarrow({G/H_1\amalg\dotsb\amalg G/H_\ell})$.
The maps $\two$ and~$\four$ are the natural ones induced by functoriality, and $\three$ is induced by the isomorphism~$\varphi$ from \autoref{finite-add}.
We are assuming that $\four$ is a $\pi_*$-isomorphism.
The composition $\four\circ\three\circ\two$ is the natural map from the finite coproduct to the product in~$\IN\Sp$, which is also a $\pi_*$-isomorphism.
So we conclude that $\two$ is a $\pi_*$-isomorphism, proving the result.
\end{proof}

Now consider a group homomorphism $\MOR{\alpha}{H}{G}$.
For any $H$-set~$S$ there is a functor
\begin{equation}
\label{eq:over-alpha}
\MOR{\overline{\alpha}}{\oid{H}{S}}{\oid{G}{\ind_\alpha S}},
\end{equation}
defined on objects by $s\mapsto[1,s]$ and on morphisms by~$h\mapsto\alpha(h)$.
On the other hand, for any $G$-set~$T$ there is a functor
\[
\MOR{\underline{\alpha}}{\oid{H}{\res_\alpha T}}{\oid{G}{T}},
\]
which is the identity on objects and $\alpha$ on morphisms.

Using $\overline{\alpha}$ we define
\begin{equation}
\label{eq:Ind}
\MOR{\Ind_\alpha=\Lan_{\overline{\alpha}}} 
{\SF_\IA(\oid{H}{S})}{\SF_\IA(\oid{G}{\ind_\alpha S})}
\end{equation}
as in \eqref{eq:pseudo-fun-fgf-mod}.
On the other hand, restriction along $\underline{\alpha}$ defines a functor
\[
\MOR{\Res_\alpha=\res_{\underline{\alpha}}}%
{\SM_\IA(\oid{G}{T})}{\SM_\IA(\oid{H}{\res_\alpha T})}
\,.
\]
If we assume that $\alpha$ is injective and $\#(G/\alpha(H))<\infty$, then $\Res_\alpha$ yields a functor
\begin{equation}
\label{eq:Res}
\MOR{\Res_\alpha=\res_{\underline{\alpha}}}%
{\SF_\IA(\oid{G}{T})}{\SF_\IA(\oid{H}{\res_\alpha T})}
\,.
\end{equation}
These induction and restriction functors \eqref{eq:Ind} and~\eqref{eq:Res} for cofinite equivariant sets satisfy the following double coset formula.

\begin{lemma}[Double coset formula]
\label{double-coset}
Let $H$ and~$K$ be subgroups of~$G$ with $\#(G/K)<\infty$, and let $\MOR{a}{H}{G}$ and $\MOR{b}{K}{G}$ be the inclusions.
Let $\CJ=K\backslash G/H$.
For every $j\in\CJ$, choose a representative~$g_j\in j=KgH$, define $L_j=H\cap g_j^{\smash{-1}}Kg_j$, and let $\MOR{d_j}{L_j}{H}$ be the inclusion and $\MOR{c_j=\conj_{g_j}}{L_j}{K}$.
There is a commutative diagram
\[
\begin{tikzcd}[row sep=scriptsize, column sep=small]
\smash{\ds\smallcoprod_{j\in\CJ}} G/L_j
\arrow[rd, "\cong" description]
\arrow[rrrd, bend left=18, "(c_j)"]
\arrow[rddd, bend right,   "(d_j)"']
\\
&
G/H\times G/K
\arrow[rr]
\arrow[dd]
&&
G/K
\arrow[dd, "b"]
\\
\\
&
G/H
\arrow[rr, "a"']
&&
G/G
\end{tikzcd}
\]
whose inner and outer squares are pullbacks, where we write $a$, $b$, $d_j$ also for the maps of orbits induced by the inclusions, and $\MOR{c_j}{G/L_j}{G/K}$ sends $xL_j$ to~$xg_j^{\smash{-1}}K$.
Then, for any cofinite $H$-set~$S$, the diagram
\begin{equation}
\label{eq:double-coset}
\begin{tikzcd}[column sep=6em, row sep=1.1em]
\ds\smash{\smallprod_{j\in\CJ}} \SF_\IA(\oid{L_j}{\res_{d_j} S})
\arrow[r, "\prod\Ind_{c_j}"]
&
\ds\smash{\smallprod_{j\in\CJ}} \SF_\IA(\oid{K}{\ind_{c_j}\res_{d_j} S})
\\
&
\ds\SF_\IA\Bigl(\oid{K}{\smash{\smallcoprod_{j\in\CJ}} \ind_{c_j} \res_{d_j} S}\Bigr)
\arrow[u, "\cong", "\ \varphi"']
\\
&
\ds\SF_\IA(\oid{K}{\res_b \ind_a S})
\arrow[u, "\cong", "\ \SF_\IA^K(\oid{K}{w^{-1}})"']
\\
\\
\ds\SF_\IA(\oid{H}{S})
\arrow[uuuu, "(\Res_{d_j})\ "]
\arrow[r, "\Ind_a"']
&
\ds\SF_\IA(\oid{G}{\ind_a S})
\arrow[uu, "\ \Res_b"']
\end{tikzcd}
\end{equation}
commutes up to natural isomorphism, where
\[
\MOR[\cong]{w}{\smallcoprod_{j\in\CJ}\ind_{c_j} \res_{d_j} S}{\res_b \ind_a S}
\]
is the natural $K$-isomorphism sending $[k,s]\in\ind_{c_j} \res_{d_j} S=K\times_{L_j} \res_{d_j} S$ to $[kg_j,s]\in\res_b \ind_a S=\res_b (G\times_H S)$,
and $\varphi$ is the natural isomorphism from \autoref{finite-add}.
\end{lemma}

\begin{proof}
First of all, notice that $\CJ$ is finite since $\#(G/K)<\infty$.
Notice also that $L_j$ and hence $c_j$ and~$d_j$ depend on the choice of representative.
It is enough to prove that diagram~\eqref{eq:double-coset} commutes up to isomorphism after composing with the projection from the top right corner to each factor $\SF_\IA(\oid{K}{\ind_{c_j}\res_{d_j} S})$.
Consider the following diagram.
\begin{equation}
\label{eq:pre-double-coset}
\begin{tikzcd}[column sep=6em, row sep=1.1em]
\ds\oid{L_j}{\res_{d_j} S}
\arrow[r, "\overline{c_j}"]
&
\ds\oid{K}{\ind_{c_j}\res_{d_j} S}
\\
&
\ds\oid{K}{\smash{\smallcoprod_{j\in\CJ}} \ind_{c_j} \res_{d_j} S}
\arrow[leftarrow, u, "\ \oid{K}{\incl_j}"']
\\
&
\ds\oid{K}{\res_b \ind_a S}
\arrow[leftarrow, u, "\cong", "\ \oid{K}{w}"']
\\
\\
\ds\oid{H}{S}
\arrow[leftarrow, uuuu, "\underline{d_j}\ "]
\arrow[r, "\overline{a}"']
\arrow[Rightarrow, ruuuu, shorten <=6em, shorten >=6em, "\nu"]
&
\ds\oid{G}{\ind_a S}
\arrow[leftarrow, uu, "\ \underline{b}"']
\end{tikzcd}
\end{equation}
Diagram~\eqref{eq:pre-double-coset} does not commute.
Given $s\in\res_{d_j}S$,
the composition through the left bottom corner sends $s$ to~$[1,s]$,
the composition through the top right corner sends $s$ to~$[g_j,s]$.
So there is a natural isomorphism $\nu=g_j$ in the indicated direction.
Notice that the functors $\overline{a}$ and $\overline{c_j}$ are equivalences of categories.
Now apply \autoref{Marco's-trick} and the definition of the functors in diagram~\eqref{eq:double-coset} to obtain the result.
\end{proof}

We are now ready to prove the main result of this section.

\begin{proof}[Proof of \autoref{constr-Mackey}]
The map~\eqref{eq:cov-mor} corresponds to the one obtained by applying $\pi_n\BT(-)$ to
\[
\SF_\IA(\oid{H}{\pt})
\xrightarrow{\Ind_\alpha}
\SF_\IA(\oid{G}{\ind_\alpha\pt})
\xrightarrow{\SF_\IA(\oid{G}{\pr})}
\SF_\IA(\oid{G}{\pt})
\,.
\]
We have to show that this extends to a Mackey functor.
The contravariant functoriality is induced by
\[
\SF_\IA(\oid{G}{\pt})
\xrightarrow{\Res_\alpha}
\SF_\IA(\oid{H}{\pt})
\,.
\]

To verify axiom~(i), notice that
if $\MOR{\alpha}{H}{G}$ is an isomorphism, then $\MOR{\overline{\alpha}=\underline{\alpha}}{\oid{H}{\pt}}{\oid{G}{\pt}}$ is an isomorphism.
Then apply \autoref{Lan-equivalence}.

(ii) follows from \autoref{Marco's-trick} and the fact that $g$ is a natural isomorphism from $\MOR{\underline{\conj}_g}{\oid{G}{\pt}}{\oid{G}{\pt}}$ to the identity.

(iii) follows by applying $\pi_n\BT(-)$ to the diagram~\eqref{eq:double-coset} in \autoref{double-coset} for~$S=\pt$ combined with the following commutative diagram.
\[
\begin{tikzcd}[column sep=large, row sep=1.1em]
\ds\smash{\smallprod_{j\in\CJ}} \SF_\IA(\oid{K}{\ind_{c_j} \pt})
\arrow[r]
&
\ds\smash{\smallprod_{j\in\CJ}} \SF_\IA(\oid{K}{\pt})
\\
\ds\SF_\IA\Bigl(\oid{K}{\smash{\smallcoprod_{j\in\CJ}} \ind_{c_j} \pt}\Bigr)
\arrow[u, "\cong", "\ \varphi"']
\arrow[r]
&
\ds\SF_\IA\Bigl(\oid{K}{\smash{\smallcoprod_{j\in\CJ}}\pt}\Bigr)
\arrow[u, "\cong", "\ \varphi"']
\\
\ds\SF_\IA(\oid{K}{\res_b \ind_a \pt})
\arrow[leftarrow, u, shorten >=.4ex, "\cong", "\ \SF_\IA^K(\oid{K}{w})"']
\arrow[r]
&
\ds\SF_\IA(\oid{K}{\pt})
\arrow[leftarrow, u, shorten >=.4ex, "\cong"]
\\
\ds\SF_\IA(\oid{G}{\ind_a \pt})
\arrow[u, "\ \Res_b"']
\arrow[r]
&
\ds\SF_\IA(\oid{G}{\pt})
\arrow[u, "\ \Res_b"']
\end{tikzcd}
\]
Here the unlabeled functors are induced by the projections to~$\pt$.
One also needs to use \autoref{eval-on-S} and the assumption on~$\BT$.
\end{proof}

Using the constructions and results in this section, we can also prove that $H^?_*(-;\BT\SF_\IA)$ is an {equivariant homology theory with restriction structure} in the sense of~\cite{L-Chern}*{Section~6, pages~217--218}.
This is not needed here, and details will appear elsewhere.


\section{The cyclotomic trace map and the main diagram}
\label{TRC}

In this section we explain the construction of the main diagram~\eqref{eq:main-diagram}.
An essential ingredient is the natural model for the cyclotomic trace map constructed by Dundas, Goodwillie, and McCarthy in~\cite{Dundas}.
Using the results from \autoref{MACKEY}, it follows that all the equivariant homology theories in~\eqref{eq:main-diagram} have Mackey structures, and therefore the theory of Chern characters from \autoref{CHERN} becomes available.
The proof of the Detection \autoref{detection} in \autoref{PROOF-DETECTION} is based on this.

In essence, diagram~\eqref{eq:main-diagram} is induced by natural transformations
\begin{equation*}
\K^{\ge0}(\IZ[-])
\FROM[\ell]
\K^{\ge0}(\IS[-])
\TO[\tau]
\smallprod_p\TC(\IS[-];p)
\TO[\sigma]
\THH(\IS[-])\times\smallprod_p\C(\IS[-];p)
\end{equation*}
of functors $\Groupoids\TO\IN\Sp$.
In detail, though, both $\tau$ and~$\sigma$ are zig-zags of natural transformations whose wrong-way maps are all natural $\pi_*$-isomorphisms.
Since natural $\pi_*$-isomorphisms induce isomorphisms of the associated homology theories, we use the same symbols for the zig-zags of natural transformations and for the induced one-way maps in homology.

We proceed to give technical details about $\ell$, $\tau$, and~$\sigma$.

The category of spectra used in~\cite{Dundas} is $\Gamma\CS_*$; see \autoref{SPECTRA}.
Given a $\Gamma\CS_*$-category~$\GS{D}$, let~$\TC(\GS{D};p)$ be the geometric realization of the spectrum denoted $\underline{TC}(\GS{D},S^0;p)$ in \cite{Dundas}*{Definition~6.4.1.1 on pages~253--254}.

In order to compare this definition to the one used in this paper, consider the lax monoidal functor
\(
\MOR{\reall{-}}{\Gamma\CS_*}{\Sigma\Sp}
\)
from \autoref{SPECTRA}.
It induces a change of base functor
\[
\MOR{\reall{-}}{\Gamma\CS_*\D\Cat}{\Sigma\Sp\D\Cat}
\,,
\]
which we denote with the same symbol; compare \autoref{ENRICHED}.
Notice that, if $\GS{A}$ is a monoid in~$\Gamma\CS_*$ (i.e., an $\mathbf{S}$-algebra in the sense of~\cite{Dundas}) and $\CC$ is an ordinary small category, then $\reall{\GS{A}[\CC]}\cong\reall{\GS{A}}[\CC]$.
The following result follows easily by inspecting the definitions.
\begin{lemma}
\label{simpl-vs-top}
For any $\Gamma\CS_*$-category~$\GS{D}$ there is a natural $\pi_*$-isomorphism in~$\IN\Sp$
\[
\TC(\GS{D};p)
\TO
\forget\TC\bigl(\reall{\GS{D}};p\bigr)
\,.
\]
\end{lemma}

Denote by $\Sym\Mon\Gamma\CS_*\D\Cat$ the category of small symmetric monoidal $\Gamma\CS_*$-categories.
The algebraic $K$-theory spectrum of a symmetric monoidal $\Gamma\CS_*$-category is defined in~\cite{Dundas}*{Subsection~5.2.1.4 on pages~192--193}.
By taking geometric realization, we get a functor
\begin{equation}
\label{eq:Dundas-K-functor}
\MOR{\BK}{\Sym\Mon\Gamma\CS_*\D\Cat}{\IN\Sp}
\,.
\end{equation}

The cyclotomic trace map is then defined in~\cite{Dundas}*{Definition~7.1.1.2 and the rest of Section 7.1.1 on page~285} as a zig-zag of natural transformations
\begin{equation}
\label{eq:trc}
\trc_p\colon
\K(\GS{D})\FROM[\simeq]\dotsb\TO\dotsb\FROM[\simeq]\TC(\GS{D};p)
\end{equation}
of functors ${\Sym\Mon\Gamma\CS_*\D\Cat}\TO{\IN\Sp}$, where the wrong-way maps are all natural $\pi_*$-isomorphisms.

We are now ready to describe $\ell$ and~$\tau$.
Fix models in the category of monoids in~$\Gamma\CS_*$ for the Eilenberg-Mac Lane ring spectrum~$\spec{H}\IZ$, the sphere spectrum~$\spec{S}$, and the Hurewicz map $\spec{S}\TO\spec{H}\IZ$.
We then obtain a natural transformation~$\ell$ from the bottom to the top composition in
\[
\begin{tikzcd}
\Groupoids
\arrow[r, shift left,  "\SF_{\spec{H}\IZ}"]
\arrow[r, shift right, "\SF_{\spec{S}}"']
&
\Gamma\CS_*\D\Cat^\copr
\arrow[r, hook]
&
\Sym\Mon\Gamma\CS_*\D\Cat
\arrow[r, "\K"]
&
\IN\Sp
\,,
\end{tikzcd}
\]
where $\K$ is the algebraic $K$-theory functor from~\eqref{eq:Dundas-K-functor}; compare also \autoref{examples-of-T}.

Next, using~\eqref{eq:trc} and taking the product over all primes~$p$, we obtain a zig-zag of natural transformations
\[
\tau=(\trc_p)\colon
\K(\GS{D})\FROM[\simeq]\dotsb\TO\dotsb\FROM[\simeq]\smallprod_p\TC(\GS{D};p)
\]
of functors ${\Sym\Mon\Gamma\CS_*\D\Cat}\TO{\IN\Sp}$, where the wrong-way maps are all natural $\pi_*$-isomorphisms.

In order to describe~$\sigma$, we first switch to enrichments in~$\Sigma\Sp$, and then we need to use Morita invariance, \autoref{Morita} below, because the zig-zag~\eqref{eq:alpha} between $\TC$ and~$\C$ is not defined for categories of modules.

In more detail, \autoref{simpl-vs-top} gives a natural $\pi_*$-isomorphism from the top to the bottom composition in the diagram
\[
\begin{tikzcd}[column sep=large, row sep=-1ex]
&
\Gamma\CS_*\D\Cat
\arrow[dr, near start, "{\raisebox{-.25ex}{$\ts\Pi$}}\TC(-;p)"]
\\
\Groupoids
\arrow[ur, near end, "\SF_\spec{S}"]
\arrow[dr, near end, "\SF_\spec{S}"']
&
&
\IN\Sp
\,.
\\
&
\Sigma\Sp\D\Cat
\arrow[ur, near start, "{\raisebox{-.25ex}{$\ts\Pi$}}\TC(-;p)"']
\end{tikzcd}
\]
\autoref{Morita}\ref{i:Morita-C-TC} gives a natural $\pi_*$-isomorphism from the bottom to the top composition in
\[
\begin{tikzcd}[column sep=large]
\Or G
\arrow[r, "\oid{G}{-}"]
&
\Groupoids
\arrow[r, shift left,  "\SF_\spec{S}"]
\arrow[r, shift right, "{\spec{S}[-]}"']
&
\Sigma\Sp\D\Cat
\arrow[r, "{\raisebox{-.25ex}{$\ts\Pi$}}\TC(-;p)"]
&
\IN\Sp
\,.
\end{tikzcd}
\]
Finally, using~\eqref{eq:alpha} and taking the product over all primes~$p$, and for $p=11$ using also the projection $\TC(-;11)\TO\THH(-)$, we obtain a zig-zag of natural transformations
\[
\sigma\colon
\smallprod_p\TC(\spec{S}[\CC];p)
\TO
\dotsb
\xleftarrow{\:\simeq\:}
\THH(\spec{S}[\CC])\times\smallprod_p\C(\spec{S}[\CC];p)
\]
of functors $\Groupoids\TO\IN\Sp$.

\begin{remark}
\label{tau-sigma-Mackey}
Because of the naturality of \autoref{constr-Mackey} in the functor~$\T$, the maps $\ell$ and~$\tau$ induce natural transformations of Mackey functors.
Arguing as above with Morita invariance and \autoref{constr-Mackey}, we see that also the homology theory associated with $\THH(\spec{S}[-])\times\smallprod_p\C(\spec{S}[-];p)$ has a Mackey structure.
However, we only know that the natural transformation induced by~$\sigma$ respects the induction structures, but not necessarily the Mackey structures.
\end{remark}

We finish this section with the well-known Morita invariance property needed above.
For a symmetric ring spectrum~$\spec{A}$, consider the symmetric spectral category~$\SF_\spec{A}(\CC)$ of finitely generated free $\spec{A}$-modules over~$\CC$ from \autoref{fgf-mod},
and the functor of symmetric spectral categories $\spec{A}[\CC]\TO\SF_\spec{A}(\CC)$ from \autoref{fgf-on-one-obj}.

\begin{lemma}[Morita invariance]
\label{Morita}
Let $\spec{A}$ be a symmetric ring spectrum.
\begin{enumerate}
\item
\label{i:Morita-THH}
The natural transformation
\[
\THH(\spec{A}[\oid{G}{-}])\TO\THH(\SF_\spec{A}(\oid{G}{-}))
\]
of functors $\Or G\TO\CW\CT^{S^1}$ is objectwise a $\pi_*^1$-isomorphism.
\item
\label{i:Morita-C-TC}
Assume that $\spec{A}$ is \cplus.
Then the natural transformations
\begin{align*}
 \C(\spec{A}[\oid{G}{-}];p)&\TO \C(\SF_\spec{A}(\oid{G}{-});p)
\,,
\\
\TC(\spec{A}[\oid{G}{-}];p)&\TO\TC(\SF_\spec{A}(\oid{G}{-});p)
\end{align*}
of functors $\Or G\TO\CW\CT$ are objectwise $\pi_*$-isomorphisms.
\end{enumerate}
\end{lemma}

\begin{proof}
\ref{i:Morita-THH}
Consider the following commutative diagram of functors of spectral categories.
\begin{equation}
\label{eq:Morita}
\begin{tikzcd}
\spec{A}[\oid{G}{\,(G/H)}]
\arrow[r]
&
\SF_\spec{A}(\oid{G}{\,(G/H)})
\\
\spec{A}[H]
\arrow[r]
\arrow[u]
&
\SF_\spec{A}(H)\cong\SF_{\spec{A}[H]}(\pt)
\arrow[u]
\end{tikzcd}
\end{equation}
The vertical maps induce $\pi_*^1$-isomorphisms on~$\THH$ by an argument as in~\cite{Dundas}*{Corollary~4.3.5.2 on page~171}.
For any symmetric ring spectrum~$\spec{B}$, the functor $\spec{B}\TO\SF_\spec{B}(\pt)$ induces a $\pi_*^1$-isomorphism on~$\THH$ by an argument similar to that of~\cite{Dundas}*{Lemma~4.3.5.11 on page~176}.
So we conclude that also the top horizontal functor in~\eqref{eq:Morita} induces a $\pi_*^1$-isomorphism on~$\THH$.

\ref{i:Morita-C-TC}
The assumption on~$\spec{A}$ implies that $\spec{A}[\CC]$ and $\SF_\spec{A}(\CC)$ are objectwise strictly connective, objectwise convergent, and very well pointed for any small category~$\CC$.
The result then follows from part~\ref{i:Morita-THH}, the fundamental fibration sequence from \autoref{THH-Adams}\ref{i:THH-Adams}, and the fact that homotopy limits respect $\pi_*$-isomorphisms.
\end{proof}


\section{Proof of the Detection Theorem}
\label{PROOF-DETECTION}

The goal of this section is to prove the Detection \autoref{detection}, thus completing the proof of \autoref{main-technical}.

Let $\CF\subseteq\FCyc$ be a family of finite cyclic subgroups of~$G$, and let $\TT\subseteq\IN$ be a subset.
Recall that Assumption~\techB{\CF}{\TT} reads as follows.
\begin{itemize}[label=\techB{\CF}{\TT}]
\item
For every~$C\in\CF$ and for every $t\in\TT$, the natural homomorphism
\[
K_t(\IZ[\zeta_c])
\TO
\prodp K_t\Bigl(\IZ_p\tensor_\IZ\IZ[\zeta_c];\IZ_p\Bigr)
\]
is $\IQ$-injective, where $c$ is the order of~$C$ and $\zeta_c$ is any primitive $c$-th root of unity.
\end{itemize}

\begin{lemma}
\label{Z[mu]=>Z[C]}
If Assumption~\techB{\CF}{\TT} holds then, for every~$C\in\CF$ and for every $t\in\TT$, the natural homomorphism
\[
K_t(\IZ[C])
\TO
\prodp K_t(\IZ_p[C];\IZ_p)
\]
is $\IQ$-injective.
\end{lemma}

\begin{proof}
Let~$\CM C$ be the maximal $\IZ$-order in~$\IQ[C]$ containing~$\IZ[C]$.
Since $\IQ[C]\cong\smallprod_{d|c}\IQ(\zeta_d)$, we have that $\CM C \cong \smallprod_{d|c} \IZ[\zeta_d]$; compare~\cite{Reiner}*{Theorem~10.5 on page~128} or \cite{Oliver}*{Theorem~1.5(i) on page~25}.
Using that algebraic $K$-theory respects finite products, we get the following commutative diagram.
\begin{equation}
\label{eq:maximal-order}
\begin{tikzcd}
\ds K_t(\IZ[C])
\arrow[r]
\arrow[d]
&
\ds\smash{\smallprod_p}\,K_t(\IZ_p[C];\IZ_p)
\arrow[d]
\\
\ds K_t(\CM C)
\cong
\bigoplus_{d|c}K_t(\IZ[\zeta_d])
\arrow[r]
&
\ds\adjustlimits\smallprod_p\bigoplus_{d|c} K_t\smash{\Bigl(}\IZ_p\tensor_\IZ\IZ[\zeta_d];\IZ_p\smash{\Bigr)}
\end{tikzcd}
\end{equation}
The finite sums here commute with the product and with rationalization.
Hence, since the family~$\CF$ is closed under passage to subgroups, Assumption~\techB{\CF}{\TT} implies that the bottom horizontal map is $\IQ$-injective for every~$t\in\TT$.
By~\cite{Reiner}*{Theorem~41.1 on page~379}, we have that $c\CM C\subseteq\IZ[C]$, and so we get the following cartesian square of rings.
\[
\begin{tikzcd}
\IZ[C] \ar[r] \ar[d] & \IZ[C]/ c\CM C \ar[d] \\
\CM C  \ar[r]        & \CM C / c\CM C
\end{tikzcd}
\]
Both rings on the right are finite, and for any finite ring~$R$ and any $t\ge1$ we have $K_t(R)\tensor_\IZ\IQ=0$ by~\cite{Weibel-K-book}*{Proposition~IV.1.16 on page~296}.
Then the Mayer-Vietoris long exact sequence in~$K_*(-)\tensor_\IZ\IZ[\frac{1}{c}]$ from~\cite{Charney}*{Corollary~2 on page~6 and examples on page~7} implies that the left-hand vertical map in~\eqref{eq:maximal-order} is a $\IQ$-isomorphism for every~$t\ge1$.
For $t=0$, the natural map $K_0(\IZ)\TO K_0(\IZ[C])$ is a $\IQ$-isomorphism by~\cite{Swan}*{Proposition~9.1 on page~573}.
Using this, we see that $K_0(\IZ[C])\TO K_0(\CM C)$ is $\IQ$-injective.
So we conclude that the top horizontal map in~\eqref{eq:maximal-order} is $\IQ$-injective for every~$t\in\TT$.
\end{proof}

Combining this lemma with a result of Hesselholt and Madsen~\cite{HM-top}, which depends on a theorem of McCarthy~\cite{McCarthy}, we obtain the following corollary.

\begin{corollary}
\label{K->TC}
If Assumption~\techB{\CF}{\TT} holds then, for every~$C\in\CF$ and for every $t\in\TT$, the map
\[
\MOR{\tau=(\trc_p)}{\BK^{\ge0}(\IS[C])}{\prodp\TC(\IS[C];p)}
\]
induced by the cyclotomic trace maps is $\pi^\IQ_t$-injective.
\end{corollary}

\begin{proof}
Consider the following commutative diagram.
\[
\begin{tikzcd}
\K^{\ge0}(\IS[C])
\arrow[r, "\ell"]
\arrow[d, "(\trc_p)"']
&
\K^{\ge0}(\IZ[C])
\arrow[r, "c"]
&
\ds\smash{\smallprod_p}\,\K^{\ge0}(\IZ_p[C])\pcompl
\arrow[d, "\prod(\trc_p)\pcompl"]
\\
\ds\smallprod_p\TC(\IS[C];p)
\arrow[rr]
&
&
\ds\smallprod_p\TC(\IZ_p[C];p)\pcompl
\end{tikzcd}
\]
The linearization map~$\ell$ is a $\pi^\IQ_*$ isomorphism by~\cite{Waldhausen-A1}*{Proposition~2.2 on page~43}.
\autoref{Z[mu]=>Z[C]} implies that $\pi_t(c)\tensor_\IZ\IQ$ is injective for every $t\in\TT$.
Since the $\IZ_p$-algebra $\IZ_p[C]$ is finitely generated as a $\IZ_p$-module, the $p$-completed cyclotomic trace map $\MOR{(\trc_p)\pcompl}{\K^{\ge0}(\IZ_p[C])\pcompl}{\TC(\IZ_p[C];p)\pcompl}$ is a $\pi_t$-isomorphism for every~$t\ge0$ and every prime~$p$ by~\cite{HM-top}*{Theorem~D on page~30}.
The result follows.
\end{proof}

\begin{proof}[Proof of \autoref{detection}]
We want to apply \autoref{rel-Chern} to the composition of the natural transformations
\[
\K^{\ge0}(\IS[-])
\TO[\tau]
\smallprod_p\TC(\IS[-];p)
\TO[\sigma]
\THH(\IS[-])\times\smallprod_p\C(\IS[-];p)
\]
of functors $\Groupoids\TO\IN\Sp$.
Here we are suppressing some wrong-way natural $\pi_*$-isomorphisms, as explained in \autoref{TRC}.
Recall from \autoref{tau-sigma-Mackey} that both $\tau$ and~$\sigma$ induce natural transformations of equivariant homology theories that respect the induction structures, but we only know that $\tau$ also respects the Mackey structures.

By \autoref{rel-Chern}\ref{i:tau-ind}, both $\tau$ and~$\sigma$ are compatible with the isomorphisms~$\mathit{c}$ in~\eqref{eq:Chern}.
Therefore it is enough to consider one $t\in\TT$ at a time.

Assume first that $t>0$.
Then \autoref{K->TC} and an argument similar to \autoref{rel-Chern}\ref{i:tau-Chern-inj} imply that the restriction of $\tau_\%\tensor_\IZ\IQ$ to the summand indexed by~$t$ is injective.
Moreover, for all finite groups~$C$, $\pi_t(\THH(\IS[C]))\tensor_\IZ\IQ=0$ by \autoref{THH-free-loop} and $\TC(\IS[C];p)\TO\C(\IS[C];p)$ is a $\pi^\IQ_t$-isomorphism for each~$p$ by \autoref{TC->TxC}.
Therefore $\sigma$ induces a $\IQ$-isomorphism between the corresponding functors $\Sub G(\Fin)\TO\Ab$, and an isomorphism between the summands indexed by~$t$.
We conclude then that the restriction $(\sigma_\%\circ\tau_\%)\tensor_\IZ\IQ$ to the summands indexed by~$t$ is injective.

Now let~$t=0$.
Recall that the linearization map $\MOR{\ell}{\K^{\ge0}(\IS[G])}{\K^{\ge0}(\IZ[G])}$ is a $\pi^\IQ_*$-isomorphism for every group~$G$, and therefore it induces an isomorphism $S_t(C;\K^{\ge0}(\IS[-]))\cong S_t(C;\K^{\ge0}(\IZ[-]))$.
By~\cite{Swan}*{Proposition~9.1 on page~573}, the natural map $K_0(\IZ)\tensor_\IZ\IQ\TO K_0(\IZ[C])\tensor_\IZ\IQ$ is an isomorphism for any finite group~$C$.
Therefore, $S_0(C;\K^{\ge0}(\IS[-]))=0$ whenever $C\ne1$.
This implies that the map $EG\TO EG(\CF)$ induces on the summand indexed by~$t=0$ an isomorphism
\[
H_n(BG;\IQ)\tensor_\IQ\Bigl(K_0(\IS)\tensor_\IZ\IQ\Bigr)
\TO[\cong]
\bigoplus_{(C)\in(\FCyc)}
H_n(BZ_GC;\IQ)\tensor_{\IQ[W_G C]}S_0(C;\K^{\ge0}(\IS[-]))
\,.
\]
So it is enough to consider the case~$\CF=1$; compare \autoref{EG(F)-computation}.
Since $H_n(BG;\IQ)\tensor_\IQ-$ clearly preserves injectivity, it is enough to show that $\pi_0(\sigma\circ\tau)$ in the following commutative diagram is $\IQ$-injective.
\[
\begin{tikzcd}[column sep=large, row sep=scriptsize]
K_0(\IS)
\arrow[r, "\pi_0(\sigma\circ\tau)"]
\arrow[d, "\pi_0(\ell)"']
&
\ds\smash{\pi_0\Bigl(\THH(\IS)\times\smallprod_p\C(\IS;p)\Bigr)}
\arrow[r, "\pi_0(\pr_1)"]
&
\pi_0\bigl(\THH(\IS)\bigr)
\arrow[d]
\\
K_0(\IZ)
\arrow[rr, "B"]
\arrow[drr, "D"']
&
&
\ds\pi_0\bigl(\THH(\IZ)\bigr)
\arrow[d]
\\
&
&
HH_0(\IZ)
\end{tikzcd}
\]
Here $HH$ is Hochschild homology, $D$~is the Dennis trace, $B$ is the B\"okstedt trace, and $\pr_1$ is the projection on the first factor.
Since $\ell$ and~$D$ are isomorphisms in dimension~$0$, it follows that $\pi_0(\sigma\circ\tau)$ is injective, thus finishing the proof.
\end{proof}


\section{Whitehead groups}
\label{WHITEHEAD}

This section is devoted to the proof of~\autoref{add-Whitehead} and its special case, \autoref{Whitehead}, about the rationalized assembly map for Whitehead groups.
First we need to introduce some notation.

The finite subgroup category~$\Sub G(\Fin)$ that appears in~\eqref{eq:whitehead} has as objects the finite subgroups $H$ of~$G$; the morphisms are defined as follows.
Given finite subgroups $H$ and~$K$ of~$G$, let $\conhom_G(H,K)$ be the set of all group homomorphisms $H\TO K$ given by conjugation by an element of~$G$.
The group $\inn(K)$ of inner automorphisms of~$K$ acts on~$\conhom_G(H,K)$ by post-composition.
The set of morphisms in~$\Sub G(\Fin)$ from $H$ to~$K$ is then defined as the quotient~$\conhom_G(H,K)/\inn(K)$.
For example, in the special case when $G$ is abelian, then $\Sub G(\Fin)$ is just the poset of finite subgroups of~$G$ ordered by inclusion.

Given a space~$X$, not pointed and not equivariant, we denote by~$\Pi(X)$ its fundamental groupoid, and consider the functor $\Top\TO\IN\Sp$ given by $X\longmapsto\K^{\ge0}(\IZ[\Pi(X)])$.
There is the classical assembly map
\begin{equation}
\label{eq:classical-assembly-Pi(X)}
X_+\sma\K^{\ge0}(\IZ)\TO\K^{\ge0}(\IZ[\Pi(X)])
\,;
\end{equation}
see~\cite{WW-assembly}.
(We suppress here an irrelevant wrong-way natural $\pi_*$-isomorphism.)
We denote by~$\Wh_\IZ^{\ge0}(X)$ the homotopy cofiber of~\eqref{eq:classical-assembly-Pi(X)}.

Now, given a group~$G$, we consider the composition of the functors above with
\[
\Or G
\xrightarrow{\oid{G}{-}}
\Groupoids
\xrightarrow{\ B\ }
\Top
\,,
\]
where $B=\real{N_\bullet(-)}$ is the classifying space functor.
We obtain the following diagram of functors $\Or G\TO\IN\Sp$
\begin{equation}
\label{eq:classical-assembly-Wh}
B(\oid{G}{-})_+ \sma \K^{\ge0}(\IZ)
\TO
\K^{\ge0}(\IZ[\Pi(B(\oid{G}{-}))])
\TO
\Wh_\IZ^{\ge0}(B(\oid{G}{-}))
\,,
\end{equation}
which, by definition, is objectwise a homotopy cofibration sequence.
For any subgroup $H$ of~$G$ we have
\begin{equation}
\label{eq:pi*(Wh)}
\pi_n
\Wh_\IZ^{\ge0}(B(\oid{G}{\,(G/H)}))
\cong\begin{cases}
\wh(H)&\text{if $n=1$;}
\\
\widetilde{K}_0(\IZ[H])&\text{if $n=0$;}
\\
0&\text{if $n<0$.}
\end{cases}
\end{equation}
From this we get the following lemma.

\begin{lemma}
\label{Wh-Bredon}
For every group~$G$ there is an isomorphism
\[
H_1^G\bigl(EG(\Fin);\Wh_\IZ^{\ge0}(B(-))\bigr)
\tensor_\IZ\IQ
\cong
\colim_{H\in\obj\Sub G(\Fin)}\wh(H)
\tensor_\IZ\IQ
\,.
\]
\end{lemma}

\begin{proof}
From~\cite{Davis-Lueck}*{Theorem~4.7 on pages~234--235} there is a spectral sequence converging to $H_{i+j}^G(EG(\Fin);\Wh_\IZ^{\ge0}(B(-)))$, with $E_2$-term  given by the Bredon homology groups
\[
H^G_i\bigl(EG(\Fin);\pi_j\Wh_\IZ^{\ge0}(B(-))\bigr)
\,.
\]
Since
\(
\pi_0\Wh_\IZ^{\ge0}(B(\oid{G}{\,(G/H)}))\tensor_\IZ\IQ
\cong
\widetilde{K}_0(\IZ[H])\tensor_\IZ\IQ
\)
vanishes if $H$ is finite~\cite{Swan}*{Proposition~9.1 on page~573}, the edge homomorphism of the spectral sequence gives an isomorphism
\[
H_1^G\bigl(EG(\Fin);\Wh_\IZ^{\ge0}(B(-))\bigr)
\tensor_\IZ\IQ
\cong
H_0^G\bigl(EG(\Fin);\pi_1\Wh_\IZ^{\ge0}(B(-))\bigr)
\tensor_\IZ\IQ
\,.
\]
Using the definition of Bredon homology, the fact that the cellular chain complex of $EG(\Fin)$ is a free resolution of the constant functor on~$\Or G(\Fin)$ with value~$\IZ$, and right exactness of $-\tensor_{\Or G(\Fin)} M$, we get an isomorphism
\[
H_0^G\bigl(EG(\Fin);\pi_1\Wh_\IZ^{\ge0}(B(-))\bigr)
\cong
\colim_{G/H\in\obj\Or G(\Fin)}\wh(H)
\,.
\]
Finally, using~\cite{LRV}*{Lemma~3.11 on page~152}, we see that
\[
\colim_{G/H\in\obj\Or  G(\Fin)}\wh(H)
\cong
\colim_{H  \in\obj\Sub G(\Fin)}\wh(H)
\,.
\qedhere
\]
\end{proof}

\begin{proof}[Proof of~\autoref{add-Whitehead}]
Consider the homotopy cofibration sequence~\eqref{eq:classical-assembly-Wh}.
The natural homeomorphism $EG\times_G S\cong B(\oid{G}{S})$ (e.g., see~\cite{LR-cyclic}*{Lemma~9.9(i) on page~628}), identifies the term on the left with
\[
\BE(-) = (EG \times_G -)_+ \sma\K^{\ge0}(\IZ)
\,.
\]
For any groupoid~$\CG$ there is an equivalence $\CG\TO\Pi(B\CG)$, which induces a $\pi_*$-isomorphism between the middle term in~\eqref{eq:classical-assembly-Wh} and $\K^{\ge0}(\IZ[\oid{G}{-}])$.

The homotopy cofibration sequence~\eqref{eq:classical-assembly-Wh} produces a long exact sequence of the associated $G$-equivariant homology theories from \autoref{equiv-homolgy-theory}.
The maps induced by the projection $EG(\Fin)\TO\pt$ give the following commutative diagram with exact columns.
\[
\begin{tikzcd}[sep=scriptsize]
H_n^G\bigl(EG(\Fin);\BE\bigr)
\arrow[d]
\arrow[rr]
&
\raisebox{1.25ex}{\one}
&
H_n^G(  \pt  ;\BE)
\arrow[d]
\\
H_n^G\bigl(EG(\Fin);\K^{\ge0}(\IZ[-])\bigr)
\arrow[d]
\arrow[rr]
&
\raisebox{1.25ex}{\two}
&
K_n(\IZ[G])
\arrow[d]
\\
H_n^G\bigl(EG(\Fin);\Wh_\IZ^{\ge0}(B(-))\bigr)
\arrow[d]
\arrow[rr]
&
\raisebox{1.25ex}{\three}
&
H_n^G\bigl(  \pt  ;\Wh_\IZ^{\ge0}(B(-))\bigr)
\arrow[d]
\\
H_{n-1}^G\bigl(EG(\Fin);\BE\bigr)
\arrow[rr]
&
\raisebox{1.25ex}{\four}
&
H_{n-1}^G(  \pt  ;\BE)
\end{tikzcd}
\]
By \cite{Davis-Lueck}*{Lemma~5.4 on page~239}, the assembly map for~$\BE$ is an isomorphism for every group~$G$ and every family of subgroups.
So the top and the bottom horizontal maps $\one$ and~$\four$ are isomorphisms.

In the notation of \autoref{main-technical}, we are assuming that condition~\techA{\FCyc}{2} holds, and \techB{\FCyc}{\{0,1\}} is satisfied by~\autoref{no-B-in-Wh}.
Therefore the map~$\two$ is $\IQ$-injective if~$n=1$.
An elementary diagram chase, the four lemma, implies that also the map~$\three$ is $\IQ$-injective if~$n=1$.

When $n=1$ the target of~$\three$ is isomorphic to~$\wh(G)$ by~\eqref{eq:pi*(Wh)}.
The source is computed rationally by \autoref{Wh-Bredon}.
\end{proof}


\section{Schneider Conjecture}
\label{SCHNEIDER}

We briefly review the relation between algebraic $K$-theory and \'etale cohomology of rings of integers in number fields, and give equivalent formulations of the Schneider \autoref{Schneider-Conj} and some useful reductions.
Other equivalent formulations are discussed in \cites{Soule-ast, Kolster, Geisser}.

Let $\CO_F$ denote the ring of integers in a finite field extension $F| \IQ$. Fix a prime~$p$.
For a  prime ideal $\mathfrak{p}$ in $\CO_F$ with $\mathfrak{p} | p$, let $\CO_{\mathfrak{p}}$ denote the completion, $F_{\mathfrak{p}}$ its field of fractions, and $k_{\mathfrak{p}}$ the residue field at~$\mathfrak{p}$.
Consider the following commutative square.
\[
\begin{tikzcd}[column sep=large]
\ds \CO_F
\arrow[r]
\arrow[d]
&
\ds \IZ_p \:\smash{\tensor_\IZ}\: \CO_F \cong \smash{\smallprod_{\mathfrak{p}|p}} \CO_\mathfrak{p} \arrow[d]
\\
\ds
\CO_F[\tfrac{1}{p}]
\arrow[r]
&
\ds \IZ_p \tensor_\IZ \CO_F[\tfrac{1}{p}] \cong
\smallprod_{\mathfrak{p}|p} F_\mathfrak{p}
\end{tikzcd}
\]
These ring homomorphisms,
together with the natural map from $K$-theory to \'etale $K$-theory \cite{Dwyer-Friedlander}*{Proposition~4.4 on page~256} for $\IZ[\tfrac{1}{p}]$-algebras, the spectral sequence from \'etale cohomology to \'etale $K$-theory \cite{Dwyer-Friedlander}*{Proposition~5.1 on page~260}, and the identification of \'etale cohomology with continuous group cohomology~\cite{Milne-ADT}*{Proposition~II.2.9 on page~170} induce the following commutative diagram.
The top horizontal map is the map in the Schneider \autoref{Schneider-Conj}.
\begin{equation}
\label{eq:K-to-etale}
\begin{tikzcd}[column sep=large]
\ds K_{2i-1}\bigl(\CO_F;\IZ_p\bigr)
\arrow[r]
\arrow[d, "\ts\one\,"']
&
\ds\smash{\smallprod_{\mathfrak{p}|p}}K_{2i-1}\bigl(\CO_\mathfrak{p};\IZ_p\bigr)
\arrow[d, "\,\ts\bone"]
\\
\ds K_{2i-1}\bigl(\CO_F[\tfrac{1}{p}];\IZ_p\bigr)
\arrow[r]
\arrow[d, "\ts\two\,"']
&
\ds\smash{\smallprod_{\mathfrak{p}|p}}K_{2i-1}\bigl(F_\mathfrak{p};\IZ_p\bigr)
\arrow[d, "\,\ts\btwo"]
\\
\ds K_{2i-1}^\et\bigl(\CO_F[\tfrac{1}{p}];\IZ_p\bigr)
\arrow[r]
\arrow[d, "\ts\three\,"']
&
\ds\smash{\smallprod_{\mathfrak{p}|p}}K_{2i-1}^\et\bigl(F_\mathfrak{p};\IZ_p\bigr)
\arrow[d, "\,\ts\bthree"]
\\
\ds H^1_\et\bigl(\CO_F[\tfrac{1}{p}];\IZ_p(i)\bigr)
\arrow[r]
\arrow[d, "\ts\four\,"']
&
\ds\smash{\smallprod_{\mathfrak{p}|p}}H^1_\et\bigl(F_\mathfrak{p};\IZ_p(i)\bigr)
\arrow[d, "\,\ts\bfour"]
\\
\ds H^1_\cont\bigl(G_\Sigma;\IZ_p(i)\bigr)
\arrow[r]
&
\ds\smallprod_{\mathfrak{p}|p}H^1_\cont\bigl(G_{\mathfrak{p}};\IZ_p(i)\bigr)
\end{tikzcd}
\end{equation}
Here $G_{\Sigma}$ is the Galois group of the extension $F_\Sigma|F$, where $F_\Sigma$ is the union of all finite extensions $F'$ of $F$ for which $\CO_F[\tfrac{1}{p}] \to \CO_{F'}[\tfrac{1}{p}]$ is unramified \citelist{\cite{Schneider}*{page~182} \cite{Milne-ADT}*{Section~I.4 on page~48}},
and $G_{\mathfrak{p}}$ denotes the absolute Galois group of the local field $F_{\mathfrak{p}}$.
We write $r_1(F)$ for the number of real embeddings  and $r_2(F)$ for the number of pairs of complex embeddings of~$F$.
The degree of $F|\IQ$ is $r_1(F)+2r_2(F)$.

\begin{theorem}
\label{K-to-etale}
Let $F$ be a number field, and if $p=2$ assume that $F$ is totally imaginary.
Let~$i\ge2$.
In diagram~\eqref{eq:K-to-etale}, all vertical maps are isomorphisms, and therefore if one of the horizontal maps is $\IQ$-injective, then all the horizontal maps are $\IQ$-injective.
All groups in~\eqref{eq:K-to-etale} are finitely generated $\IZ_p$-modules, and their $\IZ_p$-ranks are given in the following table.
\begin{center}
\begin{tabular}{c|c|c}
& left column & right column\\
\hline
$i$ odd  & $r_1(F) + r_2(F)$ &  $r_1(F) + 2 r_2(F)$\\
$i$ even & $r_2(F)$          &  $r_1(F) + 2 r_2(F)$\\
\end{tabular}
\end{center}
\end{theorem}

\begin{proof}
The fact that the maps $\one$ and $\bone$ are isomorphisms follows from the localization sequence in algebraic $K$-theory \cite{Quillen-K1}*{Section~5, Corollary on page~113} and Quillen's computation~\cite{Quillen-finite}*{Theorem~8 on page~583} of the algebraic $K$-theory groups of finite fields.
Notice that $p$-completion of spectra preserves stable homotopy fibration sequences; see \autoref{P-COMPL}.

The global and local Lichtenbaum-Quillen conjectures, asserting that the maps $\two$ and $\btwo$ are isomorphisms, follow from celebrated results of Voevodsky and others.
For fields of characteristic zero, the Milnor and Bloch-Kato conjectures were confirmed in~\cite{Voevodsky-mod-2}*{Corollary~7.5 on page~97} and~\cite{Voevodsky-mod-l}*{Theorem~6.1 on page~429}, and shown to imply the Beilinson-Lichtenbaum conjecture in~\cite{Suslin-Voevodsky}*{Theorem~7.4 on page~169}.
Using Grayson's~\cites{Grayson, Suslin} or Levine's~\cites{Levine-coniveau, Voevodsky-possible} constructions of the motivic spectral sequence, it follows that algebraic and \'etale $K$-theory, both with $\IZ_p$-coefficients, agree in positive degrees for global and local number fields.
Using localization techniques, as in~\cite{RW00}*{Section~6} and~\cite{Levine-loc}, it follows that the Dwyer-Friedlander comparison map is also an isomorphism for the rings of $p$-integers in global number fields.
For $p=2$, the comparison is made for totally imaginary number fields in~\cite{RW99}*{Theorem 0.1 on page 101}, and for real number fields in~\cite{PAO}*{Theorem~4 on page~201}.

Independently, $\btwo$ was shown to be an isomorphism for $p$ odd by Hesselholt and Madsen~\cite{HM-annals}*{Theorem~6.1.10 on page~103}, using topological cyclic homology and cyclotomic trace methods.

The isomorphisms $\three$ and $\bthree$ follow at once from the spectral sequence of~\cite{Dwyer-Friedlander}*{Proposition~5.1 on page~260},
since $\CO_F[\frac{1}{p}]$ and $F_\mathfrak{p}$ have \'etale cohomological dimension~$2$.
For $p=2$ this uses our hypothesis that $F$ is totally imaginary.

For the isomorphisms $\four$ and~$\bfour$, see \cite{Milne-ADT}*{Proposition~II.2.9 on page~170} and~\cite{Milne-EC}*{III.1.7 on page~86}.

The rank of $K_{2i-1}(\CO_F)$ was calculated by Borel~\cite{Borel}*{Proposition~12.2 on page~271}, and agrees with the $\IZ_p$-rank of $K_{2i-1}(\CO_F;\IZ_p)$ by Quillen's finite generation theorem~\cite{Quillen-fg}*{Theorem~1 on page~179}.
A similar calculation of the $\IZ_p$-rank of $K_{2i-1}(F_\mathfrak{p};\IZ_p)$ was outlined in~\cite{Wagoner}; see also~\cite{Panin}.
The corresponding results for Galois cohomology were proven by Soul\'e in~\cite{Soule-LNM}*{Fact~b) on page~376 and Remark~1 on page~390}, as applications of Iwasawa theory.
For the local result, see also~\cite{Schneider}*{\textsection3 Satz~4 on page~188}.
\end{proof}

\begin{theorem}
\label{Schneider}
Assume that $p$ is odd, and let~$i\ge2$.
The following statements are equivalent.
\begin{enumerate}
\item \label{i:schneider2}
The horizontal maps in diagram~\eqref{eq:K-to-etale} are $\IQ$-injective.
\item \label{i:schneider3}
The group $H^2 \bigl(G_\Sigma;\IQ_p/\IZ_p(1-i)\bigr)$ is finite.
\item  \label{i:schneider4}
The group $H^2 \bigl(G_\Sigma;\IQ_p/\IZ_p(1-i)\bigr)$ is zero.
\item \label{i:schneider5}
The group $H^2_\cont \bigl(G_\Sigma;\IZ_p(1-i)\bigr)$ is finite.
\item \label{i:schneider6}
The kernel of the map
\(
H^1 \bigl(G_\Sigma;\IQ_p/\IZ_p(i)\bigr)
\TO
\smallprod_{\mathfrak{p}|p}H^1 \bigl(G_{\mathfrak{p}};\IQ_p/\IZ_p(i)\bigr)
\)
is finite.
\item
\label{i:Iwasawa}
Let $F(\mu_{p^\infty}\!)$ denote the union of the fields $F(\zeta_{p^n})$.
Let $L$ be the maximal abelian pro-$p$-extension of $F(\mu_{p^\infty})$ for which all primes are completely split.
Let $G=\mathit{Gal}(F(\mu_{p^\infty}\!)|F)$ and let $Z$ be the Pontryagin dual of $\mathit{Gal}(L|F(\mu_{p^\infty}\!))$.
Then $H^0(G;Z(i))$ is finite.
\end{enumerate}
Furthermore, for a fixed number field~$F$ and a fixed odd prime~$p$, the equivalent conditions above hold for almost all~$i\ge2$.
\end{theorem}

That $H^2(G_{\Sigma};\IQ_p/\IZ_p(1-i))$ is finite was conjectured by Schneider in~\cite{Schneider}.
Neukirch, Schmidt, and Wingberg~\cite{NSW}*{page~641} also refer to this assertion as a conjecture of Schneider.
The notation $H^{2,1-i}_{\Sigma} = H^2( G_{\Sigma} ; \IQ_p / \IZ_p (1-i))$ for arbitrary $i \in \IZ$ is introduced in~\cite{Schneider}*{page~189}.
For $i=0$ this group is computed in~\cite{Schneider}*{\textsection4 Lemma~2ii) on page~190}.
For $i \neq 0$ the groups are of the form
$H^{2,1-i}_{\Sigma} \cong (\IQ_p / \IZ_p)^{i_{1-i}(F)}$
for some finite numbers $i_{1-i}(F)$ by~\cite{Schneider}*{\textsection4 Lemma~2i) on page~190}.
On page~192 it is conjectured that the numbers $i_{1-i}(F) = 0$ for all $i \neq 0$.

\begin{proof}[Proof of \autoref{Schneider}]
This is due to Schneider~\cite{Schneider}.

We first explain the equivalence of \ref{i:schneider2} and~\ref{i:schneider3}.
The lower horizontal map in \eqref{eq:K-to-etale} sits in the Poitou-Tate 9-term exact sequence, where the group to the left is the Pontryagin dual of the group in \ref{i:schneider3}, and the group further left is finite since $H^0_\cont( G_{\mathfrak{p}} ; \IZ_p (i) )$ is finite cyclic; compare \cite{Weibel-survey}*{Lemma~13 on page~148}.
Indeed, let $\Sigma$ denote the finite set of primes over $p$ together with all infinite primes.
Then the Poitou-Tate exact sequence as stated for example in \cite{NSW}*{(8.6.10)~on page~489}, with coefficients in~$\IZ/p^\nu\IZ(i)$ and ramification allowed over~$\Sigma$, consists of finite groups by \cite{Schneider}*{\textsection2 Satz~1i) and Satz~2i) on page~185}.
Passing to the limit over the reduction maps $\IZ/ p^{\nu +1} \IZ (i) \TO \IZ/ p^{\nu} \IZ (i)$ yields the desired exact sequence.

Passing instead to the colimit over the inclusions $\IZ/ p^{\nu} \IZ (i) \TO \IZ/ p^{\nu + 1} \IZ (i)$, we obtain similarly the equivalence of \ref{i:schneider5} and \ref{i:schneider6}; compare~\cite{Schneider}*{\textsection5 proof of Lemma~3 on page~196}.

Statement \ref{i:schneider3} is equivalent to \ref{i:schneider4} since the group is divisible by \cite{Schneider}*{\textsection4 Lemma~2i) on page~190}.

The equivalence of \ref{i:schneider4} and \ref{i:schneider6} is \cite{Schneider}*{\textsection5 Corollary~4 on page~197}, where the notation $R_i(F)$ for the kernel of the map in \ref{i:schneider6} is introduced on page~196.

Then $R_i(F)$ is identified with $H^0(G;Z(i))$ in~\cite{Schneider}*{\textsection6 Lemma~1 on page~199}, proving the equivalence of \ref{i:schneider6} and
\ref{i:Iwasawa}.

Finally, the last claim is~\cite{Schneider}*{Satz~3 on page~200}.
\end{proof}

\begin{lemma}
\label{Schneider-reductions}
Fix a prime~$p$ and an integer~$c\ge3$.
Let $t=2i-1$ with $i\ge2$.
\begin{enumerate}
\item
\label{i:extensions}
Suppose $E|F$ is a finite extension.
If the Schneider \autoref{Schneider-Conj} holds for $E$ and~$t$, then it holds for $F$ and~$t$.
\item
\label{i:max-real}
Let $\IQ(\zeta_c)^+$ denote the maximal real subfield of~$\IQ(\zeta_c)$.
Assume that $i$ is odd.
If the Schneider \autoref{Schneider-Conj} holds for $\IQ(\zeta_c)^+$ and~$t$, then it holds for $\IQ(\zeta_c)$ and~$t$.
\end{enumerate}
\end{lemma}

\begin{proof}
\ref{i:extensions}
The inclusion of~$F$ into~$E$ induces the following commutative diagram.
\[
\begin{tikzcd}[row sep=tiny]
\ds K_{2i-1}\bigl(\CO_F;\IZ_p\bigr)\tensor_\IZ\IQ
\arrow[r]
\arrow[d, shorten <=-1.1ex, shorten >=-.5ex]
&
\ds K_{2i-1}\bigl(\IZ_p\tensor_\IZ\CO_F;\IZ_p\bigr)\tensor_\IZ\IQ
\arrow[d, shorten <=-1.1ex, shorten >=-.5ex]
\\
\ds K_{2i-1}\bigl(\CO_E;\IZ_p\bigr)\tensor_\IZ\IQ
\arrow[r]
&
\ds K_{2i-1}\bigl(\IZ_p\tensor_\IZ\CO_E;\IZ_p\bigr)\tensor_\IZ\IQ
\end{tikzcd}
\]
The vertical maps are injective, since composed with restriction both maps are isomorphisms.

\ref{i:max-real}
Consider the diagram above for $F=\IQ(\zeta_c)^+$ and $E=\IQ(\zeta_c)$.
If $i$ is odd then the vertical map on the left is an isomorphism.
This follows from the statement about ranks in \autoref{K-to-etale} and since
$r_1(E)+r_2(E)=\varphi(c)/2=r_1(F)+r_2(F)$.
\end{proof}

\begin{remark}
There are partial results about the equivalent statements in \autoref{Schneider}, but unfortunately they cannot be used to verify Assumption~\ref{our-B}.
For example, Soul\'e verified in~\cite{Soule-thesis}*{Th{\'e}or{\`e}me~5 on page~268} that the condition~\ref{i:schneider4} of \autoref{Schneider} holds for all~$i<0$.
The results in~\cite{Schneider}*{\textsection6 Satz~6 on page~201} and~\cite{Geisser}*{Lemma~4.3a) and~b)} deal with totally real number fields and $i$ even or negative.
\end{remark}

We close the paper by showing how the last statement in~\autoref{Schneider} leads to \autoref{above-L}.

\begin{proof}[Proof of \autoref{above-L}]
If $EG(\Fin)$ is finite, then by \autoref{EGF-finite-type} there are only finitely many conjugacy classes of finite cyclic subgroup of~$G$.
So the last statement in~\autoref{Schneider} and \autoref{Schneider=>B} imply that there exists an~$M\in\IN$ such that Assumption~\techB{\FCyc}{[M,\infty)} holds.
By \autoref{EGFin-finite-type=>A}, Assumption~\ref{our-A} holds, and so also \techA{\FCyc}{N} holds for every~$N\in\IN$.
Moreover, by \autoref{EG(F)-computation}, for every finite cyclic subgroup~$C$ of~$G$ we have $H_s(BZ_GC;\IQ)\cong H_s(EG(\Fin)^C/Z_GC;\IQ)$, and these groups vanish if~$s>D=\dim EG(\Fin)$.
Now apply \autoref{main-technical} to obtain \autoref{above-L} for $L=M+D$.
\end{proof}



\vfill

\end{document}